\documentclass[11pt, a4paper]{amsart}

\usepackage{amssymb, amsmath}
\usepackage[inline]{enumitem}
\usepackage{hyphenat}
\usepackage[colorlinks,linkcolor=blue,citecolor=blue,urlcolor=black]{hyperref}
\usepackage{xcolor}
\usepackage{mathtools}
\usepackage{tikz-cd}
\usepackage[labelformat=empty]{caption}
\usepackage{bbm}

\usepackage[left=3cm, right=3cm, top = 2.8cm, bottom = 2.8cm, marginparwidth=2.4cm ]{geometry}

\makeatletter
\def\@tocline#1#2#3#4#5#6#7{\relax
  \ifnum #1>\c@tocdepth 
  \else
    \par \addpenalty\@secpenalty\addvspace{#2}%
    \begingroup \hyphenpenalty\@M
    \@ifempty{#4}{%
      \@tempdima\csname r@tocindent\number#1\endcsname\relax
    }{%
      \@tempdima#4\relax
    }%
    \parindent\z@ \leftskip#3\relax \advance\leftskip\@tempdima\relax
    \rightskip\@pnumwidth plus4em \parfillskip-\@pnumwidth
    #5\leavevmode\hskip-\@tempdima
      \ifcase #1
      \or\or \hskip 2em \or \hskip 2em \else \hskip 3em \fi%
      #6\nobreak\relax
    \dotfill\hbox to\@pnumwidth{\@tocpagenum{#7}}\par
    \nobreak
    \endgroup
  \fi}
\makeatother


\usepackage{mathrsfs}

\usepackage{aliascnt}
\usepackage{mathrsfs}
\usepackage{mathrsfs}
\usepackage{tikz}
\usepackage{cleveref}

\usepackage[color       = blue!20,
                bordercolor = black,
                textsize    = tiny,
                ]{todonotes}

\numberwithin{equation}{section}
\newtheorem{theorem}[subsection]{Theorem}
\newtheorem{corollary}[subsection]{Corollary}
\newtheorem{lemma}[subsection]{Lemma}
\newtheorem{proposition}[subsection]{Proposition}

\newtheorem{prop}[subsection]{Proposition}
\theoremstyle{definition}
\newtheorem{definition}[subsection]{Definition}
\newtheorem{df}[subsection]{Definition}
\newtheorem{remark}[subsection]{Remark}
\newtheorem{example}[subsection]{Example}

\newtheorem{rmk}[subsection]{Remark}
\newtheorem{warn}[subsection]{Warning}
\newcommand{\bA}{\mathbb{A}}

\newcommand{\bE}{\mathbb{E}}
\newcommand{\bF}{\mathbb{F}}

\newcommand{\bL}{\mathbb{L}}
\newcommand{\bN}{\mathbb{N}}

\newcommand{\bT}{\mathbb{T}}
\newcommand{\bP}{\mathbb{P}}
\newcommand{\bQ}{\mathbb{Q}}
\newcommand{\bR}{\mathbb{R}}
\newcommand{\bZ}{\mathbb{Z}}
\newcommand{\cA}{\mathcal{A}}

\newcommand{\cC}{\mathcal{C}}
\newcommand{\cD}{\mathcal{D}}
\newcommand{\cE}{\mathcal{E}}
\newcommand{\cF}{\mathcal{F}}

\newcommand{\cM}{\mathcal{M}}

\newcommand{\cO}{\mathcal{O}}

\newcommand{\cV}{\mathcal{V}}

\newcommand{\cal}{\mathcal}

\newcommand{\sS}{\mathscr{S}}

\newcommand{\sX}{\mathscr{X}}

\newcommand{\Z}{\mathbb{Z}}

\DeclareMathOperator{\Map}{Map}

\DeclareMathOperator{\THH}{THH}

\DeclareMathOperator{\HH}{HH}
\DeclareMathOperator{\logHH}{logHH}
\newcommand{\blogHH}{\mathbf{logHH}}

\DeclareMathOperator{\Mod}{Mod}

\newcommand{\colimit}{\mathop{\mathrm{colim}}}

\newcommand{\ol}{\overline}
\newcommand{\ul}{\underline}

\newcommand{\gp}{\mathrm{gp}}
\newcommand{\rep}{\mathrm{rep}}
\newcommand{\GL}{\mathrm{GL}}

\newcommand{\id}{{\mathrm{id}}}

\providecommand{\Spec}[1]{\operatorname{Spec}(#1)}

\newcommand{\SmlSm}{\mathrm{SmlSm}}
\newcommand{\iAr}{\mathrm{iAr}}
\newcommand{\lSm}{\mathrm{lSm}}
\newcommand{\Sm}{\mathrm{Sm}}
\newcommand{\lSmAr}{\mathrm{lSmAr}}
\newcommand{\dlSmAr}{\mathrm{dlSmAr}}

\newcommand{\lSch}{\mathrm{lSch}}
\newcommand{\Sch}{\mathrm{Sch}}
\newcommand{\lAff}{\mathrm{lAff}}

\newcommand{\et}{{\acute{e}t}}
\newcommand{\SCRing}{\mathcal{S}\mathrm{CRing}}
\newcommand{\sCRing}{\mathrm{sCRing}}
\newcommand{\infPsh}{\mathcal{P}\mathrm{sh}}
\newcommand{\infShv}{\mathcal{S}\mathrm{hv}}
\newcommand{\infSpc}{\mathcal{S}\mathrm{pc}}
\newcommand{\infDAb}{\mathcal{D}(\mathrm{Ab})}
\newcommand{\Fil}{\mathrm{Fil}}
\newcommand{\cofib}{\mathrm{cofib}}
\newcommand{\fiber}{\mathrm{fib}}
\newcommand{\logDAeff}{\mathrm{log}\mathcal{DA}^{\mathrm{eff}}}
\newcommand{\logDA}{\mathrm{log}\mathcal{DA}}
\newcommand{\Fun}{\mathrm{Fun}}
\newcommand{\sPsh}{\mathrm{sPsh}}
\newcommand{\sShv}{\mathrm{sShv}}

\newcommand{\boxx}{\overline{\Box}}

\newcommand{\Hyp}{\mathrm{Hyp}}

\newcommand{\Deri}{\mathrm{D}}
\newcommand{\Coh}{\mathrm{Coh}}

\newcommand{\Shv}{\mathrm{Shv}}
\newcommand{\Psh}{\mathrm{Psh}}
\newcommand{\lFan}{\mathrm{lFan}}

\newcommand{\dlog}{d\hspace{.07em}\mathrm{log}}

\newcommand{\CMon}{\mathrm{CMon}}
\newcommand{\CRing}{\mathrm{CRing}}
\newcommand{\sCMon}{\mathrm{sCMon}}

\newcommand{\sPreLog}{\mathrm{sPreLog}}

\newcommand{\PreLog}{\mathrm{PreLog}}
\newcommand{\SemiLog}{\mathrm{SemiLog}}

\newcommand{\Bl}{\mathrm{Bl}}
\newcommand{\Bigwedge}{\bigwedge\nolimits}
\newcommand{\unit}{\mathbbm{1}}
\newcommand{\pt}{\mathrm{pt}}
\newcommand{\Hom}{\mathrm{Hom}}
\newcommand{\map}{\mathrm{map}}

\title[A logarithmic HKR theorem and residue 
sequences]{A Hochschild-Kostant-Rosenberg theorem and residue sequences for logarithmic Hochschild homology}

\author{Federico Binda}
\address{Department of Mathematics ``F. Enriques'', University of Milan, Italy}
\email[F. Binda]{federico.binda@unimi.it}

\author{Tommy Lundemo}
\address{Department of Mathematics and Informatics, University of Wuppertal, Germany}
\email{lundemo@math.uni-wuppertal.de}

\author{Doosung Park}
\address{Department of Mathematics and Informatics, University of Wuppertal, Germany}
\email{dpark@uni-wuppertal.de}

\author{Paul Arne {\O}stv{\ae}r}
\address{Department of Mathematics ``F. Enriques'', University of Milan, Italy \&
Department of Mathematics, University of Oslo, Norway}
\email{paul.oestvaer@unimi.it \& paularne@math.uio.no}

\thanks{The authors gratefully acknowledge the hospitality and support of the 
Centre for Advanced Study at the Norwegian Academy of Science and Letters in Oslo, Norway, 
which funded and hosted the research project “Motivic Geometry" during the 2020/21 academic year, 
and the RCN Frontier Research Group Project no. 250399 “Motivic Hopf Equations" and no. 312472 “Equations in Motivic Homotopy."
Binda would like to thank the Isaac Newton Institute for Mathematical Sciences for support and hospitality during the program 
``$K$-theory, algebraic cycles and motivic homotopy theory'' in 2022 when part of the work on this paper was undertaken, 
partially supported by the EPSRC Grant EP/R014604/1 (UK) and by the 
PRIN “Geometric, Algebraic and Analytic Methods in Arithmetic” at MUR (Italy).  
Lundemo was partially supported by the NWO-grant 613.009.121.
Park was partially supported by the research training group GRK 2240 ``Algebro-Geometric Methods in Algebra, 
Arithmetic and Topology.''
{\O}stv{\ae}r was partially supported by a Guest Professorship awarded by The Radboud Excellence Initiative}

\date{\today}

\begin{document}
\maketitle

\begin{abstract} 
This paper incorporates the theory of Hochschild homology into our program on log motives. 
We discuss a geometric definition of logarithmic Hochschild homology of animated pre-log rings 
and construct an Andr{\'e}-Quillen type spectral sequence. 
The latter degenerates for derived log smooth maps between discrete pre-log rings. 
We employ this to show a logarithmic version of the Hochschild-Kostant-Rosenberg theorem and that 
logarithmic Hochschild homology is representable in the category of log motives. 
Among the applications, 
we deduce a generalized residue sequence involving blow-ups of log schemes.
\end{abstract}

\tableofcontents

\section{Introduction} 
Let $R$ be a commutative ring. 
If $A$ is a commutative $R$-algebra, 
the Hochschild homology $\HH(A / R)$ of $A$ relative to $R$ admits a filtration whose graded pieces are the derived exterior powers 
$(\Bigwedge_A^i {\mathbb L}_{A/R})[i]$ of the cotangent complex ${\mathbb L}_{A/R}$ \cite[Proposition IV.4.1]{NS18}. 
If $A$ is smooth over $R$, 
this gives an isomorphism of graded commutative rings
\begin{equation}
\label{equation:HKRthm}
\Omega^*_{A / R} \xrightarrow{\cong} \HH_*(A / R).
\end{equation}
The isomorphism in \eqref{equation:HKRthm} is 
commonly referred to as the Hochschild--Kostant--Rosenberg (=HKR) theorem after \cite{HKR62}. 
The \emph{HKR-filtration} alluded to above is a motivic filtration in the sense of 
Bhatt--Morrow--Scholze \cite{BMS19} 
(see also Mathew \cite{Mat22} and Raksit \cite{Rak20}), and it is a key input in many classical $K$-theoretic computations involving trace methods.

One of the main differences between algebraic $K$-theory and  Hochschild homology is that the latter (as well as its topological counterpart) does not enjoy a satisfactory localization property. For example, if $A$ is a discrete valuation ring with finite residue field $k$ and fraction field $F$, a classical result of Quillen \cite{Quillen} implies the existence of a cofiber sequence
\begin{equation}\label{dvrcofiberk} K(A) \to K(F) \to \Sigma K(k). \end{equation}
Such a sequence does not exist if we replace $K$-theory with $\HH$ or $\THH$, and the trace map from $K(F)$ to $\HH(F)$ or to $\THH(F)$ carries much less information than the one for $A$ or $k$.  
One of the initial motivations for studying Hochschild homology of log rings, 
as introduced by Rognes \cite{Rog09}, 
is that it allows to fix part of this problem, since it participates in cofiber sequences
that are unavailable for ordinary Hochschild homology.  For example,
there is a 
cofiber sequence 
\begin{equation}
\label{dvrcofiber}
\HH(A) \to {\rm logHH}(A, A - \{0\}) \to \HH(k)[1]
\end{equation} 
relating 
the Hochschild homology of $A$, 
its log Hochschild homology (for the usual log structure) and the Hochschild homology of $k$, 
see e.g., \cite[Example 5.7]{RSS15}. 
\vspace{0.1in}

With the goal of better understanding the term $\logHH$ and studying its relationship with the 
motivic theory developed in \cite{logSH}, \cite{logDMcras}, \cite{logDM}, in this paper 
we are interested in generalizing the HKR-theorem and several of its surrounding statements 
to the context of log geometry. 

Our first main result is a generalization of the HKR-theorem to pre-log rings. We state the result in the context of animated pre-log rings: These are modeled by simplicial pre-log rings in Sagave--Sch\"urg--Vezzosi's work on derived log geometry \cite{SSV16}.  
\begin{theorem}[Theorem \ref{thm:logaqss}]
\label{thm:mainthm} 
For a morphism of animated pre-log rings $(R, P) \to (A, M)$, 
the log Hochschild homology $\logHH((A, M) / (R, P))$ admits a descending separated filtration with 
graded pieces the derived exterior powers $(\Bigwedge_A^i {\mathbb L}_{(A, M) / (R, P)})[i]$ 
of Gabber's cotangent complex ${\mathbb L}_{(A, M) / (R, P)}$. 
If $(A, M)$ is discrete and \emph{derived log smooth} over a discrete $(R, P)$-algebra, 
there is an isomorphism of graded commutative rings 
\[
\Omega^*_{(A, M) / (R, P)} 
\xrightarrow{\cong} 
\logHH_*((A, M) / (R, P))
\] 
relating the log de Rham complex and log Hochschild homology.  
\end{theorem}

Derived log smooth maps $(R, P) \to (A, M)$ of discrete pre-log rings are log smooth in the sense of classical log geometry \cite[\S 3]{Kat89}, see \cite[Theorem 6.4]{SSV16}.
We will proceed to explain the meaning and necessity of the additional adjective \emph{derived} in the statement of the logarithmic HKR-theorem, after which we state generalizations and applications of Theorem \ref{thm:mainthm}.

\begin{remark} In the context of discrete pre-log rings, the log Hochschild homology in the statement of Theorem \ref{thm:mainthm} is equivalent to Rognes' definition \cite{Rog09}, see \Cref{prop:rognescomparison}, 
and also used by Krause--Nikolaus \cite{KN19}. See also Leip \cite[Page 886]{Ob18} for a discussion of this definition of log ${\rm THH}$. 
The above should not be confused with Hesselholt--Madsen's \cite{HM03} ``relative term" $\THH(A | K)$ defined in terms of Waldhausen categories. Indeed, the sequence \eqref{dvrcofiberk} comes to life from d\'evissage, a result which is not available for topological Hochschild homology. This is very eloquently explained in the introduction of \cite{BL23}. Rognes' logarithmic Hochschild homology and Hesselholt--Madsen's relative term circumvent this problem in two different ways: While the latter is set up as to allow for a d\'evissage argument, the cofiber sequence \eqref{dvrcofiber} arises from a direct analysis of the morphism ${\rm HH}(A) \to {\rm logHH}(A, A - \{0\})$. This construction makes no reference to arguments in the spirit of d\'evissage, and consequently it is not at all clear how the two constructions compare. This is closely related to \cite[Conjecture 7.5]{Rog14}.

Olsson \cite{Ols} has proposed another definition involving his stack of log structures. 
As we explain in Section \ref{subsec:futurework}, we hope the results herein yield a comparison result relating 
Olsson's definition of logarithmic Hochschild homology with the one in this paper. 
\end{remark}

\subsection{The logarithmic Andr\'e--Quillen spectral sequence} 
One way to prove the HKR-theorem for ordinary Hochschild homology is to exhibit $\HH_*(A / R)$ as the abutment of an \emph{Andr\'e--Quillen spectral sequence}. Indeed, one readily shows  that the HKR-theorem holds for $A = R[x]$, from which it follows that it holds for any free commutative $R$-algebra. Using this, results of Quillen \cite[Theorem 8.1]{Qui70} yield a strongly convergent spectral sequence 
\begin{equation}
\label{equation:AQss}
E^2_{p, q} 
= 
\pi_p(\Bigwedge_A^q {\mathbb L}_{A / R}) 
\implies 
\pi_{p + q}(A \otimes_{A \otimes_R^{\mathbb L} A}^{\mathbb L} A)
\end{equation}
relating the derived wedge powers of ${\mathbb L}_{A / R}$ and the derived self-intersections of the (derived) 
``diagonal embedding" $A \otimes_R^{\mathbb L} A \to A$. 
If $R \to A$ is smooth, 
then ${\mathbb L}_{A / R} \simeq \Omega^1_{A / R}[0]$. 
Since the derived tensor product in \eqref{equation:AQss} computes $\HH_*(A / R)$, 
this finishes the proof. 
\vspace{0.1in}

To apply this proof strategy in the context of animated log rings, 
we define log Hochschild homology ${\rm logHH}((A, M) / (R, P))$ to be the derived self-intersections of a 
derived version of Kato--Saito's log diagonal embedding \cite[\S 4]{KS04}. 
We refer to Definition \ref{def:loghh} for details of this construction. 
The fact that our geometric description agrees with the direct generalization to animated log rings of 
Rognes' notion of log $\HH$ for discrete log rings \cite[Definition 3.21]{Rog09} is the content of the 
following result.

\begin{proposition}[\Cref{prop:rognescomparison2}]\label{prop:rognescomparison} 
There is a natural equivalence between Rognes' log Hochschild homology and 
${\rm logHH}((A, M) / (R, P))$. 
\end{proposition}
Our definition of log $\HH$ is inspired by the observation that in the underived context, 
the conormal bundle of Kato--Saito's log diagonal coincides with the module of log K\"ahler differentials 
$\Omega^1_{(A, M) / (R, P)}$ \cite[Corollary 4.2.8(ii)]{KS04}. 
Passing to the derived context begets Gabber's log cotangent complex ${\mathbb L}_{(A, M) / (R, P)}$ essentially 
by definition, 
see Proposition \ref{prop:logcotangentreplete}.
Combining this observation with \Cref{prop:rognescomparison}, we can apply
Quillen's construction once again  to obtain a strongly convergent spectral sequence \begin{equation}\label{intro:logquillenss}E^2_{p, q} = \pi_p(\Bigwedge_A^q {\mathbb L}_{(A, M) / (R, P)}) \implies \pi_{p + q} \logHH((A, M) / (R, P))\end{equation} relating the derived wedge powers of Gabber's log cotangent complex with log Hochschild homology. 
This is the point of the proof where the adjective ``derived" in the statement of Theorem \ref{thm:mainthm} 
comes into play: In general, 
Gabber's log cotangent complex is \emph{not} the log differentials concentrated in degree zero for log smooth maps, see Example \ref{exm:logetale}. Thus the log Andr\'e--Quillen spectral sequence generally fails to collapse for log smooth maps. For this reason, it is necessary to impose additional flatness assumptions on the log smooth map $(R, P) \to (A, M)$: It suffices to require that the map ${\mathbb Z}[P] \to {\mathbb Z}[M]$ is flat, 
a condition that holds for any integral log smooth map. 

\subsection{\'Etale base change for log Hochschild homology} Using the Andr\'e-Quillen spectral sequence,
Weibel--Geller \cite{WG91} proved the canonical equivalence 
$$B \otimes_A \HH(A/R) \xrightarrow{\simeq} \HH(B/R)$$ for 
\'etale maps $A \to B$ of commutative $R$-algebras. 
Their result follows from the transitivity sequence for the cotangent complex combined with the vanishing 
of the relative cotangent complex for \'etale morphisms. 
As Gabber's log cotangent complex enjoys a similar transitivity property,
the spectral sequence \eqref{intro:logquillenss} yields:

\begin{theorem}[Theorem \ref{prop:loghhbasechange}]
\label{thm:intrologetalebasechange} 
Let $(R, P)$ be an animated pre-log ring. A map $(A, M) \to (B, N)$ of $(R, P)$-algebras is derived log \'etale if and only if the canonical map 
\[
B \otimes_A \logHH((A, M) / (R, P)) 
\to 
\logHH((B, N) / (R, P))
\]
is an equivalence.
\end{theorem}

Theorem \ref{thm:intrologetalebasechange} applies for example to tamely ramified extensions of discrete valuation rings (for which the classical result doesn't hold). 
The adjective ``derived" is necessary here for the same reason as in Theorem \ref{thm:mainthm}: 
There exist 
log \'etale maps for which Gabber's log cotangent complex does not vanish, see again Example \ref{exm:logetale}.

\begin{remark}
In this work we use Gabber's cotangent complex for log schemes.
Olsson has proposed another definition \cite{Ols05}. 
Each has its own advantages and disadvantages, 
and there is no choice of log cotangent complex satisfying all desirable properties \cite[\S 7]{Ols05}. 
The main reasons for our choice are the following: 
1) Gabber's cotangent complex enjoys the same kind of transitivity property that the classical cotangent complex has, 
and 2) Gabber's cotangent complex admits a straightforward adaptation to animated log rings. 
\end{remark}

\subsection{Extension to fs log schemes}
Hochschild homology of rings extends to schemes owing to the \'etale base change theorem \cite{WG91}. 
Likewise, 
log Hochschild homology of log rings extends to log schemes by appeal to Theorem \ref{thm:intrologetalebasechange}. 

Several Grothendieck topologies are available in the category of log schemes. We can broadly divide them into two classes: the \emph{strict} topologies, where the coverings are defined on the underlying schemes, and the log structures are simply pulled back; and the \emph{log} topologies, where a non-trivial modification of the log structure happens along the covering. Among these, the \emph{dividing Nisnevich} topology, 
which we review in Section \ref{section:dNis}, 
plays a crucial role for the logarithmic motivic homotopy theory in \cite{logSH}, \cite{logDMcras}, \cite{logDM}. 

While log Hochschild homology satisfies strict \'etale descent (in fact, it satisfies \emph{integral} \'etale descent) thanks to 
\Cref{thm:intrologetalebasechange}, it does not satisfy dividing Nisnevich descent; 
see Example \ref{arrow.13}.
To remedy this, 
we need to consider the dividing Nisnevich localization of log Hochschild homology, which we now briefly discuss. 
In this process, 
we fix a base fs log scheme $S$ and define $\logHH(X / S)$ for every morphism of fs log schemes $X\to S$.
The precise construction carried out in \Cref{section:dNis,section:soac} involves arrow categories, and produces two sheaves of complexes, $\logHH(-/-)$ in $\Shv_{sNis}(\Fun(\Delta^1,\lAff))$ and $\logHH^{-}_{dNis}(-)$ in $\Shv_{dNis}(\Fun(\Delta^1,\lAff))$. The dividing Nisnevich sheafification process is well-controlled: 
in fact, if $X\to S$ is integral log smooth, 
Proposition \ref{dZarsheaf.17} shows that the localization map 
$\logHH(X/S)\rightarrow\logHH_{dNis}(X/S)$ is an equivalence. Along the way, we obtain similar results for the dividing Nisnevich localized log cotangent complex $\bL^{dNis}$.

\subsection{Representability of log Hochschild homology}

We refer to \cite{logDMcras} for an overview of log motives explaining our choice $\boxx=(\mathbb{P}^{1},\infty)$ 
of the unit interval. 
For an fs log scheme $S$, 
the $\infty$-category $\logDAeff(S)$ of effective motives without transfers is defined as the $\boxx$-localization 
of the $\infty$-category of dividing Nisnevich sheaves of chain complexes of abelian groups on $\lSm/S$, 
see also \cite{logSH}. 
Here $\lSm/S$ denotes the category of fs log schemes log smooth over $S$.
We claim that the dividing Nisnevich sheaf $\logHH_{dNis}(- / S)$ is $\boxx$-invariant and hence representable in $\logDAeff(S)$.
Owing to the logarithmic Andr\'e--Quillen spectral sequence \eqref{intro:logquillenss},
it suffices to show that there is a canonical equivalence 
\[
\bL^{dNis}(X\times \boxx\to S)\simeq \bL^{dNis}(X\to S)
\]
for all $X\in \lSm/S$.
Proposition \ref{dZarsheaf.13} yields a canonical equivalence
\[
\bL^{dNis}(X\to S)\simeq R\Gamma_{Nis}(\ul{X},\Omega_{X / S}^1),
\]
where $\ul{X}$ denotes the underlying scheme of $X$.
Finally, 
by an explicit calculation, 
we have
\[
R\Gamma_{Nis}(\ul{X},\Omega_{X / S}^1)
\simeq
R\Gamma_{Nis}(\ul{X\times \boxx},\Omega_{X\times \boxx / S}^1).
\]
This establishes a basic result relating logarithmic Hochschild homology and log motives.
We write $\lSch$ for the category of noetherian fs log schemes of finite Krull dimensions and finite type morphisms.

\begin{theorem}[Theorem \ref{dZarsheaf.16}]\label{thm:loghhrep}
Suppose that $S$ is a noetherian fs log scheme of finite Krull dimension.
Then we have 
\[
\logHH_{dNis}(- / S)\in \logDAeff(S).
\]
In particular, for every $X\in \lSm/S$, we have 
\[\Map_{\logDAeff(S)}(M_S(X), \logHH_{dNis}(- / S))= \logHH_{dNis}(X/S).\] 
\end{theorem}
Note that if $S$ has a valuative log structure (for example, $S$ is the spectrum of a DVR with canonical log structure, or $S$ has trivial log structure), then every morphism $f\colon X\to S$ is integral, and thus the dividing Nisnevich sheafification is redundant. See Section \ref{sec:motivic_repr} for more details.

\subsection{Residue sequences in log Hochschild homology}\label{subsec:resseqintro} Cofiber sequences for log Hochschild homology as in \eqref{dvrcofiber} have to this point only been available for log structures generated by a single element. Theorem \ref{thm:loghhrep} allows for the results of \cite{logDM} to be applied in the context of log Hochschild homology.
As we discuss in Section \ref{sec:gysinseq}, this perspective shows that the cofiber sequences in 
log $\HH$ should be thought of as generalizations of residue sequences in log geometry, as opposed to approximations of localization sequences in algebraic $K$-theory. For this reason, it seems the term ``\emph{residue sequence}" is more appropriate than 
``\emph{localization sequence}" in the context of cofiber sequences in log (topological) Hochschild homology of \cite{RSS15}. 

Moreover, the representability result Theorem \ref{thm:loghhrep} allows us to apply the results of \cite{logDM} to generalize the residue sequence in log Hochschild homology in the following way:

\begin{theorem}[Theorem \ref{Gysin.2}]\label{thm:regularseq} Fix a noetherian base $S$ of finite Krull dimension. If $X$ is a smooth $S$-scheme and $Z \subset X$ is a smooth closed subscheme, then there is a cofiber sequence \[
\HH(X/S)
\to
\logHH((\Bl_Z X,E) / S)
\to
\HH(Z/S)[1].
\] 
\end{theorem}
In fact, a more general result holds, when a non-trivial log structure on both $X$ and $Z$ is allowed. See Theorem \ref{Gysin.2} for the exact statement. 
As we explain in Remark \ref{rem:regularsequence}, one special case of Theorem \ref{thm:regularseq} is the cofiber sequence \[{\rm HH}(A/R) \to {\rm logHH}(({\rm Bl}_{(f_1, \dots, f_r)}{\rm Spec}(A), E)/R) \to {\rm HH}((A/(f_1, \dots, f_r)) / R)[1]\] for a regular sequence $(f_1, \dots, f_r)$ cutting a smooth subscheme in a smooth commutative $R$-algebra $A$, where $R$ is a commutative ring of finite Krull dimension. This recovers the discrete version of the Rognes--Sagave--Schlichtkrull sequence for $r = 1$. In other words, in order to have cofiber sequences in log Hochschild homology for log structures generated by several elements, we must leave the affine context behind and invoke blow-ups.

In Remark \ref{rem:bp} we discuss the existence of residue sequences in (log) topological Hochschild homology \begin{equation}
\label{bpcofiberseq}
{\rm THH}(BP \langle n \rangle) \to {\rm logTHH}(BP\langle n \rangle) \to {\rm THH}(BP \langle n - 1 \rangle)[1]
\end{equation} involving the truncated Brown--Peterson spectrum $BP \langle n \rangle$ for which there is provably no analogous localization sequence in algebraic $K$-theory.
In a slightly different setup,
a version of this sequence will appear in the forthcoming work of Ausoni--Bayındır--Moulinos.
Using an appropriate notion of blow-ups in spectral algebraic geometry, 
we suspect that Theorem \ref{thm:regularseq} admits a generalization to log topological Hochschild homology,
and that this will give rise to the generalization of the sequence \eqref{bpcofiberseq}. 
We 
refer to Remark \ref{rem:bp} for further discussion in this direction.

\subsection{Analogous results for log ring spectra} Some of the results in this paper are algebraic analogues of results proved in the context of log ring spectra in \cite{Lun21}, and these results (appearing in Sections $\S 2$ to $\S 5$) also appear in the second-named author's thesis \cite{Lun22}. At its time of writing, the second-named author was unaware of the algebro-geometric log diagonal interpretation of these results. 
Versions of \Cref{prop:rognescomparison} and \Cref{thm:intrologetalebasechange} hold for log 
\emph{topological} Hochschild homology of discrete pre-log rings in the sense of 
Krause--Nikolaus \cite{KN19} and Rognes--Sagave--Schlichtkrull \cite{RSS15}, \cite{RSS18}.

\subsection{Upcoming work}\label{subsec:futurework} The relationship between the circle action on log $\HH$ and the log de Rham differential is not treated in the present paper. We intend to prove a logarithmic analogue of the main result of \cite{TV11}, identifying log Hochschild homology with derived log de Rham theory in characteristic zero. 
Olsson \cite{Ols} has proposed another definition of log Hochschild homology using his approach to log geometry 
via algebraic stacks. 
We expect a comparison theorem relating these definitions of log Hochschild homology, 
analogous to \cite[\S 8.31]{Ols05} for log cotangent complexes.
We hope to return to these questions, 
in addition to those raised in Section \ref{subsec:resseqintro} and Remark \ref{rem:bp}.

\subsection{Notation and conventions}
\label{subsec:notation}
We refer to Ogus \cite{Ogu18} for standard conventions on monoids, 
(pre)-log rings and schemes. 
If $A$ is a discrete commutative ring and $M$ is a discrete commutative monoid, 
the triple $(A,M, \alpha)$ defines a pre-log ring if $\alpha\colon M\to A$ is a homomorphism of monoids, 
where multiplication gives the monoid operation on $A$.
If no confusion arises, we will drop $\alpha$ from the notation and simply write $(A,M)$. 
A morphism of pre-log rings $(f, f^\flat)\colon (A,M)\to (B,N)$ is a pair of morphisms $f\colon A\to B$, 
$f^\flat \colon M\to N$ such that the obvious square commutes. Since our constructions are all implicitly derived, we will omit the symbol $\mathbb{L}$ on top of the tensor product in the main body of the text.

For a category $\cC$, we employ the following notations.
\vspace{0.05in}

\begin{tabular}{l|l}
$\Sch$ & noetherian schemes of finite Krull dimensions
\\
$\lSch$ & noetherian fs log schemes of finite Krull dimensions
\\
$\Psh(\cC)$ & presheaves on $\cC$
\\
$\Shv_t(\cC)$ & $t$-sheaves on $\cC$ with respect to the topology $t$
\\
$\sPsh(\cC)$ & simplicial presheaves on $\cC$
\\
$\sShv_t(\cC)$ & simplicial $t$-sheaves on $\cC$ with respect to the topology $t$
\\
$\Coh(X)$ & coherent sheaves on $X$
\\
$\Deri^b(\cA)$ & bounded derived category of an abelian category $\cA$
\end{tabular}

\section{Affine derived logarithmic geometry} 
\label{section:dalg} 
The definition of Hochschild homology, even in the classical case, is inherently derived and better understood in the setting of derived algebraic geometry. In order to pursue this point of view in the logarithmic context, 
let us recall from  \cite[\S 2, 3]{SSV16} the notion of animated (therein modelled by simplicial) (pre-)log rings.

\subsection{Logarithmic morphisms}
We begin with the following basic definition.
\begin{definition}
We denote by $\sCMon$ the category of simplicial commutative monoids.
By elementary model category techniques (see e.g.\ \cite[Proposition 2.1]{SSV16}), this admits a proper simplicial combinatorial model structure in which a map is a fibration or a weak equivalence precisely when the underlying map of simplicial sets is so in the (standard) Kan--Quillen model structure. This is referred to as the \emph{standard} model structure on $\sCMon$. 
We write ${\rm Ani(CMon)}$ for the associated $\infty$-category of simplicial commutative monoids, 
which we shall refer to as the $\infty$-category of \emph{animated (commutative) monoids}. We remark that ${\rm Ani(CMon)}$ is presentable and in particular cocomplete \cite[Proposition A.3.7.6]{HA}.
We let $\oplus$ denote the coproduct in ${\rm Ani(CMon)}$.
\end{definition}

We now proceed to define and study derived analogues of logarithmic morphisms. See 
\cite[Chapter I.4.1]{Ogu18} for the context of discrete commutative monoids.

\begin{definition}
Let $M$ be an animated monoid.
Its \emph{space of units} is the grouplike animated monoid $\GL_1(M)$ of path components representing units in $\pi_0(M)$. 
\end{definition}

This definition allows for the following generalization of \cite[Definition I.4.1.1(3)]{Ogu18}:

\begin{definition}\label{def:logarithmic}
A map $f\colon P\to M$ of animated monoids is \emph{logarithmic} if the projection
\[
f^{-1}\GL_1(M):=P\times_M \GL_1(M) \to \GL_1(M)
\]
is an equivalence. If $f \colon P \to M$ is an arbitrary map of simplicial commutative monoids, we define its \emph{logification} $f^a \colon P^a \to M$ by the universal property of the cocartesian square \[\begin{tikzcd}[row sep = small]f^{-1}{\rm GL}_1(M) \ar{r} \ar{d} & {\rm GL}_1(M) \ar{d} \\ P \ar{r} &  P^a.\end{tikzcd}\]
\end{definition}

\begin{proposition}
Let $f\colon P\to M$ be a map of animated monoids. 
Then the logification $f^a \colon P^a \to M$ is the initial logarithmic map under $f$.
\end{proposition}

\begin{proof}
By construction, 
the composite map $(f^a)^{-1}{\rm GL}_1(M) \xrightarrow{} P^a \xrightarrow{f^a} M$ is equivalent to the inclusion 
${\rm GL}_1(M) \to M$, so that $f^a$ is logarithmic. 
The universal property of a pushout verifies the remaining claim. 
\end{proof}

\subsection{Animated pre-log rings}

Let ${\rm Ani(CRing)}$ denote the $\infty$-category of \emph{animated commutative rings}: We model this by the underlying $\infty$-category associated with the model structure on simplicial commutative rings $\sCRing$ with fibrations and weak equivalences defined on the level of simplicial sets.  

\begin{definition}
A \emph{pre-log simplicial ring} $(A, M, \alpha)$ is a simplicial object in the category of pre-log rings. That is, a simplicial commutative ring $A$ equipped with a simplicial commutative monoid $M$ and a map of simplicial monoids $\alpha\colon M\to (A,\cdot)$.
\end{definition}

By adjunction, 
$\alpha$ gives rise to a map $\bar{\alpha} \colon {\mathbb Z}[M] \to A$ of simplicial commutative rings. 
We let $\sPreLog$ denote the category of pre-log simplicial rings. 
By \cite[Proposition 3.3]{SSV16}, it admits a \emph{projective} model structure. In this model structure, $(f, f^\flat)$ is a fibration if both $f$ and $f^\flat$ are Kan fibrations of simplicial sets.

Let ${\rm Ani(PreLog)}$ denote the associated $\infty$-category of  
\emph{animated pre-log rings}; see Remark \ref{rem:loganimation} for a justification of the terminology.
If no confusion arises, 
we will write $(A,M)$ for a pre-log simplicial ring, 
omitting the structure map $\alpha$.
The coproduct of animated pre-log rings is defined using the underlying animated commutative rings and animated commutative monoids; that is, it is given by the formula $(A, M) \otimes (B, N):= (A \otimes B, M \oplus N)$ with the obvious structure map.

\begin{df}\label{def:logderived} Let $(R, P,\beta)$ be an animated pre-log ring. 

\begin{enumerate} \item $(R, P, \beta)$ is a \emph{log animated ring} if $\beta$ is logarithmic in the sense of Definition \ref{def:logarithmic}.
\item The \emph{logification} $(R, P^a, \beta^a)$ of $(R, P, \beta)$ is the logification of $\beta$ in the sense of Definition \ref{def:logarithmic}. 
\end{enumerate} 
\end{df}

\begin{remark} 
A log simplicial ring is not a
synonym with a simplicial object in the category of log rings, see \cite[Remark 4.18]{SSV16}. 
We use the term \emph{animated log ring} for an animated pre-log ring satisfying the condition of 
Definition \ref{def:logderived}(1). 
\end{remark}

There are Quillen pairs between model categories 
\begin{equation}
\label{equation:Quillenpairs}
\sCRing\rightleftarrows \sPreLog,
\;
\sCMon\rightleftarrows \sCRing,
\text{ and }
\sCMon\rightleftarrows \sPreLog.
\end{equation} 
The left adjoints in \eqref{equation:Quillenpairs} send $A \in \sCRing$ to $(A, \{1\})$, 
$M \in \sCMon$ to $\Z[M]$, 
and $M\in \sCMon$ to $(\Z[M],M)$.
On the underlying $\infty$-categories, 
the space of maps 
$${\rm Map}_{\rm Ani(PreLog)}((A, M), (B, N))$$ 
from $(A, M)$ to $(B, N)$ in ${\rm Ani(PreLog)}$ 
is given by the pullback of 
\begin{equation}
\label{mappingspace} 
{\rm Map}_{\rm Ani(CMon)}(M, N) \xrightarrow{} 
{\rm Map}_{\rm Ani(CRing)}({\mathbb Z}[M], B) 
\xleftarrow{} {\rm Map}_{\rm Ani(CRing)}(A, B).
\end{equation} 
The maps in \eqref{mappingspace} are naturally induced by the structure maps. 

\subsection{Free pre-log algebras}

Following Bhatt \cite[\S 6]{Bha12} and Olsson \cite[\S 8.2]{Ols05}, we discuss free pre-log algebras.
Let $(R_0, P_0)$ be a discrete pre-log ring.
The forgetful functor
\[
{\rm PreLog}_{(R_0, P_0)/} \to {\rm Set} \times {\rm Set}
\]
from pre-log $(R_0, P_0)$-algebras (sending $(A_0, M_0)$ to $(M_0, A_0)$) 
admits a left adjoint $F_{(R_0, P_0)}$ given by
\begin{equation}
\label{freeprelog}
F_{(R_0, P_0)}(X, Y) = (R_0 \langle X \sqcup Y \rangle, P_0 \oplus \langle X \rangle).
\end{equation}
The underlying commutative ring of $F_{(R_0, P_0)}(X, Y)$ is the polynomial ring $R_0\langle X \sqcup Y \rangle$ 
(we use the bracket notation to avoid confusion with the monoid ring construction). 
Moreover, 
its pre-log structure is induced by the pre-log structure on $R_0$ and the canonical map from the free 
commutative monoid on $X$ to $(R_0 \langle X \sqcup Y \rangle, \cdot)$.
Notice that a free pre-log ring, 
viewed as a constant simplicial pre-log ring, 
is cofibrant in the projective model structure. 

\begin{remark}[Animating pre-log rings]\label{rem:loganimation} We now explain that the $\infty$-category of animated pre-log rings indeed fits in the framework of animation, thus justifying the terminology. We largely follow the exposition of \u{C}esnavi\u{c}ius--Scholze \cite[Section 5]{CS}.

Let $(R_0, P_0)$ be a discrete pre-log ring. Define the category of \emph{polynomial pre-log rings} ${\rm Poly}_{(R_0, P_0)}$ to consist of the pre-log rings $F_{(R_0, P_0)}(X, Y)$ for finite sets $X$ and $Y$. Concretely, ${\rm Poly}_{(R_0, P_0)}$ has objects of the form \[(R_0[x_1, \dots, x_n, y_1, \dots, y_m], P_0 \oplus \langle x_1, \dots, x_n \rangle),\] where $\langle x_1, \dots, x_r \rangle \cong {\Bbb N}^r$ is the free commutative monoid generated by the $x_i$'s. 

Let ${\rm PreLog}_{(R_0, P_0)/}$ denote the category of (discrete) $(R_0, P_0)$-algebras and let ${\rm PreLog}^{\rm sfp}_{(R_0, P_0)/}$ denote the full subcategory consisting of those $(A_0, M_0)$ that are \emph{strongly of finite presentation}; that is, ${\rm Hom}_{{\rm PreLog}_{(R_0, P_0)/}}((A_0, M_0), (-, -))$ commutes with sifted colimits. Following the strategy of \cite[Example 5.1.3]{CS} for the free-forgetful adjunction \[{\rm Set} \times {\rm Set} \rightleftarrows {\rm PreLog}_{(R_0, P_0)/},\] we find that \cite[Corollary 4.7.3.18]{HA} implies that ${\rm PreLog}^{\rm sfp}_{(R_0, P_0)/}$ consists of retracts of objects in ${\rm Poly}_{(R_0, P_0)/}$. 

As summarized in \cite[Section 5.1.4]{CS}, the material of \cite[Sections 5.5.8 and 5.5.9]{HTT} implies that the $\infty$-category associated to that of pre-log simplicial rings may be modelled by that of finite product-preserving functors \[{\rm Poly}_{(R_0, P_0)}^{\rm op} \to {\rm Ani}\] from polynomial $(R_0, P_0)$-algebras to that of anima; that is, the animation of the category of $(R_0, P_0)$-algebras. 
\end{remark}

The following remark was suggested to us by an anonymous referee:

\begin{remark} One convenient way to model ${\rm Ani(PreLog)}$ is as the Grothendieck construction of the functor \[{\rm Ani(CRing)} \to {\rm Cat}_{\infty}, \quad R \mapsto {\rm Ani(CMon)}_{/(R, \cdot)}.\] Informally, the objects are pairs $(R, P)$ with a specified structure map $P \to (R, \cdot)$, and the morphisms are readily checked to coincide with the ones described above. In particular, \cite[Theorem 10.3]{GHN17} applies to conclude that ${\rm Ani(PreLog)}$ is presentable, while \cite[Proposition 2.4.4.3]{HTT} applies to obtain the description of the mapping spaces in ${\rm Ani(PreLog)}$ as the pullback of \eqref{mappingspace}.
\end{remark}

Using the $\infty$-category ${\rm Ani(PreLog)}$ as the basis of derived log geometry, 
we make the following definitions:

\begin{definition}\label{def:der_log_affine} Let $(A, M)$ be a pre-log simplicial ring. 

\begin{enumerate}\item
The $\infty$-category of \emph{affine pre-log derived schemes} 
is the 
opposite of the $\infty$-category of animated pre-log rings. 

\item Let $\Spec{A,M}$ denote the affine derived pre-log scheme corresponding to $(A, M)$.

\item For an affine derived pre-log scheme $X$, we let $\ul{X}$ denote the underlying affine derived scheme of $X$.
More precisely, if $X=\Spec{A,M}$, then $\ul{X}:=\Spec{A}$.
\end{enumerate}
\end{definition}

\subsection{Strict morphisms of pre-log simplicial rings}

Following \cite[\S 3.4]{SSV16}, we recall the basic functoriality properties of pre-log and log animated rings:

\begin{enumerate}
\item Let $(R, P, \beta)$ be an animated pre-log ring and let $f \colon R \to A$ be a map of animated commutative rings. 
The \emph{inverse image pre-log structure} $(A, f^*P)$ is the animated pre-log ring with a structure map 
$P \xrightarrow{\beta} (R, \cdot) \xrightarrow{(f, \cdot)} (A, \cdot)$. 
\item Let $(A, M, \alpha)$ be an animated pre-log ring and let $f \colon R \to A$ be a map of animated commutative rings. 
The \emph{direct image pre-log struct} is given by the fiber product of animated monoids
\[\begin{tikzcd}
f^*M \arrow[r]\arrow[d] & (R, \cdot) \arrow[d, "f"]\\
M \arrow[r, "\alpha"] & (A, \cdot).
\end{tikzcd}
\]
\end{enumerate}

\begin{definition} A map $(R, P) \xrightarrow{(f, f^\flat)} (A, M)$ of animated pre-log rings is \emph{strict} if the canonical map $(A, (f^*P)^a) \to (A, M^a)$ is an equivalence.  
\end{definition}

This notion of strictness differs slightly from the one in ordinary log geometry, 
where one typically requires $(R, P)$ and $(A, M)$ to be 
log rings already. 

\subsection{Repletions of animated commutative monoids}
Next we discuss \emph{repletions} of animated commutative monoids following Rognes \cite{Rog09} 
and Sagave--Sch\"urg--Vezzosi \cite{SSV16} in the derived context.  
An essentially identical notion appearing earlier in Kato--Saito \cite{KS04} traces back to 
exactifications of closed immersions of log schemes: 
See Remark \ref{rem:rogkatosaitocomp} for details.

Recall that an animated monoid is \emph{grouplike} if the commutative monoid $\pi_0(M)$ is a group. The (de)looping adjunction $(B, \Omega)$ in the category of animated commutative monoids is a Quillen pair for the standard model structure and induces an adjunction on the underlying $\infty$-categories. 
The unit of the adjunction \[
\Omega B(-)\colon {\rm Ani(CMon)}\to {\rm Ani(CMon)}, \quad M \mapsto \Omega B M,
\] gives rise to the group completion construction $M\to \Omega B M$;
if $M$ is grouplike, this is an equivalence, cf.\ \cite[5.2.6.11, 5.2.6.12, 5.2.6.15]{HA}.

\begin{definition} The \emph{group completion} of an animated monoid $M$ is the grouplike animated monoid $M^\gp := \Omega B M$. 
\end{definition}

Recall the following definition from \cite[Definition 3.6]{Rog09}, \cite[\S 2.2]{SSV16}.

\begin{definition} Let $f \colon N \to M$ be a map of animated commutative rings.
Then $f$ is 
\begin{enumerate}
\item \emph{virtually surjective} if 
$\pi_0(f^{\gp}) \colon \pi_0(N^{\gp}) \to \pi_0(M^{\gp})$ is a surjection
\item \emph{exact} if the induced square of animated monoids
\[\begin{tikzcd}N \ar{r} \ar[d,"f"'] & N^{\gp} \ar{d}{f^{\gp}} \\ M \ar{r} & M^{\gp}\end{tikzcd}\] is cartesian.
\item \emph{replete} if it is virtually surjective and exact. 
\end{enumerate}
The \emph{exactification} relative to the map $f$ is given by the pullback
\[
N^{\mathrm{ex}}
:=
M\times_{M^\gp} N^\gp.
\]
If $f$ is virtually surjective, 
we use the terminology \emph{repletion} instead of \emph{exactification} and set $N^{\rep}:=N^{\mathrm{ex}}$ 
(see \cite{Rog09}).
The map $f$ admits a canonical factorization $N \xrightarrow{} N^\rep \xrightarrow{f^\rep} M$.
\end{definition}

\begin{rmk}\label{rem:rogkatosaitocomp}
An analogous notion in the context of discrete integral commutative monoids 
is called \emph{exactification} by Ogus \cite[I.4.2.17]{Ogu18}. 
More generally, 
a morphism of discrete commutative monoids $P_0 \to Q_0$ is exact \cite[I.2.1.15]{Ogu18} 
if $P_0 = Q_0 \times_{Q_0^\gp} P_0^\gp$, 
where $M_0^\gp$ is the classical group completion of the discrete commutative monoid $M_0$. Under the additional assumption of virtual surjectivity, repletion has been considered (with different terminology) by Kato--Saito \cite[Proposition 4.2.1]{KS04}.
\end{rmk}

\begin{rmk}
We briefly discuss various perspectives on the repletion of the addition map ${\mathbb N}^2 \to {\mathbb N}$.
\begin{enumerate}
\item As a map of discrete commutative monoids, 
${\mathbb N}^2 \to {\mathbb  N}$ is a Kan fibration. Hence its repletion is given by the ordinary pullback of 
${\mathbb N} \xrightarrow{} {\mathbb Z} \xleftarrow{+} {\mathbb Z}^2$, 
which we identify with ${\mathbb N} \oplus {\mathbb Z}$. 
The isomorphism can be chosen so that the corresponding map from ${\mathbb N}^2$ to its repletion sends $(m, n)$ to $(m + n, n)$.
\item To give a geometric description of the repletion, 
recall that a  closed immersion of (discrete) log schemes is exact if and only if it is strict. 
As a non-example, 
the diagonal map 
$\Delta\colon\Spec{\mathbb{Z}[\mathbb{\mathbb{N}}]}
=\mathbb{A}_{\mathbb{N}}\to
\mathbb{A}_{\mathbb{N}^2}
=\Spec{\mathbb{Z}[\mathbb{\mathbb{N}}^2]}$
for the addition $\mathbb{N}^2 \to \mathbb{N}$ 
is clearly not strict (the monoids have different $\mathbb{N}$-ranks). 
Hence it is not exact. 
On the other hand, 
the repletion of 
$\Delta$ is constructed geometrically via the log blow-up 
$\Bl_{(0,0)}(\mathbb{A}_{\mathbb{N}^2}) \to \mathbb{A}_{\mathbb{N}^2}$ equipped with its canonical log structure. 
The strict transform of the diagonal $\Delta$ is a strict closed immersion; hence it is exact.

\item More generally, 
any quasi-compact morphism of quasi-compact and fine log schemes can be made exact by a suitable log blow-up by
\cite[Theorems III.2.6.7, II.1.8.1]{Ogu18}.
\end{enumerate}
\end{rmk}

In most of our examples, the map $f \colon N \to M$ admits a section and is therefore virtually surjective. 
In this case, the repletion admits a convenient description. 
The following is an analogue of \cite[Lemma 3.11]{Rog09} and \cite[Proposition I.4.2.19]{Ogu18} in the context of 
discrete commutative monoids, 
and \cite[Lemma 2.12]{Lun21} in the context of graded ${\mathbb E}_{\infty}$-spaces that form the basis for 
spectral log geometry.

\begin{proposition}\label{prop:repletesplit} Let $N \xrightarrow{f} M$ be a map of animated commutative monoids, and assume that it has a section $\eta\colon M\to N$.
Then there is an equivalence \[N^\rep \xrightarrow{\simeq} M \oplus N^\gp/M^\gp\] over and under $M$, where $N^\gp/M^\gp$ denotes the cofiber of $M^\gp \xrightarrow{\eta^\gp } N^\gp$. 
\end{proposition}

\begin{proof} 
Consider the map of cofiber sequences of grouplike animated monoids
\[
\begin{tikzcd}[row sep = small, column sep=small]M^\gp \ar{r} \ar{d}{=} & 
N^\gp \ar{r} \ar[dashed]{d} & N^\gp/M^\gp \ar{d}{=} \\ 
M^\gp \ar{r} & M^\gp \oplus (N^\gp/M^\gp) \ar{r} & N^\gp/M^\gp
\end{tikzcd}
\] 
The dashed arrow is induced by $f^\gp$ and the canonical map $N^\gp \to N^\gp/M^\gp$ (recall that finite products and finite coproducts of animated monoids coincide), and is an equivalence since the outer vertical maps are. 
Consider now the commutative cube of animated monoids 
\[
\begin{tikzcd}[row sep = small, column sep = small]
&
M \oplus (N^\gp/M^\gp)
\ar{rr}{}
\ar[]{dd}[near end]{}
& & M^\gp \oplus (N^\gp/M^\gp)
\ar{dd}
\\
N^\rep
\ar[crossing over]{rr}[near start]{}
\ar{dd}[swap]{}
\ar{ur}
& & N^\gp
\ar{ur}{\simeq}
\\
&
M
\ar[near start]{rr}{}
& & M^\gp.
\\
M
\ar{ur}{=}
\ar{rr}{}
& & M^\gp
\ar[crossing over, leftarrow, near start,swap]{uu}{f^\gp}
\ar[swap]{ur}{=}
\end{tikzcd}\]
The bottom face and the back vertical face are evidently cartesian, while the front face is cartesian by definition. Hence the top face is cartesian, which concludes the proof.  
\end{proof}

The notion of repletion extends from animated monoids to animated pre-log rings in the following manner:

\begin{definition}\label{def:prelogrep} Let $(p, p^\flat) \colon (B, N) \to (A, M)$ be a map of animated pre-log rings, and assume that $p^\flat$ is virtually surjective. The \emph{repletion} $p^\rep$ of $p$ is the canonical map \[p^\rep \colon B \otimes_{{\mathbb Z}[N]} {\mathbb Z}[N^\rep] \to A \otimes_{{\mathbb Z}[M]} {\mathbb Z}[M] \cong A,\] where $N^\rep$ is the repletion of the map $p^\flat$.
\end{definition} 

If $(p, p^\flat)$ corresponds to a map $Y \to X$ of affine derived pre-log schemes, we shall write $Y\to X^{\rep} \to X$ for the corresponding repletion (factoring the original map $Y\to X$). The following is a consequence of Proposition \ref{prop:repletesplit}:

\begin{corollary}\label{cor:sectionstrict} Let $(p, p^\flat) \colon (B, N) \to (A, M)$ be a morphism of animated pre-log rings which admits a section. Then the repletion $(p^\rep, (p^\flat)^\rep)$ is strict. \qed
\end{corollary}
\begin{remark}Let $Y\to X$ be a closed immersion of affine derived pre-log schemes (i.e., the corresponding map of animated pre-log rings induces a surjection on $\pi_0$), and assume that it has a section $X\to Y$. Then, by the previous Corollary, the repletion $Y\to X^\rep$ is a strict (hence exact) closed immersion.  
\end{remark}

\section{Log differentials and the log cotangent complex} \label{sec:differentials_and_cotangent}

Following \cite{SSV16}, 
we interpret Gabber's cotangent complex as the derived functor of log K\"ahler differentials. 
We discuss log \'etale maps between affine pre-log derived schemes and note that such maps have 
contractible cotangent complexes. 
It turns out that \emph{integral} log \'etale maps of ordinary affine log schemes are log \'etale in the derived sense.  

\subsection{The log differentials}\label{subsec:logkahler} Let $(f_0, f_0^\flat) \colon (R_0, P_0, \beta_0) \to (A_0, M_0, \alpha_0)$ be a map of discrete pre-log rings. The log K\"ahler differentials $\Omega^1_{(A_0, M_0) / (R_0, P_0)}$ are defined as the $A_0$-module \[(\Omega^1_{A_0 / R_0} \oplus (A_0 \otimes_{{\mathbb Z}} M_0^{\gp}))/\!\!\sim.\] The equivalence relation is $A_0$-linearly generated by $(d\alpha_0(m), 0) \sim (0, \alpha_0(m) \otimes \gamma_0(m))$ and $0 \sim (0, 1 \otimes \gamma(f_0^\flat(n)))$, where $\gamma_0 \colon M_0 \to M_0^{\gp}$ is the canonical map from $M_0$ to its group completion.  It is common to write $\dlog(m)$ for the element $(0, 1 \otimes \gamma_0(m))$, so that the first relation reads $d\alpha_0(m) = \alpha_0(m) \dlog(m)$ and the second relation reads $0 = \dlog(f_0^\flat(n))$. 
By definition, the log K\"ahler differentials of the canonical map $(R_0, P_0) \to F_{(R_0, P_0)}(X, Y)$ are given by
\begin{equation}\label{logkahlerfree}\Omega^1_{F_{(R_0, P_0)}(X, Y) / (R_0, P_0)} = R_0\langle X \sqcup Y \rangle \otimes_{\mathbb Z} (\bigoplus_{x \in X} {\mathbb Z}\{\dlog(x)\} \oplus \bigoplus_{y \in Y} {\mathbb Z}\{{\rm d}y\}).
\end{equation} 

Let $\PreLog$ denote the category of ordinary (i.e.,  discrete) pre-log rings.
The module of log K\"ahler differentials corepresent logarithmic derivations as in \cite[IV.1.1.1]{Ogu18} \[{\rm Der}_{(R_0, P_0)}((A_0, M_0), J_0) := {\rm Hom}_{\PreLog_{(R_0, P_0)//(A_0, M_0)}}((A_0, M_0), (A_0 \oplus J_0, M_0 \oplus J_0))\] with values in some $A_0$-module $J_0$. Here the pre-log structure $M_0 \oplus J_0 \to (A_0 \oplus J_0, \cdot)$ on the trivial square-zero extension $A_0 \oplus J_0$ is defined by $(m, j) \mapsto (\alpha_0(m), \alpha_0(m) \cdot j)$. This is the coproduct of the natural map $$M_0 \xrightarrow{\alpha_0} (A_0, \cdot) \xrightarrow{(d_0, \cdot)} (A_0 \oplus J_0, \cdot)$$ and the inclusion of the units $J_0 \cong 1 + J_0 \subset {\rm GL}_1(A_0 \oplus J_0)$ in $A_0 \oplus J_0$.

\subsection{The log differentials via the log diagonal}\label{subsec:logdifflogdiag} Recall that the module of relative differentials $\Omega^1_{A_0 / R_0}$ arises as the indecomposables of $A_0 \otimes_{R_0} A_0 \to A_0$, that is, the conormal bundle of the diagonal map. If $(R_0, P_0) \to (A_0, M_0)$ is a map of pre-log rings, this naturally leads to the question of how the multiplication map \[(A_0 \otimes_{R_0} A_0, M_0 \oplus_{P_0} M_0) \to (A_0, M_0)\] relates to the log differentials. It is not clear what should be meant by the module of indecomposables of this map; forming the indecomposables of the underlying map of commutative rings merely recovers the ordinary module of differentials $\Omega^1_{A_0 / R_0}$. 

However, after replacing the multiplication map with its repletion, we obtain a map \[((A_0 \otimes_{R_0} A_0) \otimes_{\bZ[M_0 \oplus_{P_0} M_0]} \bZ[(M_0 \oplus_{P_0} M_0)^{\rm rep}], (M_0 \oplus_{P_0} M_0)^{\rm rep}) \to (A_0, M_0)\] of pre-log rings. Moreover, this map is strict by Corollary \ref{cor:sectionstrict}. Hence it is determined by the underlying map of commutative rings, and it makes sense to define its module of indecomposables as that of the underlying map \begin{equation}\label{discrepaugm}(A_0 \otimes_{R_0} A_0) \otimes_{{\mathbb Z}[M_0 \oplus_{P_0} M_0]} {\mathbb Z}[(M_0 \oplus_{P_0} M_0)^{\rm rep}] \to A_0\end{equation} of commutative rings. As a formal consequence of \cite[Corollary 4.2.8(ii)]{KS04} (see also \cite[Remark IV.1.1.8]{Ogu18}), the module of indecomposables of \eqref{discrepaugm} is precisely the module of log differentials $\Omega^1_{(A_0, M_0) / (R_0, P_0)}$. Indeed, the map \eqref{discrepaugm} is a special case of Kato--Saito's \emph{log diagonal}. 

This implies that the log derivations ${\rm Der}_{(R_0, P_0)}((A_0, M_0), J_0)$ is naturally isomorphic to the set of augmented $A_0$-algebra maps \[{\rm Hom}_{{\rm CRing}_{A_0 // A_0}}((A_0 \otimes_{R_0} A_0) \otimes_{{\mathbb Z}[M_0 \oplus_{P_0} M_0]} {\mathbb Z}[(M_0 \oplus_{P_0} M_0)^\rep], A_0 \oplus J_0).\] This works as expected because log K\"ahler differentials are corepresented by an $A_0$-module. This observation and its derived analogue serve as a starting point for our approach to logarithmic Hochschild homology.

\subsection{Gabber's cotangent complex} We now discuss the Gabber cotangent complex in the context of animated pre-log rings. In addition to establishing some properties which will be essential to us, we explain how to obtain a description of it analogous to that of the log differentials given in Section \ref{subsec:logdifflogdiag}.

Let $(R,P) \to (A,M)$ be a morphism of animated pre-log rings, and let $J$ be a connective $A$-module.
Let $A\oplus J$ be the derived square-zero extension  of \cite[Remark 7.3.4.15]{HA}. 
There is a canonical map $A\oplus J \to A$, 
informally given by the projection onto the first component. 
It admits a canonical section $s\colon A \to A\oplus J$.
Let $(1 + J)$ be the fiber of the canonical projection ${\rm GL}_1(A \oplus J) \to {\rm GL}_1(A)$, 
and set $M \oplus J := M \oplus (1 + J)$ in grouplike animated monoids. 

We form the log square-zero extension associated with $J$ in the following way. 

\begin{definition}
We let $(A\oplus J, M\oplus J)$ be the animated pre-log ring with a structure map $M \oplus J \to (A \oplus J, \cdot)$ induced by the structure map $M \xrightarrow{\alpha} (A, \cdot) \xrightarrow{s} (A \oplus J, \cdot)$ and the inclusion $(1 + J) \to {\rm GL}_1(A \oplus J) \to (A \oplus J, \cdot)$.
\end{definition}

There is a canonical map $(A\oplus J, M\oplus J)\to (A,M)$, 
informally given by the projection onto the first component. 
It admits a canonical section $(A,M)\to (A\oplus J, M\oplus J)$. 
We define the space $\mathrm{Der}_{(R, P)}((A,M), J)$ of $(R,P)$-linear derivations of $(A,M)$ with values in $J$ 
as the fiber of the map 
\begin{equation}
\label{logderivationsfiber} 
\Map_{{\rm Ani(PreLog)}_{(R, P)/}}( (A,M) ,(A\oplus J, M\oplus J) ) 
\to \Map_{{\rm Ani(PreLog)}_{(R,P)/}}((A,M),(A,M)) 
\end{equation}
induced by $(A\oplus J, M\oplus J)\to (A,M)$ at the point $\mathrm{id}_{(A,M)}$. 
As in 
derived algebraic geometry \cite[1.4.1.14]{TV08},  
we say that $(R,P) \to (A,M)$ has a 
cotangent complex 
(or that $(A,M)$ admits a 
cotangent complex over $(R,P)$) 
if there exists a connective $A$-module co-representing the functor 
\[ 
\Mod_A^{ \rm cn} \to \mathcal{S}, \quad J \mapsto \mathrm{Der}_{(R, P)}((A,M), J). 
\]
We note that such a log cotangent complex is unique up to equivalence in the $\infty$-category $\Mod_A^{\rm cn}$ of  
connective $A$-modules.

Let $(R, P)$ be a simplicial pre-log ring. 
By extending the definition of the log K\"ahler differentials 
degreewise, 
we obtain a functor 
\[\Omega^1_{(-, -) / (R, P)} \colon {\rm sPreLog}_{(R, P)/} \to {\rm sMod}_R\] 
from the category of pre-log simplicial $(R, P)$-algebras to simplicial $R$-modules. 
As discussed in \cite[\S 4.2]{SSV16}, 
this is a left Quillen functor whose right adjoint is given by the levelwise square-zero construction 
$J \mapsto (A \oplus J, M \oplus J)$. 
This gives rise to an adjunction on the underlying $\infty$-categories; 
we denote the left adjoint by ${\mathbb L}$.

The following is the definition of the log cotangent complex pursued in \cite{SSV16}:

\begin{definition} Let $(R, P) \to (A, M)$ be a map of animated pre-log rings. The \emph{log cotangent complex} ${\mathbb L}_{(A, M)/(R, P)}$ is the $R$-module \[{\mathbb L}_{(A, M)/(R, P)} := {\mathbb L}(A, M).\] 
We note that ${\mathbb L}_{(A, M)/(R, P)}$ has a natural $A$-module structure.
Following Illusie and Quillen, 
we write $\Bigwedge^i_A {\mathbb L}_{(A, M)/(R, P)}$ for the $i$-th derived exterior power of ${\mathbb L}_{(A, M)/(R, P)}$.
\end{definition}

It is clear by construction that ${\mathbb L}_{(A, M)/(R, P)}$ co-represents the functor of log-derivations. 
In particular, every morphism of animated pre-log rings $(R,P)\to (A,M)$ admits a log cotangent complex. Well-known properties of the log cotangent complex (as recorded in e.g.\ \cite[Lemma 2.2.2.5]{Lun22}) that we will freely use include

\begin{enumerate}
\item strict invariance: the canonical map ${\Bbb L}_{A / R} \to {\Bbb L}_{(A, M) / (R, P)}$ is an equivalence for $(R, P) \to (A, M)$ strict;
\item base-change: if $(C, K)$ is the pushout of the diagram $(B, N) \xleftarrow{} (R, P) \xrightarrow{} (A, M)$, then the canonical map $C \otimes_B {\Bbb L}_{(B, N) / (R, P)} \to {\Bbb L}_{(C, K) / (A, M)}$ is an equivalence; and
\item transitivity: for a composite $(R, P) \to (A, M) \to (B, N)$, the sequence \begin{equation}\label{eq:transitivity}B \otimes_A {\Bbb L}_{(A, M) / (R, P)} \to {\Bbb L}_{(B, N) / (R, P)} \to {\Bbb L}_{(B, N) / (A, M)}\end{equation} is a cofiber sequence of $B$-modules.
\end{enumerate}

\begin{remark}
\label{rmk:qcohcotangent}
If $X \to S$ is a map of affine pre-log derived schemes, we  write ${\mathbb L}_{X/S}$ for the log cotangent complex. 
\end{remark}

\begin{remark} The cotangent complex defined above is a model for Gabber's cotangent complex, 
defined as the $A$-module 
${\mathbb L}_{(A, M)/(R, P)} := A \otimes_{A_{\bullet}} \Omega^1_{(A_{\bullet}, M_{\bullet}) / (R, P)}$ 
for a free $(R, P)$-algebra resolution $(A_{\bullet}, M_{\bullet}) \xrightarrow{\simeq} (A, M)$. 
This should not be confused with Olsson's log cotangent complex in \cite[Definition 3.2]{Ols05}. 
Note that the derived exterior powers can be modeled as 
$\Bigwedge^q_A {\mathbb L}_{(A, M) / (R, P)} = A \otimes_{A_{\bullet}} \Omega^q_{(A_{\bullet}, M_{\bullet}) / (R, P)}$.
\end{remark}

We now aim to describe some elementary properties of the log cotangent complex. 
Using the definition of log derivations in \eqref{logderivationsfiber} and the description of mapping spaces in ${\rm Ani(PreLog)}$ \eqref{mappingspace}, we learn that the log cotangent complex can be realized as the pushout of the diagram \begin{equation}\label{logcotangentpushout}{\mathbb L}_{A / R} \xleftarrow{} A \otimes_{{\mathbb Z}[M]} {\mathbb L}_{{\mathbb Z}[M] / {\mathbb Z}[P]} \xrightarrow{} A \otimes_{{\mathbb Z}[M]} {\mathbb L}_{({\mathbb Z}[M], M)/ ({\mathbb Z}[P], P)}\end{equation} of $A$-modules. 
The following describes the log cotangent complex in a way that is reminiscent of the 
description of the spectral log cotangent complex given by Rognes \cite[\S 11]{Rog09} and Sagave \cite{Sag14} 
(therein called \emph{logarithmic ${\rm TAQ}$}):

\begin{proposition}\label{prop:cotangentquotient} 
There is a canonical equivalence 
$$
{\mathbb L}_{({\mathbb Z}[M], M)/({\mathbb Z}[P], P)} 
\simeq 
{\mathbb Z}[M] \otimes_{\mathbb Z} (M^\gp/P^\gp).
$$
\end{proposition}

In Proposition \ref{prop:cotangentquotient}, we implicitly use \cite[Appendix Q]{FM94} (see also \cite[Lemma 2.10]{SSV16}) 
to model the grouplike animated monoid $M^\gp/P^\gp$ by an animated abelian group, 
so that the tensor product ${\mathbb Z}[M] \otimes_{{\mathbb Z}} (M^\gp/P^\gp)$ is defined.

\begin{proof}[Proof of \Cref{prop:cotangentquotient}] We show that there is a natural equivalence \begin{equation}\label{desiredequivalence}{\rm Map}_{{\rm Mod}_{{\mathbb Z}[M]}}({\mathbb L}_{({\mathbb Z}[M], M) / ({\mathbb Z}[P], P)}, J) \simeq {\rm Map}_{{\rm Mod}_{{\mathbb Z}[M]}}({\mathbb Z}[M] \otimes_{{\mathbb Z}} (M^{\rm gp}/P^{\rm gp}), J)\end{equation} of mapping spaces for any ${\mathbb Z}[M]$-module $J$. To see this, we first use that the log cotangent complex ${\mathbb L}_{({\mathbb Z}[M] , M) / ({\mathbb Z}[P], P)}$ corepresents log derivations \[{\rm Map}_{{\rm Ani(PreLog)}_{({\mathbb Z}[P], P)//({\mathbb Z}[M], M)}}(({\mathbb Z}[M], M), ({\mathbb Z}[M] \oplus J, M \oplus J)).\] Moreover, by the description of mapping spaces in \eqref{mappingspace}, this can be simplified to \[{\rm Map}_{{\rm Ani(CMon)}_{P//M}}(M, M \oplus J) \simeq {\rm Map}_{{\rm Ani(CMon)}_{M//M}}(M \oplus_P M, M \oplus J).\] Since the map $M \oplus J \to M$ is replete, the universal property of repletion implies that there is a natural equivalence \[{\rm Map}_{{\rm Ani(CMon)}_{M//M}}(M \oplus_P M, M \oplus J) \simeq {\rm Map}_{{\rm Ani(CMon)}_{M//M}}((M \oplus_P M)^{\rm rep}, M \oplus J).\] We now apply \Cref{prop:repletesplit} to infer an equivalence \begin{equation}{\rm Map}_{{\rm Ani(CMon)}_{M//M}}((M \oplus_P M)^\rep, M \oplus J) \simeq {\rm Map}_{{\rm Ani(CMon)}_{M//M}}(M \oplus (M^\gp/P^\gp), M \oplus J).\end{equation} By restriction of scalars and pullback along $* \to M$, we obtain a natural equivalence \[{\rm Map}_{{\rm Ani(CMon)}_{M//M}}(M \oplus (M^\gp/P^\gp), M \oplus J) \simeq {\rm Map}_{\rm Ani(Ab)}(M^{\rm gp}/P^{\rm gp}, J).\] By extension of scalars along ${\mathbb Z} \to {\mathbb Z}[M]$ we obtain the equivalence \eqref{desiredequivalence}, which concludes the proof. 
\end{proof}

The following Corollary is of fundamental importance for us since it makes it possible to infer properties of the log cotangent complex from its non-log counterpart, together with some simple information from an animated abelian group.

\begin{corollary}
\label{cor:logcotangentseq}
Let $(R, P) \to (A, M)$ be a map of animated pre-log rings. There is a cofiber sequence of $A$-modules
\[
A \otimes_{\mathbb Z} (M^\gp/P^\gp) \to {\mathbb L}_{(A, M)/(R, P)} 
\to {\mathbb L}_{A / R \otimes_{{\mathbb Z}[P]} {\mathbb Z}[M]}. 
\] 
\end{corollary}

\begin{proof} Recall that all tensor products are implicitly derived. There is a pushout square
\[\begin{tikzcd}[row sep = small]A \otimes_{{\mathbb Z}[M]} {\mathbb L}_{{\mathbb Z}[M] / {\mathbb Z}[P]} \ar{r} \ar{d} & A \otimes_{{\mathbb Z}[M]} {\mathbb L}_{({\mathbb Z}[M], M) / ({\mathbb Z}[P], P)} \ar{d} \\ {\mathbb L}_{A / R} \ar{r} & {\mathbb L}_{(A, M)/(R, P)}\end{tikzcd}\] of $A$-modules \eqref{logcotangentpushout}. By Proposition \ref{prop:cotangentquotient}, there is an equivalence \[A \otimes_{{\mathbb Z}[M]} {\mathbb L}_{({\mathbb Z}[M], M)/({\mathbb Z}[P], P)} \simeq A \otimes_{\mathbb Z} (M^{\gp}/P^{\gp}).\] Thus it suffices to prove that the upper row in the diagram \begin{equation}\label{desiredseq}\begin{tikzcd}[row sep = small]A \otimes_{{\mathbb Z}[M]} {\mathbb L}_{{\mathbb Z}[M] / {\mathbb Z}[P]} \ar{r} \ar{d} & {\mathbb L}_{A/R} \ar{r} \ar{d}{=} & {\mathbb L}_{A / R \otimes_{{\mathbb Z}[P]} {\mathbb Z}[M]} \ar{d}{=} \\ A \otimes_{R \otimes_{{\mathbb Z}[P]} {\mathbb Z}[M]} {\mathbb L}_{R \otimes_{{\mathbb Z}[P]} {\mathbb Z}[M] / R} \ar{r} & {\mathbb L}_{A / R} \ar{r} & {\mathbb L}_{A / R \otimes_{{\mathbb Z}[P]} {\mathbb Z}[M]}\end{tikzcd}\end{equation} is a cofiber sequence. Here the lower row is the transitivity sequence associated to the composite $R \to R \otimes_{{\mathbb Z}[P]} {\mathbb Z}[M] \to A$; thus, it is a cofiber sequence. The left-hand map is an equivalence by derived base change for the classical cotangent complex applied to the pushout of the diagram $R \xleftarrow{} {\mathbb Z}[P] \xrightarrow{} {\mathbb Z}[M]$, from which we conclude the proof.
\end{proof}

The proof of Corollary \ref{cor:logcotangentseq} uses derived base change for the classical cotangent complex  \cite[\href{https://stacks.math.columbia.edu/tag/08QQ}{Tag 08QQ}]{stacks-project}. 
This suggests that some flatness hypothesis is necessary to transport results from the derived setting to 
the discrete case.
In the log setting, this is guaranteed by the integrality condition:
A discrete commutative monoid $M_0$ is \emph{integral} if the canonical map $M_0 \to M_0^\gp$ is an injection, while a map $P_0 \to M_0$ of discrete commutative monoids is \emph{integral} if its pushout along any map to an integral monoid remains integral \cite[Definition I.4.6.2(3)]{Ogu18}. This holds for any map of integral monoids $P_0 \to M_0$ provided  $P_0$ is valuative  \cite[Proposition I.4.6.3(5)]{Ogu18}. 
In particular, 
the canonical pre-log structure on any discrete valuation ring is valuative. 

\begin{corollary}\label{cor:discretecofiber}  Let $(R_0, P_0) \to (A_0, M_0)$ be a map of discrete pre-log rings. If $P_0 \to M_0$ is integral and $P_0^\gp \to M_0^\gp$ is injective, there is a cofiber sequence of $A_0$-modules \[A_0 \otimes_{\mathbb Z} (M_0^\gp/P_0^\gp) \to {\mathbb L}_{(A_0, M_0)/(R_0, P_0)} \to {\mathbb L}_{A_0 / R_0 \otimes_{{\mathbb Z}[P_0]} {\mathbb Z}[M_0]}. 
\] 
\end{corollary}

The integrality hypothesis is essential, see Example \ref{exm:logetale}. Corollary \ref{cor:discretecofiber} is particularly convenient when relating the log cotangent complex with the characterization of log \'etale and smooth maps in terms of charts \cite[Theorem 3.5]{Kat89}, see Proposition \ref{prop:cotangent_vs_omega}.

\begin{proof}  
The flat base change argument in the proof of Corollary \ref{cor:logcotangentseq} goes through since the 
integrality hypothesis implies that ${\mathbb Z}[P_0] \to {\mathbb Z}[M_0]$ is flat by \cite[Remark I.4.6.6]{Ogu18}.
Moreover, 
by \cite[Lemma 8.18(ii)]{Ols05}, 
there is an equivalence of $A_0$-modules
\[
A_0 \otimes_{{\mathbb Z}[M_0]} {\mathbb L}_{({\mathbb Z}[M_0], M_0)/({\mathbb Z}[P_0], P_0)} 
\simeq 
A_0 \otimes_{\mathbb Z} (M_0^{\gp}/P_0^{\gp}).
\]  \end{proof}

Since the log cotangent complex corepresents (derived) log derivations, 
the argument of \cite[Corollary 6.7]{Sag14} yields

\begin{lemma}\label{lem:logcotangentloginv} The logification construction induces equivalences of $A$-modules 
\[{\mathbb L}_{(A, M)/(R, P)} \xrightarrow{\simeq} {\mathbb L}_{(A, M^a) / (R, P)} 
\xrightarrow{\simeq} {\mathbb L}_{(A, M^a)/(R, P^a)}.
\] 
\end{lemma}

We shall also have occasion to use that the log cotangent complex of discrete pre-log rings is invariant under 
logification by mimicking Definition \ref{def:logarithmic} using ordinary pullbacks and pushouts. 
This is the content of \cite[Theorem 8.16]{Ols05}.

\begin{corollary}\label{Cor:strict_cot}
Let $f\colon (R_0, P_0)\to (A_0,M_0)$ be a strict map of discrete pre-log rings. Then ${\mathbb L}_{(A_0,M_0)/(R_0,P_0)} \simeq \mathbb{L}_{A_0/R_0}$, where the latter denotes the classical relative cotangent complex. In particular, if $f$ is strict \'etale, then  ${\mathbb L}_{(A_0,M_0)/(R_0,P_0)}$ is contractible.
\end{corollary}
\begin{proof}
Since the construction is invariant under logification, 
we may assume that the log structure on $(A_0,M_0)$ is induced by the inverse image pre-log structure 
$P_0\to (R_0, \cdot) \to (A_0, \cdot)$. 
Now the equivalence ${\mathbb L}_{(A_0,M_0)/(R_0,P_0)} \simeq \mathbb{L}_{A_0/R_0}$ follows immediately from 
Corollary \ref{cor:logcotangentseq}.  
The last statement is a consequence of \cite[\href{https://stacks.math.columbia.edu/tag/08R2}{Tag 08R2}]{stacks-project}.
\end{proof}

Let $(R,P)\xrightarrow{(f,f^\flat)} (A,M)\xrightarrow{(g,g^\flat)} (B,N)$ be maps of animated pre-log rings. By combining \eqref{eq:transitivity} with Corollary \ref{Cor:strict_cot}, 
we see that for every strict \'etale map $(A_0, M_0)\to (B_0, N_0)$ of discrete pre-log rings, 
there is a natural equivalence
\begin{equation}
\label{eq:strictet_bc}
B_0 \otimes_{A_0} \mathbb{L}_{(A_0,M_0)/(R_0, P_0)} 
\xrightarrow{\simeq}
\mathbb{L}_{(B_0, N_0)/(R_0, P_0)}.    
\end{equation}

\subsection{Derived indecomposables}\label{subsec:derivedindec} Let $A$ be an animated commutative ring, 
and let $A \to B \to A$ be an animated augmented $A$-algebra. 
We write $Q_A(B)$ for the $A$-module of indecomposables in $B$. 
For example, 
$Q_A(-)$ can be realized as the functor $Q_A(-) \colon {\rm Ani(Ring)}_{A//A} \to {\rm Mod}_A$ of $\infty$-categories associated 
to the left Quillen functor $Q_A(-) \colon {\rm sCRing}_{A // A} \to {\rm sMod}_A$ obtained by taking indecomposables levelwise; 
its right adjoint is 
determined by the trivial square-zero extension. If $A \to C$ is a map of animated commutative rings, there is a base-change formula \[Q_C(C \otimes_A B) \simeq C \otimes_A Q_A(B)\] which we will use freely.

\subsection{The replete diagonal map} 
In what follows, 
we use the repletion to describe the log cotangent complex as the conormal of the \emph{replete} diagonal.

\begin{definition}
Suppose $f\colon X\to S$ is a map of affine pre-log derived schemes.
The \emph{replete diagonal map}
\[
\Delta^{\rep}
\colon
X\to X\times_S^{\rep}X
\]
is the repletion of the diagonal map $\Delta \colon X\to X\times_S X$.
\end{definition}

\begin{remark} If $f \colon X \to S$ is a map of ordinary affine log schemes, 
the replete diagonal map defined above is a derived analogue of Kato-Saito's log diagonal map 
$X\to X\times_{S,[P]}^{\log} X$ \cite[Corollary 4.2.8(i)]{KS04}. 
See also \cite[Example III.2.3.6]{Ogu18} for an equivalent construction. 
\end{remark}

The following gives a description of the log cotangent complex analogous to that of the log differentials in Section \ref{subsec:logdifflogdiag}:

\begin{proposition}\label{prop:logcotangentreplete} Let $X \to S$ be a map of affine pre-log derived schemes. The cotangent complex ${\mathbb L}_{X / S}$ is equivalent to the conormal of $\ul{\Delta^\rep}$, the map of derived schemes underlying the replete diagonal $\Delta^\rep \colon X \to X \times_S^\rep X$.
\end{proposition}

For a morphism of affine derived schemes arising from a map of animated commutative rings admitting a section, 
we define its indecomposables as the module obtained by applying the indecomposable functor levelwise.
It may be helpful to observe that the ordinary cotangent complex ${\mathbb L}_{A / R}$ arises as the indecomposables of $A \otimes_R A \to A$: Indeed, the derived tensor product $A \otimes_R A$ can be computed via a polynomial $R$-algebra resolution $A_{\bullet} \xrightarrow{\simeq} A$, and the indecomposables of \[A \otimes_R A_{\bullet} \cong A \otimes_{A_{\bullet}} (A_{\bullet} \otimes_R A_{\bullet}) \to A\] is $A \otimes_{A_{\bullet}} \Omega^1_{A_{\bullet} / R}$, 
which is the definition of the cotangent complex ${\mathbb L}_{A / R}$. 

\begin{proof}[Proof of Proposition \ref{prop:logcotangentreplete}] 
The replete diagonal corresponds to an augmentation 
\[
((A \otimes_R A) \otimes_{{\mathbb Z}[M \oplus_P M]} {\mathbb Z}[(M \oplus_P M)^\rep], (M \oplus_P M)^\rep) 
\to 
(A, M)
\] 
of animated pre-log rings. 
This evidently admits a section, 
so that Corollary \ref{cor:sectionstrict} implies the map is strict.
The module in question is therefore the levelwise indecomposables of 
\begin{equation}\label{indecofrep}(A \otimes_R A) \otimes_{{\mathbb Z}[M \oplus_P M]} {\mathbb Z}[(M \oplus_P M)^\rep] \to A.\end{equation} 
Since the ordinary cotangent complex $\bL_{A / R}$ is the indecomposables of $A \otimes_R A \to A$, and similarly for $\bL_{\bZ[M] / \bZ[P]}$, the indecomposables of \eqref{indecofrep} is the pushout of the diagram \[\bL_{A / R} \xleftarrow{} A \otimes_{\bZ[M]} \bL_{\bZ[M] / \bZ[P]} \to Q_A(A \otimes_{\bZ[M]} \bZ[(M \oplus_P M)^{\rm rep}]),\] where $Q_A(-) \colon {\rm CAlg}_{A // A} \to {\rm Mod}_A$ is the indecomposables functor. By Proposition \ref{prop:repletesplit}, the right-hand $A$-module is equivalent to $Q_A(A \otimes_{\bZ} \bZ[M^{\rm gp}/P^{\rm gp}]) \simeq A \otimes_{\bZ} Q_{\bZ}(\bZ[M^{\rm gp}/P^{\rm gp}])$. Since, for any group $G_0$, the module of indecomposables of the augmentation $\bZ[G_0] \to \bZ$ is isomorphic to abelianization of $G_0$, this further identifies $A \otimes_{\bZ} Q_{\bZ}(\bZ[M^{\rm gp}/P^{\rm gp}])$ with $A \otimes_{\bZ} (M^{\rm gp}/P^{\rm gp})$. Now Proposition \ref{prop:cotangentquotient} and \eqref{logcotangentpushout} apply to conclude the proof. 
\end{proof}

\begin{remark} 
Proposition \ref{prop:logcotangentreplete} gives rise to a model-independent construction of the log cotangent complex. Indeed, since we have described ${\mathbb L}_{(A, M) / (R, P)}$ as the indecomposables of a certain augmented $A$-algebra, which can be used to describe the ${\mathbb E}_{\infty}$-cotangent complex, we may now mimic the approach of \cite[7.3.5]{HA}.  We refer to \cite[\S 9]{Lun21} for details. The second-named author is currently pursuing a logarithmic variant of the cotangent complex formalism of 
\cite[7.3]{HA}, and \cite[\S 5]{Lun22} contains the expected identification of the fibers of the resulting \emph{replete tangent bundle} with the category of ordinary $A$-modules.
\end{remark}

\section{{\'E}tale and smooth morphisms in derived log geometry}

We shall use the following notion of formal \'etaleness in the context of derived log geometry from \cite[Definition 5.3]{SSV16}:
\begin{definition}\label{def:logetale} A map $(f, f^\flat) \colon (R, P) \to (A, M)$ of animated pre-log rings is \emph{derived formally log \'etale} if the log cotangent complex ${\mathbb L}_{(A, M) / (R, P)}$ vanishes.
\end{definition}

Following \cite{SSV16}, we say that a map $(f, f^\flat)$ is derived log \'etale if it is formally derived log \'etale and if $f$ is homotopically finitely presented in the sense of \cite[Definition 1.2.3.1]{TV08}.   By \cite[Theorem 5.6]{SSV16}, a derived log \'etale map $(f, f^\flat)$ such that $\pi_0(f^\flat)$ is finitely presented induces a log \'etale map $(\pi_0(f), \pi_0(f^\flat))$ of discrete log rings in the sense that $(\pi_0(f), \pi_0(f^\flat))$ lifts uniquely against log square-zero extensions \cite{Kat89}. 
In particular, \cite[Theorem 3.5]{Kat89} shows that, locally on charts, 
$(\pi_0(f), \pi_0(f^\flat))$ satisfies the following properties:

\begin{enumerate}
\item $\pi_0(R) \otimes_{{\mathbb Z}[\pi_0(P)]} {\mathbb Z}[\pi_0(M)] \to \pi_0(A)$ is \'etale;
\item $\pi_0(M)^\gp/\pi_0(P)^\gp$ is finite of order invertible in $\pi_0(A)$. 
\end{enumerate}

In future work, we intend to develop the deformation theory of derived log schemes.
We prove a derived analogue of \cite[Theorem 3.5]{Kat89}, 
obtaining several equivalent notions of log \'etaleness in the context of derived log geometry. 

\begin{df}
A map $(f, f^\flat) \colon (R, P) \to (A, M)$ of animated pre-log rings is \emph{derived formally log smooth} if the log cotangent complex ${\mathbb L}_{(A, M) / (R, P)}$ is a projective $A$-module, i.e., a direct summand of a free $A$-module.
\end{df}

Our definition of derived formal log smoothness is equivalent to \cite[Definition 6.3]{SSV16} by 
\cite[Propositions 7.2.2.6, 7.2.2.7]{HA}, \cite[1.2.8.3]{TV08}.
One can define a map of animated pre-log rings to be \emph{derived log smooth} if the underlying map of 
animated commutative rings is homotopically finitely presented.
By \cite[Theorem 6.4]{SSV16}, this yields a log smooth map of log rings on $\pi_0$. 

\begin{proposition}
\label{prop:composition}
The classes of derived formally log \'etale morphisms and derived formally log smooth morphisms of animated pre-log rings are closed under compositions and pushouts.
\end{proposition}
\begin{proof}
Let $(R,P)\xrightarrow{(f,f^\flat)} (A,M)\xrightarrow{(g,g^\flat)} (B,N)$ be maps of animated pre-log rings, and consider the transitivity sequence \eqref{eq:transitivity}.
If $(f,f^\flat)$ and $(g,g^\flat)$ are derived formally log smooth, 
then $B\otimes_A \bL_{(A,M) / (R,P)}$ and $\bL_{(B,N) / (A,M)}$ are projective $B$-modules.
Hence $\bL_{(B,N) / (R,P)}$ is a projective $B$-module by \cite[Proposition 7.2.2.6]{HA}.
It follows that $(gf,g^\flat f^\flat)$ is derived formally log smooth.
Let
\[
\begin{tikzcd}
(R,P)\ar[d]\ar[r,"{(f,f^\flat)}"]&
(A,M)\ar[d]
\\
(R',P')\ar[r,"{(f',f'^\flat)}"]&
(A',M')
\end{tikzcd}
\]
be a pushout square of animated pre-log rings.
By \cite[Proposition 4.12(ii)]{SSV16}, 
we have 
\[
A'\otimes_A \bL_{(A,M) / (R,P)}
\simeq\bL_{(A',M')/(R',P')}.
\]
The assumption on $(f,f^\flat)$ implies that 
$\bL_{(A,M)/(R,P)}$ is a projective $A$-module.
Therefore, 
$\bL_{(A',M')/(R',P')}$ is a projective $A'$-module, 
i.e., 
$(f',f'^\flat)$ is derived formally log smooth.

The claim for derived formally log \'etale maps is shown similarly. 
\end{proof}

\subsection{Derived formally log \'etale maps of discrete log rings}  

Example \ref{exm:logetale} below shows that log \'etale maps of ordinary log rings need not give rise to derived formally log \'etale maps of affine pre-log derived schemes. This is in contrast to ordinary derived algebraic geometry and closely related to the fact that the log cotangent complex may fail to be contractible for a given log \'etale map. However, a derived formally log \'etale map has a vanishing cotangent complex by definition. Derived formally log smooth maps enjoy the property of being concentrated as the log differentials in degree zero:

\begin{proposition}
\label{prop:cotangent_vs_omega2}
If $(R_0, P_0) \xrightarrow{(f_0, f^\flat_0)} (A_0, M_0)$ is a derived formally log smooth map of discrete pre-log rings, 
there is a canonical equivalence
\[
{\mathbb L}_{(A_0, M_0)/(R_0, P_0)} 
\xrightarrow{\simeq}
\pi_0({\mathbb L}_{(A_0, M_0) / (R_0, P_0)}) 
\cong 
\Omega^1_{(A_0, M_0) / (R_0, P_0)}.
\]
\end{proposition}
\begin{proof}
Since ${\mathbb L}_{(A_0, M_0) / (R_0, P_0)}$ is a direct summand of a free $A_0$-module and $A_0$ is discrete, ${\mathbb L}_{(A_0, M_0) / (R_0, P_0)}$ is concentrated in degree $0$.
This implies the claim.
\end{proof}

As suggested by the application of base change in the proof of Corollary \ref{cor:logcotangentseq} and the resulting integrality hypothesis in Corollary \ref{cor:discretecofiber}, flatness hypotheses are necessary when passing from ordinary to derived log geometry. Corollary \ref{cor:discretecofiber} is particularly convenient when using the characterization of log \'etale and log smooth morphisms in terms of charts, for which we need to impose an additional finiteness hypothesis:

\begin{proposition}
\label{prop:cotangent_vs_omega}
Let $(R_0, P_0) \xrightarrow{(f_0, f^\flat_0)} (A_0, M_0)$ be an integral map of discrete fine pre-log rings. 

\begin{enumerate}
\item If $(f_0, f_0^\flat)$ is log smooth, then it is derived log smooth.
\item If $(f_0, f_0^\flat)$ is log \'etale, then it is derived log \'etale.
\end{enumerate}
\end{proposition}

\begin{proof}We give a proof of (1), since (2) is similar (in fact, easier). Using the base change property for strict \'etale maps \eqref{eq:strictet_bc}, one readily verifies that the claims can be checked strict \'etale locally on $(A_0, M_0)$. Indeed, suppose $(A_0,M_0)\to (B_0,N_0)$ is a strict \'etale cover such that  $(R_0,P_0)\to (B_0,N_0)$ is derived log smooth.
Then $B_0\otimes_{A_0}\bL_{(A_0,M_0) / (R_0,P_0)} \simeq \mathbb{L}_{(B_0, N_0) / (R_0, P_0)}$ is a finitely generated projective $B_i$-module.
By fppf descent  \cite[03DX]{stacks-project}, we deduce that $\bL_{(A_0,M_0) / (R_0,P_0)}$ is a finitely generated projective $A_0$-module as well.

Since the log cotangent complex is invariant under logification \cite[Theorem 8.16]{Ols05}, 
it suffices to check the statement on charts. 
Since the monoids are assumed to be fine, 
by appealing to \cite[Theorem 3.5]{Kat89} or \cite[Theorem IV.3.3.1]{Ogu18} we may assume
\begin{enumerate}
\item $R_0 \otimes_{{\mathbb Z}[P_0]} {\mathbb Z}[M_0] \to A_0$ is smooth, and
\item the induced map $P_0^\gp \to M_0^\gp$ is injective, 
and the torsion part of $M_0^\gp/P_0^\gp$ is of finite order invertible in $A_0$.
\end{enumerate}
After a further localization, 
\cite[Theorem IV.3.3.1]{Ogu18} allows us to assume that $P_0 \to M_0$ remains integral. 
Hence Corollary \ref{cor:discretecofiber} yields a cofiber sequence 
\[
A_0 \otimes_{\mathbb Z}^{\mathbb L} (M_0^\gp/P_0^\gp) \to 
{\mathbb L}_{(A_0, M_0) / (R_0, P_0)} \to 
{\mathbb L}_{A_0 / R_0 \otimes_{{\mathbb Z}[P_0]} {\mathbb Z}[M_0]}.
\] 

The right-hand module is concentrated in degree zero as the differentials by (1), 
while the left-hand module is concentrated in degree zero as the free $A_0$-module of rank that of the 
free part of $M_0^\gp/P_0^\gp$ by (2). 
Hence the middle term is concentrated in degree zero as the log differentials.
\end{proof}

In summary, there are implications for morphisms of discrete fine pre-log rings:
\begin{gather}
\label{eq:implications1}
\text{(integral log smooth)}
\Rightarrow
\text{(derived log smooth)}
\Rightarrow
\text{(log smooth),}
\\
\label{eq:implications2}
\text{(integral log \'etale)}
\Rightarrow
\text{(derived log \'etale)}
\Rightarrow
\text{(log \'etale).}
\end{gather}

\noindent Here the first implications in both (\ref{eq:implications1}) and (\ref{eq:implications2}) follow from \Cref{prop:cotangent_vs_omega}, while second implications follow from \cite[Theorem 5.6 and Theorem 6.4]{SSV16}.

\begin{example} Recall that the replete diagonal participates in a factorization \[X \xrightarrow{\Delta^\rep} X \times_S^\rep X \to X \times_S X\] of the diagonal map associated to a map $X \to S$ of affine derived pre-log schemes. The second map in this composite is an example of a log \'etale map according to Corollary \ref{cor:logcotangentseq}. This is analogous to the log diagonal studied by Kato--Saito \cite[\S 4]{KS04}.
\end{example}

\begin{example}
\label{exm:logetale2}
Let $A_0$ be a ring, and let $\theta\colon P_0 \to M_0$ be a (not necessarily integral) map of finitely generated monoids. 
We assume that the kernel and the torsion part of the cokernel of $\theta^\gp$ are finite groups with order invertible in $A_0$. This implies that the tensor product $A_0 \otimes_{\mathbb Z}^{\mathbb L} (M_0^\gp/P_0^\gp)$ is a free $A_0$-module concentrated in degree zero. Moreover, it does not make a difference whether we consider the ordinary or derived quotient $M_0^\gp/P_0^\gp$ after inducing up to $A_0$. Hence Corollary \ref{cor:logcotangentseq} implies that ${\mathbb L}_{(A_0[M_0], M_0) / (A_0[P_0], P_0)}$ is a free $A_0$-module concentrated in degree zero.
This means that the naturally induced map $f\colon (A_0[P_0],P_0)\to (A_0[M_0],M_0)$ is derived log smooth even if we do not assume that $\theta$ is integral.
If we further assume that the cokernel of $\theta^\gp$ is a finite group whose order is invertible in $A_0$, 
the same argument shows there is a canonical equivalence \[A_0 \otimes_{\mathbb Z}^{\mathbb L} (M_0^\gp/P_0^\gp) \xrightarrow{\simeq} \bL_{(A_0[M_0],M_0) / (A_0[P_0],P_0)}\] with contractible source, so that $f$ is derived log \'etale.
\end{example}

\begin{example}
\label{exm:logetale}
Let $k$ be a field of characteristic $0$, and let $P$ be the submonoid of $\bN^2$ generated by $(2,0)$, $(1,1)$, and $(0,2)$.
The inclusion $P\to \bN^2$ naturally induces
\[
f\colon (k[P],P)\to (k[\bN^2],\bN^2).
\]
Example \ref{exm:logetale2} shows that $f$ is  derived log \'etale.
The homomorphism $u\colon k[P]\to k$ sending all elements of $P$ to $0$ gives the pushout map
\[
f'\colon (k,P)\to (k\otimes_{k[P]}k[\bN^2],\bN^2).
\]
If $\bL_{(k\otimes_{k[P]}k[\bN^2],\bN^2) / (k,P)}$ is contractible, 
the argument of Bauer in \cite[\S 7.2]{Ols05} would imply there is a naturally induced equivalence
\[
(k\otimes_{k[P]}k[\bN^2])\otimes_k \bL_{(k,P) / (k[P],P)}
\xrightarrow{\simeq}
\bL_{(k\otimes_{k[P]}k[\bN^2],\bN^2) / (k[\bN^2],\bN^2)}
\]
This is false, however, as observed in \cite[Example 7.3]{Ols05}.
Thus $f'$ is \emph{not} derived log \'etale even though $f'$ is log \'etale.
The reason for this discrepancy is that $f'$ is the pullback of $f$ but not the homotopy pullback of $f$, i.e.,
$
k\otimes_{k[P]}k[\bN^2]
\not\simeq \pi_0(
k\otimes_{k[P]}k[\bN^2])
$.
\end{example}

\section{Logarithmic Hochschild homology of animated pre-log rings} 
\label{section:lhh}
After a brief review of Hochschild homology, 
we define log Hochschild homology of animated pre-log rings and 
note that it recovers Rognes' definition for discrete pre-log rings. 
Throughout, 
all tensor products are implicitly derived unless explicitly stated otherwise.

\subsection{Hochschild homology} Let $R \to A$ be a map of animated commutative rings. The \emph{Hochschild homology} $\HH(A / R)$ of $A$ relative to $R$ is the colimit $S^1 \otimes_R A$ of the constant diagram $A$ from $S^1$ to animated commutative $R$-algebras. The maps $* \to S^1 \to *$ exhibit $A \to {\rm HH}(A / R) \to A$ as an animated augmented commutative $A$-algebra, while the equivalence $S^1 \simeq * \sqcup_{* \sqcup *} *$ gives an equivalence $\HH(A / R) \simeq A \otimes_{A \otimes_R A} A$. If $\ul{X} \to \ul{S}$ denotes the corresponding map of affine derived schemes, we see that \[\HH({\ul{X}} / {\ul{S}}) \simeq {\ul X} \times_{\ul{X} \times_{\ul{S}} \ul{X}} \ul{X},\] 
the derived self-intersection of the derived diagonal $\Delta \colon {\ul{X}} \to {\ul{X}} \times_{\ul{S}} \ul{X}$.

If $R_0 \to A_0$ is a map of discrete commutative rings, 
we note there are isomorphisms (see e.g., \cite[Corollary 1.18, Proposition 1.1.10]{Lod92})
\begin{equation}
\label{twofirsthh}
\pi_0\HH(A_0 / R_0)\cong A_0, 
\qquad \pi_1\HH(A_0 / R_0)\cong \Omega^1_{A_0 / R_0}. 
\end{equation}

\subsection{Log Hochschild homology} 
Inspired by the description of Hochschild homology as the derived self-intersections of the diagonal 
and the relationship between the replete diagonal and the log cotangent complex described in 
Proposition \ref{prop:logcotangentreplete}, 
we propose the following definition 
(the reader is invited to compare with \cite[\S 13]{Rog09}):

\begin{definition}\label{def:loghh} Let $(R, P) \to (A, M)$ be a map of animated pre-log rings. The \emph{log Hochschild homology} ${\rm logHH}((A, M) / (R, P))$ is the animated commutative ring \[A\otimes_{(A\otimes_R A)_{M \oplus_P M}^{\rep}} A = A \otimes_{(A \otimes_R A) \otimes_{{\mathbb Z}[M \oplus_P M]} {\mathbb Z}[(M \oplus_P M)^\rep]} A.\]
\end{definition}

Let $X \to S$ be the map of affine pre-log derived schemes corresponding to $(R, P) \to (A, M)$. 
Then log Hochschild homology 
${\rm logHH}(X / S)$ is by definition the affine derived scheme \[\underline{X \times_{X \times_S X}^\rep X}\] underlying the derived self-intersections $X \times_{X \times_S X}^\rep X$ of the replete diagonal $X \to X \times_S^\rep X$. 

\begin{proposition}\label{prop:basechangeHH}
Let $(R, P) \to (A, M) \to (B, N)$ be maps of animated pre-log rings. There is an equivalence of animated commutative rings 
\[
A \otimes_{\logHH((A, M) / (R, P))} \logHH((B, N) / (R, P)) \simeq \logHH((B, N) / (A, M)).
\] 
\end{proposition}

\begin{proof} It suffices to prove that \begin{equation}\label{itsufficestoprovethat} A \otimes_{(A \otimes_R A)^\rep_{M \oplus_P M}} (B \otimes_R B)^\rep_{N \oplus_P N} \simeq (B \otimes_A B)^\rep_{N \oplus_M N}.\end{equation} Since $A \otimes_{A \otimes_R A} (B \otimes_R B) \simeq B \otimes_A B$ and ${\mathbb Z}[M] \otimes_{{\mathbb Z}[M \oplus_P M]} {\mathbb Z}[N \oplus_P N] \simeq {\mathbb Z}[N \oplus_M N]$, \eqref{itsufficestoprovethat} follows from the equivalences \begin{align*}{\mathbb Z}[M] \otimes_{{\mathbb Z}[(M \oplus_P M)^\rep]} {\mathbb Z}[(N \oplus_P N)^\rep] &\simeq {\mathbb Z}[M] \otimes_{{\mathbb Z}[M \oplus (M^\gp/P^\gp)]} {\mathbb Z}[N \oplus (N^\gp/P^\gp)] \\ &\simeq {\mathbb Z}[N \oplus (N^\gp/M^\gp)] \\ &\simeq {\mathbb Z}[(N \oplus_M N)^\rep]\end{align*} provided by Proposition \ref{prop:repletesplit}.
\end{proof}

\subsection{Comparison with Rognes' definition of log \texorpdfstring{$\HH$}{HH}} 
In \cite[Definition 3.21]{Rog09}, 
Rognes introduced absolute log Hochschild homology of discrete log rings. A direct generalization goes as follows: for a map $(R, P) \to (A, M)$ of animated pre-log rings, let $S^1 \oplus_P M$ denote the tensor with $S^1$ in the $\infty$-category of animated commutative monoids under $P$. Let $(S^1 \oplus_P M)^\rep$ denote the repletion of the augmentation $S^1 \oplus_P M \to M$, and consider the pushout of the diagram \begin{equation}\label{rognesloghh}S^1 \otimes_R A \xleftarrow{} S^1 \otimes_{{\mathbb Z}[P]} {\mathbb Z}[M] \simeq {\mathbb Z}[S^1 \oplus_P M] \xrightarrow{} {\mathbb Z}[(S^1 \oplus_P M)^\rep]\end{equation}  in the $\infty$-category of animated commutative rings.

\begin{proposition}
\label{prop:rognescomparison2} Let $(R, P) \to (A, M)$ be a map of animated
pre-log rings. The pushout of the diagram \eqref{rognesloghh} is naturally equivalent to ${\rm logHH}((A, M) / (R, P))$. 
\end{proposition}

\begin{proof} There are natural equivalences 
\[
S^1 \otimes_R A \simeq A \otimes_{A \otimes_R A} A 
\quad 
\text{and} 
\quad 
S^1 \otimes_{{\mathbb Z}[P]} {\mathbb Z}[M] \simeq {\mathbb Z}[M] \oplus_{{\mathbb Z}[M \oplus_P M]} {\mathbb Z}[M].
\] Moreover, we see that \[{\mathbb Z}[(S^1 \oplus_P M)^\rep] \simeq {\mathbb Z}[M] \otimes_{{\mathbb Z}[(M \oplus_P M)^\rep]} {\mathbb Z}[M]\] by applying Proposition \ref{prop:repletesplit} twice and the equivalence $B^{\rm cy}_{P^\gp}(M^\gp)/M^\gp \simeq B(M^\gp/P^\gp)$: Indeed, there are equivalences of animated monoids \[B^{\rm cy}_{P}(M)^\rep \simeq M \oplus B(M^\gp/P^\gp) \simeq M \oplus_{M \oplus (M^\gp/P^\gp)} M \simeq M \oplus_{(M \oplus_P M)^\rep} M.\] For the second equivalence we use that $B(-)$ is the suspension functor on the category of animated monoids. Hence the result follows from commuting colimits in the diagram 
\[
\begin{tikzcd}[row sep = small] 
A & A \otimes_R A\ar{l} \ar{r} & A \\ 
{\mathbb Z}[M] \ar{u} \ar{d}{=} & {\mathbb Z}[M \oplus_P M] \ar{l} \ar{u} \ar{r} \ar{d} & {\mathbb Z}[M] \ar{u} \ar{d}{=} \\ 
{\mathbb Z}[M] & {\mathbb Z}[(M \oplus_P M)^\rep] \ar{l} \ar{r} & {\mathbb Z}[M]
\end{tikzcd}
\] 
of animated commutative rings: 
Forming the horizontal 
pushouts first gives Rognes' definition, 
while forming the vertical pushouts first gives our formulation.  
\end{proof}

\subsection{Log Hochschild homology of free pre-log algebras} Let $(R_0, P_0)$ be a discrete pre-log ring and let $(A_0, M_0)$ be a free pre-log algebra over $(R_0, P_0)$. Since $(A_0, M_0)$ is cofibrant over $(R_0, P_0)$, no replacement is necessary when defining log Hochschild homology ${\rm logHH}((A_0, M_0) / (R_0, P_0))$. In particular, by virtue of \Cref{prop:rognescomparison2}, we have that ${\rm logHH}((A_0, M_0) / (R_0, P_0))$ is naturally equivalent to the point-set pushout of the diagram \begin{equation}\label{underivedloghh}{\rm HH}(A_0/R_0) \xleftarrow{} {\mathbb Z}[B^{\rm cy}_{P_0}(M_0)] \xrightarrow{} {\mathbb Z}[B^\rep_{P_0}(M_0)]\end{equation} of simplicial commutative rings.
Here $B^{\rm cy}_{P_0}(M_0)$ is the (underived) cyclic bar construction of $M_0$ relative to $P_0$ and $B^\rep_{P_0}(M_0)$ 
is the (discrete) pullback of $M_0 \xrightarrow{} M_0^\gp \xleftarrow{} B^{\rm cy}_{P_0^\gp}(M_0^\gp)$. 

We will consider these (underived) constructions for any map of discrete commutative monoids $P_0 \to M_0$
and apply the following results when $M_0$ is free over $P_0$:

\begin{lemma}\label{lem:discreterepletebar}
Let $P_0\to M_0$ be a map of discrete commutative monoids.
There is an isomorphism of animated commutative monoids
\[M_0 \oplus B(M_0^\gp/P_0^\gp) \xrightarrow{\cong} B^\rep_{P_0}(M_0).\] 
In simplicial degree $q$, 
it is given by 
$(m, [g_1], \dots, [g_q]) \mapsto (m, (\gamma_0(m)(g_1 \cdots g_q)^{-1}, g_1, \dots, g_q)$.
\end{lemma}

\begin{proof} The proposed map is well-defined, as for different representatives of the elements in $M_0^\gp/P_0^\gp$, the inversion used in the map will cancel their difference. An inverse sends a $q$-simplex $(m, (g_0, \dots, g_q))$ to the element $(m, [g_1], \dots, [g_q])$. It is straight-forward to check that this map respects the simplicial structure maps.\qedhere 
\end{proof}

\begin{corollary}\label{cor:lowdegreecomputations} Let $P_0 \to M_0$ be a map of discrete commutative monoids. There are isomorphisms of ${\mathbb Z}[M_0]$-modules 
\[\pi_0{\mathbb Z}[B^\rep_{P_0}(M_0)] \cong {\mathbb Z}[M_0] \quad \text{ and } \quad  \pi_1{\mathbb Z}[B^\rep_{P_0}(M_0)] \cong {\mathbb Z}[M_0] \otimes_{\mathbb Z} (M_0^{\gp}/P_0^{\gp}).\]
\end{corollary}

\begin{proof} By Lemma \ref{lem:discreterepletebar}, the replete bar construction $B^\rep_{P_0}(M_0)$ is levelwise isomorphic to $M_0 \oplus B(M_0^\gp/P_0^\gp)$. Hence the Moore complex associated to ${\mathbb Z}[B^\rep_{P_0}(M_0)]$ takes the form \[\cdots \xrightarrow{} {\mathbb Z}[M_0] \otimes_{\mathbb Z} {\mathbb Z}[M_0^{\gp}/P_0^{\gp}]^{\otimes 2} \xrightarrow{\partial_2} {\mathbb Z}[M_0] \otimes_{\mathbb Z} {\mathbb Z}[M_0^{\gp}/P_0^{\gp}] \xrightarrow{\partial_1 = 0} {\mathbb Z}[M_0] \to 0.\] This immediately gives the desired identification in degree $0$. In degree $1$, we notice that the map \[({\mathbb Z}[M_0] \otimes_{{\mathbb Z}} {{\mathbb Z}[M_0^\gp/P_0^\gp]})/\!\!\sim \xrightarrow{} {\mathbb Z}[M_0] \otimes_{{\mathbb Z}} (M_0^\gp/P_0^\gp)\] sending $[m \otimes g]$ to $m \otimes g$ is a well-defined isomorphism, where $\sim$ is the equivalence relation ${\mathbb Z}[M_0]$-linearly generated by $m \otimes g_2 - m \otimes g_1g_2 + m \otimes g_1$. The source of this map is by definition $H_1$ of the Moore complex of the replete bar construction. 
\end{proof}

The following description of the log Hochschild homology of a free $(R_0, P_0)$-algebra, which follows from Proposition \ref{prop:basechangeHH}, will prove useful:

\begin{lemma}
\label{lem:loghhfree} 
Let $(R_0, P_0)$ be a discrete pre-log ring. 
There is a natural equivalence of animated commutative rings
\[
{\rm logHH}((F_{(R_0, P_0)}(X, Y)) / (R_0, P_0)) \simeq 
R_0 \otimes_{{\mathbb Z}} \HH({\mathbb Z}\langle Y \rangle) 
\otimes_{{\mathbb Z}} {\rm logHH}(({\mathbb Z}\langle X \rangle, \langle X \rangle) / ({\mathbb Z}, \{1\})).\] 
\end{lemma}

\subsection{Recovering the log K\"ahler differentials from log Hochschild homology} 
Let $(f_0, f_0^\flat) \colon (R_0, P_0, \beta_0) \to (A_0, M_0, \alpha_0)$ be a map of discrete pre-log rings. 
Recall the definition of the log K\"ahler differentials from Section \ref{subsec:logkahler}. 
Lemma \ref{lem:loghhfree} gives the following:

\begin{corollary}\label{cor:loghhkahlerfree} Let $F_{(R_0, P_0)}(X, Y)$ be a free pre-log algebra. 
There is an isomorphism of $R_0\langle X \sqcup Y \rangle$-modules
\[\Phi_1 \colon \Omega^1_{F_{(R_0, P_0)}(X, Y) / (R_0, P_0)} \to \pi_1{\rm logHH}((F_{(R_0, P_0)}(X, Y)) / (R_0, P_0)).\] 
\end{corollary}

\begin{proof} Since $F_{(R_0, P_0)}(X, Y)$ is free, we can model log Hochschild homology by the pushout of \eqref{underivedloghh}. Apply Lemma \ref{lem:loghhfree} and consider the convergent ${\rm Tor}$-spectral sequence \[E^2_{p, q} := {\rm Tor}_p^{{\mathbb Z}}(\pi_*\HH^{R_0}(R_0\langle Y \rangle), \pi_*{\mathbb Z}[B^\rep\langle X \rangle])_q \implies \pi_{p + q}{\rm logHH}((F_{(R_0, P_0)}(X, Y)) / (R_0, P_0)).\] By Corollary \ref{cor:lowdegreecomputations}, we obtain $E^2_{p, 0}$ for $p > 0$ and $E^2_{0, 1} \cong \Omega^1_{F_{(R_0, P_0)}(X, Y) / (R_0, P_0)}$. 
\end{proof}

From this, 
we deduce the following more general result:

\begin{proposition}
\label{prop:logkahler} 
Let $(R_0, P_0) \to (A_0, M_0)$ be a map of discrete pre-log rings. 
There is an isomorphism of $A_0$-modules
\[
\Phi_1 
\colon 
\Omega^1_{(A_0, M_0)/(R_0, P_0)} 
\to 
\pi_1\logHH((A_0, M_0) / (R_0, P_0)).
\] 
\end{proposition}
\begin{proof} In view of \Cref{prop:rognescomparison2}, 
we model 
log Hochschild homology by taking a cofibrant replacement $(A_0^{\rm cof}, M_0^{\rm cof}) \xrightarrow{\simeq} (A_0, M_0)$, where $(A_0^{\rm cof}, M_0^{\rm cof})$ is a levelwise free simplicial $(R_0, P_0)$-algebra and applying the 
construction \eqref{underivedloghh} levelwise.
As the resulting object $\logHH((A_0, M_0)/(R_0,P_0))$ is the diagonal of a bisimplicial commutative ring, there is a convergent spectral sequence \[E^2_{p, q} := \pi_p(\pi_q(B^{\rm cy}_{R_0}(A_0^{\rm cof}) \otimes_{{\mathbb Z}[B^{\rm cy}_{P_0}(M_0^{\rm cof})]} {\mathbb Z}[B^\rep_{P_0}(M_0^{\rm cof})])) \implies \pi_{p + q}\logHH((A_0, M_0) / (R_0, P_0)). \] Since $\pi_0{\mathbb Z}[B^{\rm cy}_{P_0}(M_0^{\rm cof})] \cong \pi_0{\mathbb Z}[B^\rep_{P_0}(M_0^{\rm cof})] \cong {\mathbb Z}[M_0^{\rm cof}]$ by Corollary \ref{cor:lowdegreecomputations}, we have \[\pi_0(B^{\rm cy}_{R_0}(A_0^{\rm cof}) \otimes_{{\mathbb Z}[B^{\rm cy}_{P_0}(M_0^{\rm cof})]} {\mathbb Z}[B^\rep_{P_0}(M_0^{\rm cof})]) \cong A_0^{\rm cof}.\] Thus $E_{p, 0} = 0$ for all $p > 0$. Let us compute $E^2_{0, 1}$. Since $(A_0^{\rm cof}, M_0^{\rm cof})$ is levelwise free, Corollary \ref{cor:loghhkahlerfree} applies to obtain \[\pi_1(B^{\rm cy}_{R_0}(A_0^{\rm cof}) \otimes_{{\mathbb Z}[B^{\rm cy}_{P_0}(M_0^{\rm cof})]} {\mathbb Z}[B^\rep_{P_0}(M_0^{\rm cof})] \cong \Omega^1_{(A^{\rm cof}_0, M_0^{\rm cof}) / (R_0, P_0)} \xleftarrow{\simeq} {\mathbb L}_{(A_0, M_0) / (R_0, P_0)}.\] Consequently, $E^2_{0, 1} \cong \pi_0{\mathbb L}_{(A_0, M_0) / (R_0, P_0)} \cong \Omega^1_{(A_0, M_0) / (R_0, P_0)}$ \cite[Lemma 8.9]{Ols05}, as desired. 
\end{proof}

Keeping track of the maps involved in constructing the map $\Phi_1$, we find that it admits the explicit description \[\Phi_1(da) = [1 \otimes a], \qquad \Phi_1(\dlog(m)) = [1 \otimes \gamma_0(m)].\] By abuse of notation we identify $[a \otimes b]$ in $\pi_1\HH(A_0 / R_0)$ and $[1 \otimes \gamma_0(m)]$ in $\pi_1{{\mathbb Z}[B^\rep_{P_0}(M_0)]}$ with their images in $\pi_1\logHH((A_0, M_0) / (R_0, P_0))$.
The map $\Phi_1$ in Proposition \ref{prop:logkahler} and the universal property of the exterior algebra give rise to a map
of graded commutative rings
\begin{equation}
\label{antisymm} 
\Phi_* \colon \Omega^*_{(A_0, M_0) / (R_0, P_0)} 
\to 
\pi_*\logHH((A_0, M_0) / (R_0, P_0)). 
\end{equation}

\subsection{The log ${\rm HKR}$-filtration} 
We are ready to construct a logarithmic version of the ${\rm HKR}$-filtration \cite[Proposition IV.4.1]{NS18}:

\begin{theorem}\label{thm:logaqss} Let $(R, P) \to (A, M)$ be a map of animated pre-log rings. Then the relative log Hochschild homology $\logHH((A, M) / (R, P))$ admits a descending separated filtration with graded pieces of the form $(\Bigwedge_A^n {\mathbb L}_{(A, M)/(R, P)})[n]$. In particular, there is a strongly convergent spectral sequence \[E^2_{p, q} := \pi_p(\Bigwedge\nolimits_A^q {\mathbb L}_{(A, M)/(R, P)}) \implies \pi_{p + q}\logHH((A, M) / (R, P))\] relating the derived wedge powers of the log cotangent complex and log Hochschild homology.
\end{theorem}

Theorem \ref{thm:logaqss} is proven at the end of this section. 
As a corollary, we obtain a version of the ${\rm HKR}$-theorem for log smooth log rings:

\begin{corollary}\label{cor:loghkr} Let $(R_0, P_0) \to (A_0, M_0)$ be a derived formally log smooth map of discrete pre-log rings. Then the map defined in \eqref{antisymm} is an isomorphism of graded commutative rings.
\end{corollary}

\begin{proof} By Proposition \ref{prop:cotangent_vs_omega2}, the spectral sequence of Theorem \ref{thm:logaqss} collapses to give the desired isomorphism.
\end{proof}

We first establish the ${\rm HKR}$-theorem for free pre-log algebras and use this to construct the logarithmic ${\rm HKR}$-filtration.

\subsection{The log ${\rm HKR}$-theorem for free pre-log algebras} The general log ${\rm HKR}$-theorem for free pre-log algebras will reduce to the following:

\begin{lemma}\label{lem:loghkrfree} The map \[\Omega^*_{F_{(R_0, P_0)}(\{x\}, \emptyset)/(R_0, P_0)} \xrightarrow{} \pi_*{\rm logHH}(F_{(R_0, P_0)}(\{x\}, \emptyset) / (R_0, P_0))\] defined in \eqref{antisymm} is an isomorphism of graded rings.
\end{lemma}

Stated differently, the log ${\rm HKR}$-theorem holds for free pre-log algebras $(R_0\langle x \rangle, \langle x \rangle)$ on a single generator.

\begin{proof} By Lemma \ref{lem:discreterepletebar} and Lemma \ref{lem:loghhfree}, there is an equivalence \[{\rm logHH}(F_{(R_0, P_0) / (R_0, P_0)}(\{x\}, \emptyset)) \simeq R_0 \otimes_{{\mathbb Z}} {\rm logHH}({\mathbb Z}[x], \langle x \rangle) \cong R_0[x] \otimes_{{\mathbb Z}} {\mathbb Z}[B \langle x \rangle^\gp].\]  The homotopy groups of ${\mathbb Z}[B \langle x \rangle^\gp]$ are the ${\rm Tor}$-groups ${\rm Tor}^{{\mathbb Z}[x, x^{-1}]}_*({\mathbb Z}, {\mathbb Z})$. 
Indeed, observe that $B\langle x \rangle^\gp$ is the suspension of $\langle x \rangle^\gp$ in the category of animated abelian groups. Hence, since ${\mathbb Z}[-]$ commutes with homotopy pushouts, the animated commutative ring ${\mathbb Z}[B \langle x \rangle^\gp]$ is equivalent to the derived tensor product ${\mathbb Z} \otimes_{{\mathbb Z}[\langle x \rangle^\gp]}^{\mathbb L} {\mathbb Z}$.
Using the free resolution \[0 \to {\mathbb Z}[x, x^{-1}] \xrightarrow{\cdot (x - 1)} {\mathbb Z}[x, x^{-1}] \xrightarrow{} {\mathbb Z} \to 0,\] we find the ${\rm Tor}$-groups to be 
copies of $\bZ$ concentrated in degrees 0 and 1
 
so that the graded ring
$\pi_*{\rm logHH}(F_{(R_0, P_0)}(\{x\}, \emptyset) / (R_0, P_0))$ is concentrated as $R_0[x]$ in degrees $0$ and $1$. By Proposition \ref{prop:logkahler}, there is a preferred isomorphism 
\[\Phi_1 \colon \Omega^1_{F_{(R_0, P_0)}(\langle x \rangle, \emptyset) / (R_0, P_0)} \xrightarrow{\cong} \pi_1{\rm logHH}(F_{(R_0, P_0)}(\langle x \rangle, \emptyset) / (R_0, P_0)).\] Since $\pi_*\logHH(F_{(R_0, P_0)}(\langle x \rangle, \emptyset) / (R_0, P_0))$ is concentrated in degrees $0$ and $1$, this suffices to conclude that the map is an isomorphism.
\end{proof}

By commuting colimits, Lemma \ref{lem:loghhfree} implies the log ${\rm HKR}$-theorem for any free pre-log ring 
$F_{(R, P)}(X,\emptyset)$. Combining this with the ordinary ${\rm HKR}$-theorem, we obtain:

\begin{corollary}\label{cor:loghkrfree} Let $F_{(R_0, P_0)}(X, Y)$ be a free $(R_0, P_0)$-algebra. Then the map \eqref{antisymm} is an isomorphism of strictly commutative graded rings 
\[\Phi_* \colon \Omega^*_{F_{(R_0, P_0)}(X, Y) / (R_0, P_0)} 
\xrightarrow{\cong} \pi_*{\rm logHH}(F_{(R_0, P_0)}(X, Y) / (R_0, P_0)).\]
\end{corollary}

We can now deduce the log Andr\'e--Quillen spectral sequence:

\begin{proof}[Proof of Theorem \ref{thm:logaqss}] On free pre-log algebras, it is clear by Corollary \ref{cor:loghkrfree} that the Whitehead filtration of ${\rm logHH}((A, M) / (R, P))$ gives the result. Since both the log cotangent complex and log Hochschild homology are in general left Kan extended from free pre-log algebras, this also gives the filtration in general. Hence we obtain a conditionally convergent spectral sequence \[E^2_{p, q} := \pi_p(\Bigwedge\nolimits_A^q {\mathbb L}_{(A, M)/(R, P)}) \implies \pi_{p + q}\logHH((A, M) / (R, P))\]  concentrated in the first quadrant. Therefore it is strongly convergent, see e.g.\ \cite[Theorem 7.1]{Boa99} and the subsequent remark.  \end{proof}

\begin{remark}
From a more classical perspective, 
\Cref{cor:loghkrfree} allows us to use
 the results of Quillen \cite{Qui70} to obtain the log HKR-theorem. One key point is that Corollary \ref{cor:loghkrfree} implies the simplicial augmentation ideal of $$(A \otimes_R A)_{(M \oplus_P M)}^\rep \to A$$ is levelwise \emph{quasi-regular} in the sense of \cite[Definition 6.10]{Qui70}. The definition of $\logHH$ as derived self-intersections makes \cite[Theorem 6.3]{Qui70} applicable and the resulting spectral sequence converges. We also observe that the spectral sequence can be obtained from the totalization of a double complex, see e.g., \cite[\S 8]{Qui70}.
\end{remark}

\subsection{Base change for log Hochschild homology} 
The following is a generalization of Weibel--Geller's \cite{WG91} \'etale base change formula for Hochschild homology:

\begin{theorem}\label{prop:loghhbasechange} Let $(R, P)$ be an animated pre-log ring.
A map $(A, M) \to (B, N)$ of pre-log $(R, P)$-algebras is derived formally log \'etale
if and only if the canonical map \[B \otimes_A \logHH((A, M) / (R, P)) \xrightarrow{} \logHH((B, N) / (R, P))\] is an equivalence. 
\end{theorem}

\begin{remark}
If $(R_0, P_0) \to (A_0, M_0)$ is log \'etale and integral, 
then  $\mathbb{L}_{(A_0,M_0) / (R_0,P_0)}$ is contractible by Proposition \ref{prop:cotangent_vs_omega}(2). 
On the other hand, 
\Cref{exm:logetale} shows the base change formula does not hold for a general log \'etale map of discrete commutative rings.
\end{remark}

The next result makes use of the indecomposables functor discussed in Section \ref{subsec:derivedindec}. 

\begin{lemma}
\label{lem:loghhindeccotcx} 
Let $(R, P) \to (A, M)$ be a map of animated pre-log rings. 
There is a natural equivalence of $A$-modules 
\[Q_A(\logHH((A, M) / (R, P))) \simeq {\mathbb L}_{(A, M) / (R, P)}[1].\]
\end{lemma}
\begin{proof} Recall $\logHH((A, M) / (R, P))$ is the suspension of the augmented simplicial commutative $A$-algebra $(A \otimes_R A) \otimes_{{\mathbb Z}[M \oplus_P M]} {\mathbb Z}[(M \oplus_P M)^\rep]$. Since $Q_A$ is left Quillen, we obtain \[Q_A(\logHH((A, M) / (R, P))) \simeq Q_A((A \otimes_R A) \otimes_{{\mathbb Z}[M \oplus_P M]} {\mathbb Z}[(M \oplus_P M)^\rep])[1].\] The result now follows from Proposition \ref{prop:logcotangentreplete}.
\end{proof}

\begin{proof}[Proof of Theorem \ref{prop:loghhbasechange}] 
If the base change map is an equivalence, 
then the induced map of indecomposables is also an equivalence. 
By \Cref{lem:loghhindeccotcx}, 
this holds if and only if there is an equivalence 
\begin{equation}\label{shiftedcotangent}B \otimes_A {\mathbb L}_{(A, M) / (R, P)}[1] \to {\mathbb L}_{(B, N) / (R, P)}[1],\end{equation} implying that the cofiber ${\mathbb L}_{(B, N) / (A, M)}[1]$ is contractible. 
Conversely, if ${\mathbb L}_{(B, N) / (A, M)}$ is contractible, the map \eqref{shiftedcotangent} is an equivalence, and so we obtain the base change formula for log Hochschild homology by comparing the Quillen spectral sequences converging to $B \otimes_R {\rm logHH}((A, M) / (R, P))$ and ${\rm logHH}((B, N) / (R, P))$. 
\end{proof}

\begin{corollary}\label{cor:unitmapeq} Let $(R, P) \to (A, M)$ be a derived formally log \'etale map of derived pre-log rings. Then the unit map $A \to \logHH((A, M) / (R, P))$ is an equivalence. 
\end{corollary}

\begin{proof} This follows from Proposition \ref{prop:basechangeHH} and Theorem \ref{prop:loghhbasechange}.
\end{proof}

\begin{proposition} Let $R \to A$ be a tamely ramified finite extension of discrete valuation rings with perfect residue fields, and let $(f, f^\flat) \colon (R, R - \{0\}) \to (A, A - \{0\})$ be the induced map of log rings.  
Then there is a canonical equivalence \[
A \otimes_R \logHH(R, R -\{0\}) \to \logHH(A, A- \{0\}).
\]
\end{proposition}

\begin{proof}
Owing to \cite[Proposition IV.3.1.19]{Ogu18}, $(f,f^\flat)$ is log \'etale.
Furthermore, 
$(f,f^\flat)$ is integral since $R - \{0\}$ is a valuative monoid \cite[Proposition I.4.6.3(5)]{Ogu18}.
Proposition \ref{prop:cotangent_vs_omega}(2) shows that $(f,f^\flat)$ is derived log \'etale, 
and we conclude by Theorem \ref{prop:loghhbasechange}. 
\end{proof}

\section{A recollection of the dividing Nisnevich topology}
\label{section:dNis}
We want to understand the global properties of logarithmic Hochschild homology and the 
Andr\'e-Quillen spectral sequence. 
To that end, 
we consider the dividing Nisnevich topology on log schemes introduced in \cite{logDM}. 
For the reader's convenience we recall some of its basic properties.
From now on, all schemes and rings are underived.

\subsection{Dividing Nisnevich cd-structure}
See \Cref{subsec:notation} for our notation on schemes and log schemes.
If $S\in \lSch$, 
let $\lSm/S$ be the category of log smooth morphisms $X\to S$ in $\lSch$.
Likewise, if $S\in \Sch$, 
we let $\SmlSm/S$ be the category of log smooth morphisms $X\to S$ in $\lSch$ such that $\ul{X}\to S$ is smooth 
($\ul{X}$ is the underlying scheme of $X$).

\begin{df}
The \emph{Zariski topology on $\lSch$} (resp.\ \emph{strict \'etale topology on $\lSch$}) is the Grothendieck topology generated by strict morphisms $X\to S$ in $\lSch$ such that $\ul{X}\to \ul{S}$ is a Zariski cover (resp.\ a strict \'etale cover).
\end{df}

\begin{df}
A morphism $f\colon X\to S$ in $\lSch$ is called a \emph{dividing cover} if it is a surjective proper log \'etale monomorphism.
For details, we refer to \cite[\S A.11]{logDM}. 
\end{df}

We refer to \cite{Vcdtop} for Voevodsky's notion of cd-structures.
The following cd-structures are defined in \cite{logDM}.

\begin{df}
\label{dZarsheaf.1}
\begin{enumerate}
\item[(i)]
The \emph{strict Nisnevich cd-structure} on $\lSch$ is the collection of cartesian squares
\begin{equation}
\label{dZarsheaf.1.1}
Q := \begin{tikzcd}[row sep = small]
X'\ar[d,"f'"']\ar[r,"g'"]&X\ar[d,"f"]
\\
S'\ar[r,"g"]&S
\end{tikzcd}
\end{equation}
such that the underlying square $\ul{Q}$ of schemes is a Nisnevich distinguished square and $f$, $f'$, $g$, 
and $g'$ are strict maps.
\item[(ii)]
The \emph{dividing cd-structure} on $\lSch$ is the collection of cartesian squares of the form \eqref{dZarsheaf.1.1} such that $X'=S'=\emptyset$ and $f$ is a dividing cover.
\item[(iii)]
The \emph{dividing Nisnevich cd-structure} on $\lSch$ is the union of the strict Nisnevich and dividing cd-structures.
\end{enumerate}
 
The topology associated with the strict Nisnevich (resp.\ dividing, dividing Nisnevich) cd-structure is called the \emph{strict Nisnevich} (resp.\ \emph{dividing}, \emph{dividing Nisnevich}) topology.
\end{df}

\subsection{Comparison of sheaves}
Let us review the comparison lemma \cite[Th\'eor\`eme III.4.1]{SGA4} and its $\infty$-categorical generalization.
As usual, $\Delta$ denotes the simplex category.
We write $\infSpc$ for the $\infty$-category of spaces, 
and let $\infDAb$ denote the derived $\infty$-category of abelian groups $\mathrm{Ab}$.

\begin{definition}
\label{dZarsheaf.8}
Suppose $\cV$ is an $\infty$-category admitting colimits and $\cC$ is an $\infty$-category.
Let $\infPsh(\cC,\cV)$ be the $\infty$-category $\Fun(\cC^{op},\cV)$.

Suppose $\sX$ is a simplicial object in $\cC$.
For $\cF\in \infPsh(\cC,\cV)$, we define
\[
\lvert\cF(\sX)\rvert
:=
\colimit_{n\in \Delta}\cF(\sX_n).
\]
If $t$ is a topology on $\cC$, 
let $\infShv_t(\cC,\cV)$ be the $\infty$-category of hypercomplete sheaves, 
i.e., 
the full subcategory of $\infPsh(\cC,\cV)$ consisting of $\cF$ such that the naturally induced map
\[
\cF(X)
\to
\lvert\cF(\sX)\rvert
\]
is an equivalence for every $t$-hypercover $\sX\to X$.
\end{definition}

\begin{example}
When $\cC$ is a $1$-category,
the $\infty$-category $\infShv_t(\cC,\infSpc)$ is equivalent to the underlying $\infty$-category of 
Jardine's $t$-local projective model structure on $\sPsh(\cC)$ according to \cite[Proposition 6.5.2.14]{HTT}, 
and likewise for $\infShv_t(\cC,\infDAb)$ and chain complexes of presheaves of abelian groups 
$\mathrm{Ch}(\Psh(\cC,\bZ))$).
\end{example}

Suppose $\alpha \colon \cC\to \cD$ is a fully faithful functor of categories and $t$ is a topology on $\cD$.
Consider the topology on $\cC$ induced by $t$.
Owing to \cite[Proposition III.1.2]{SGA4} the functor
\[
\alpha^*\colon \Shv_t(\cD)\to \Shv_t(\cC)
\]
sending $\cF\in \Shv_t(\cD)$ to $(X\in \cC)\mapsto \cF(\alpha(X))$ admits a left adjoint
\[
\alpha_!^t \colon \Shv_t(\cC)\to \Shv_t(\cD).
\]
Similarly, 
if $\alpha$ is instead a fully faithful functor of $\infty$-categories and $t$ is a topology on $\cD$,
we write $\alpha_!^t$ for the left adjoints of the functors
\[
\alpha^*\colon \infShv_t(\cD,\infSpc)\to \infShv_t(\cC,\infSpc)
\text{ and }
\alpha^*\colon \infShv_t(\cD,\infDAb)\to \infShv_t(\cC,\infDAb).
\]

\begin{proposition}
\label{dZarsheaf.7}
Suppose every object of $\cD$ admits a $t$-cover in $\cC$.
Then there is an equivalence
\[
\alpha^*\colon \Shv_t(\cD)\xrightarrow{\simeq} \Shv_t(\cC)
\]
in the case that $\alpha$ is a functor of $1$-categories and there are equivalences
\begin{align*}
\alpha^*\colon& \infShv_t(\cD,\infSpc)\xrightarrow{\simeq} \infShv_t(\cC,\infSpc),
\\
\alpha^*\colon& \infShv_t(\cD,\infDAb)\xrightarrow{\simeq} \infShv_t(\cC,\infDAb).
\end{align*}
in the case that $\alpha$ is a functor of $\infty$-categories.
\end{proposition}
\begin{proof}
The first $\alpha^*$ is an equivalence by the implication (i)$\Rightarrow$(ii) in \cite[Th\'eor\`eme III.4.1]{SGA4}.
The second one is an equivalence as a special case of \cite[Proposition 6.22]{PY16}.
The proof of \emph{loc.\ cit.} can be adapted to show that the third one is an equivalence too.
\end{proof}

\begin{df}
For a fan $\Sigma$, let $\bT_\Sigma$ be the fs log scheme $Z$ whose underlying scheme is the toric variety associated with $\Sigma$ and whose log structure is the compactifying log structure associated with the torus embedding.

For an fs monoid $P$ such that $P^\gp$ is torsion-free, 
let $\Spec{P}$ denote the associated fan whose only maximal cone is the dual monoid of $P$.

Let $\lFan$ be the full subcategory of $\lSch$ consisting of disjoint unions $\amalg_{i\in I}X_i$ such that each $X_i$ admits a strict morphism $X\to \bT_{\Sigma_i}$ with a fan $\Sigma_i$.
Let $\eta\colon \lFan\to \lSch$ be the inclusion functor.
\end{df}

\begin{lemma}
\label{dZarsheaf.3}
If $S\in \lSch$, 
there is a naturally induced equivalence
\[
\eta^*
\colon
\Shv_{sNis}(\lSch/S) \xrightarrow{\simeq} \Shv_{sNis}(\lFan/S),
\]
Similar equivalences hold for $\infShv(-,\infSpc)$ and $\infShv(-,\infDAb)$.
\end{lemma}
\begin{proof}
Every object of $\lSch/S$ admits a strict Nisnevich cover in $\lFan/S$ according to \cite[Proposition II.2.3.7]{Ogu18}. 
Proposition \ref{dZarsheaf.7} finishes the proof.
\end{proof}

\subsection{Dividing Nisnevich descent property}

We refer to \cite[Definitions 3.3.22, 3.4.2]{logDM} and \cite{Vcdtop} for the definitions of quasi-bounded, squareable, bounded, complete, and regular cd-structures.

\begin{proposition}
\label{dZarsheaf.10}
Suppose $P$ is a quasi-bounded, regular, and squareable cd-structure on a category $\cC$ with an initial object.
Let $t_P$ be the topology on $\cC$ associated with $P$.
Then $\cF\in \infPsh(\cC,\infSpc)$ satisfies $t_P$-descent if and only if $\cF(Q)$ is cocartesian for every $P$-distinguished square $Q$.
\end{proposition}
\begin{proof}
This is shown in \cite[Proposition 3.8(2), (3)]{Vcdtop} if $P$ is a bounded, complete, and regular cd-structure.
The same proof works in over setting by replacing \cite[Lemma 3.5]{Vcdtop} with \cite[Lemma 3.4.10]{logDM}.
\end{proof}

\begin{lemma}
\label{dZarsheaf.9}
Suppose $\cV$ is an $\infty$-category admitting colimits and $\cF\in \infPsh(\lSch/S,\cV)$, where $S\in \lSch$.
Then $\cF$ satisfies dividing descent if and only if $\cF(X)\to \cF(Y)$ is an equivalence for every dividing cover $Y\to X$ in $\lSch/S$.
The same result holds for $\lFan/S$.
\end{lemma}
\begin{proof}
We focus on $\lSch/S$ since the proofs are similar.
Suppose $\cF$ satisfies dividing descent.
For a dividing cover $f\colon Y\to X$ in $\lSch/S$, the \v{C}ech nerve associated with $f$ is isomorphic to the constant simplicial scheme $Y$.
The above shows that $\cF(X)\to \cF(Y)$ is an equivalence.
Conversely, suppose $\cF(X)\to \cF(Y)$ is an equivalence for every dividing cover $Y\to X$ in $\lSch/S$.
If $\sX\to X$ is a dividing hypercover, then $\sX_i\to X$ is a dividing cover for every integer $i\geq 0$.
Hence $\cF(X)\to \lvert\cF(\sX)\rvert$ is an equivalence, i.e., $\cF$ satisfies dividing descent.
\end{proof}

\begin{proposition}
\label{dZarsheaf.4}
Suppose $\cF\in \infPsh(\lSch/S,\cV)$, where $S\in \lSch$, and $\cV=\infSpc$ or $\infDAb$.
Then $\cF$ satisfies dividing Nisnevich descent if and only if $\cF$ satisfies both strict Nisnevich descent and dividing descent.
\end{proposition}
\begin{proof}
Due to Proposition \ref{dZarsheaf.10} and \cite[Proposition 3.3.30]{logDM}, $\cF$ satisfies $t$-descent if and only if $\cF(Q)$ is cocartesian for every $t$-distinguished square $Q$ for $t=sNis,div,dNis$ when $\cV=\infSpc$.
The same holds by \cite[Proposition 3.4.11]{logDM} when $\cV=\infDAb$.
This implies the claim since the union of the strict Nisnevich cd-structure and the dividing cd-structure is the dividing Nisnevich cd-structure.
\end{proof}

\subsection{Dividing invariance of Hodge cohomology}

Owing to \cite[(9.1.1)]{logDM}, 
for all $X\in \Sch$, quasi-coherent sheaf $\cF$ on $X$, and integer $i$, 
we have isomorphisms
\begin{equation}
\label{cot.0.1}
H_{Zar}^i(X,\cF)
\cong
H_{sNis}^i(X,\cF)
\cong
H_{s\et}^i(X,\cF).
\end{equation}
If $X\in \lSch$, 
$t=Zar,sNis,s\et$, 
we have a canonical isomorphism
\[
H_t^i(\ul{X},\cF)
\cong
H_t^i(X,\cF).
\]
In fact, 
cohomology groups of quasi-coherent modules are invariant under dividing covers.

\begin{theorem}
\label{cot.5}
Let $f\colon Y\to X$ be a dividing cover in $\lSch$.
Then the functor
\[
\ul{f}^*\colon \Deri^b(\Coh(\ul{X}))\to \Deri^b(\Coh(\ul{Y}))
\]
of the bounded derived categories of coherent sheaves naturally induced by the underlying morphism $\ul{f}\colon \ul{Y}\to \ul{X}$ is fully faithful.
In particular, the naturally induced homomorphism
\begin{equation}
\label{cot.4.2}
\mathbf{H}_{Zar}^i(\ul{X},\cF)
\to
\mathbf{H}_{Zar}^i(\ul{Y},f^*\cF)
\end{equation}
is an isomorphism for all $\cF\in \Deri^b(\Coh(\ul{X}))$ and integers $i$.
\end{theorem}

We prove Theorem \ref{cot.5} at the end of this section.
It improves on \cite[Proposition 9.2.4]{logDM}, 
where we assumed that $X$ and $Y$ are log smooth over a field.

\begin{corollary}
\label{cot.13}
Suppose $S\in \lSch$.
Then
\[
R\Gamma_{sNis}(-,\Omega_{-/S}^d)\in \infShv_{sNis}(\lSch/S,\infDAb)
\]
satisfies dividing Nisnevich descent.
\end{corollary}
\begin{proof}
Owing to \eqref{cot.0.1}, it suffices to show that the naturally induced homomorphism
\[
f^*
\colon
H_{Zar}^i(X,\Omega_{X/S}^d)
\to
H_{Zar}^i(Y,\Omega_{Y/S}^d)
\]
is an isomorphism for all dividing covers $Y\to X$ in $\lSch/S$ and integers $d,i\geq 0$.

We may assume $d>0$ due to Theorem \ref{cot.5}.
Corollary IV.3.2.4(1) in \cite{Ogu18} shows there is an isomorphism $f^*\Omega_{X/S}^1 \cong \Omega_{Y/S}^1$.
Hence we have $f^*\Omega_{X/S}^d\cong \Omega_{Y/S}^d$.
To conclude, 
we apply Theorem \ref{cot.5} to see that the homomorphism
\[
\Hom_{ \Deri^b(\Coh(\ul{X}))}(\cO_X,\Omega_{X/S}^d[i])
\to
\Hom_{ \Deri^b(\Coh(\ul{Y}))}(\cO_Y,\Omega_{Y/S}^d[i])
\]
is an isomorphism for all integers $i$.
\end{proof}

A fan $\Sigma$ has a lattice $N_{\Sigma}$ and a dual lattice $M_{\Sigma}$.
We let $N_{\Sigma,\bR}$ be shorthand for the tensor product $N_{\Sigma}\otimes \bR$.
Let $\lvert \Sigma\rvert$ denote the support of $\Sigma$, which is a closed subset of $N_{\Sigma,\bR}$.

\begin{proposition}
\label{cot.6}
Let $A$ be a commutative ring.
For a fan $\Sigma$ we set $X:=\Spec{A}\times \bT_{\Sigma}$ and $Z_m:=\{x\in \Sigma:\langle m,x\rangle\geq 0\}$ for $m\in M_\Sigma$.
If $i\geq 0$, there is a canonical isomorphism
\[
H_{Zar}^i(X,\cO_X)
\cong
\bigoplus_{m\in M_{\Sigma}} H_{Z_m}^i(\lvert\Sigma\rvert,A).
\]
\end{proposition}
\begin{proof}
If $A$ is a field, 
this is a special case of \cite[Theorem 7.2]{Dan78}.
We argue similarly for a general commutative ring $A$.
First, consider the case when $\Sigma$ consists of a single cone $\sigma$ with dual monoid $P$.
Then there is an isomorphism $X\simeq \Spec{A}\times \bA_P$.
Since $\ul{X}$ is affine, 
we have the vanishing $H_{Zar}^i(X,\cO_X)\cong 0$ for $i>0$ and also $H_{Zar}^0(X,\cO_X)\cong A[P]$.
On the other hand,
there is a long exact sequence
\[
\cdots
\to
H^{i-1}(\lvert\Sigma\rvert-Z_m,A)
\to
H^i_{Z_m}(\lvert\Sigma\rvert,A)
\to
H^i(\lvert\Sigma\rvert,A)
\to
H^i(\lvert\Sigma\rvert-Z_m,A)
\to
\cdots.
\]
Since $|\Sigma|$ is convex, it is contractible.
Furthermore,
$|\Sigma|-Z_m$ is contractible (resp.\ empty) if and only if $m\in P$ (resp.\ $m\notin P$).
When combined with the above, 
we deduce
\[
H_{Z_m}^i(\lvert\Sigma\rvert,A)
\cong
\left\{
\begin{array}{ll}
A & \text{if }i=0\text{ and }m\in P,
\\
0 & \text{otherwise}.
\end{array}
\right.
\]

For general $\Sigma$,
let $\sigma_1,\ldots,\sigma_n$ be the maximal cones of $\Sigma$,
and let $P_1,\ldots,P_n$ be their dual cones.
We set $U_i:=\Spec{A}\times \bA_{P_i}$ for simplicity of notation.
The family $\{U_1,\ldots,U_n\}$ is a Zariski covering of $X$,
and the family $\{\lvert\sigma_1\rvert,\ldots,\lvert\sigma_n\rvert\}$ is a closed covering of $\lvert\Sigma\rvert$.
For all integers $1\leq i_1<\ldots<i_r\leq n$,
the intersection of $\sigma_{i_1},\ldots,\sigma_{i_r}$ is a cone.
What we have established above for the case of cones produces an isomorphism of \v{C}ech complexes
\[
\begin{tikzcd}[column sep=tiny, row sep=small]
\cdots\ar[r]&
\bigoplus_{1\leq i,j\leq n}\Gamma(U_i\cap U_j,\cO_{U_i\cap U_j})\ar[r]\ar[d,"\cong"']&
\bigoplus_{1\leq i\leq n}\Gamma(U_i,\cO_{U_i})\ar[d,"\cong"]
\\
\cdots\ar[r]&
\bigoplus_{1\leq i,j\leq n}\bigoplus_{m\in M_\Sigma}\Gamma_{Z_m\cap \lvert\sigma_i\cap \sigma_j\rvert}(|\sigma_i\cap \sigma_j|,A)\ar[r]&
\bigoplus_{1\leq i\leq n}\bigoplus_{m\in M_\Sigma}\Gamma_{Z_m\cap \lvert\sigma_i\rvert}(\lvert\sigma_i\rvert,A).
\end{tikzcd}
\]
Taking the $i$-th cohomology groups of the horizontal rows finishes the proof.
\end{proof}

\begin{proposition}
\label{cot.7}
Suppose $S$ is a scheme and $\theta\colon \Sigma'\to \Sigma$ is a subdivision of fans.
Let $f\colon S\times \bT_{\Sigma'}\to S\times \bT_{\Sigma}$ be the naturally induced morphism.
Then the homomorphism
\[
f^*
\colon
H_{Zar}^i(S\times \bT_{\Sigma},\cO_{S\times \bT_{\Sigma}})
\to
H_{Zar}^i(S\times \bT_{\Sigma'},\cO_{S\times \bT_{\Sigma'}})
\]
is an isomorphism for all integers $i$.
\end{proposition}
\begin{proof}
The question is Zariski local on $S$ so we may assume it is affine.
Proposition \ref{cot.6} finishes the proof since $\lvert\Sigma\rvert=\lvert\Sigma'\rvert$
\end{proof}

\begin{proposition}
\label{cot.8}
There is a canonical equivalence
\[
\cO_{S\times \bT_{\Sigma}} \xrightarrow{ad} R_{Zar} f_*f^*\cO_{S\times \bT_{\Sigma}}.
\]
\end{proposition}
\begin{proof}
We may assume $S$ is affine since the assertion is Zariski local.
Let $\sigma$ be a cone of $\Sigma$, and let $\Delta'$ be the subfan of $\Sigma'$ consisting of cones $\sigma'\in \Sigma'$ such that $\theta(\sigma')\subset \sigma$.
For all integers $i$, the restriction
\[
R_{Zar}^i f_*f^*\cO_{S\times \bT_{\Sigma}}|_{S\times \bT_{\sigma}}
\]
is the sheaf associated to the module $H_{Zar}^i(S\times \bT_{\Delta'},\cO_{S\times \bT_{\Delta'}})$.
Since the $\bT_\sigma$'s, for all cones $\sigma$ of $\Sigma$, form a Zariski cover of $\bT_{\Sigma}$, 
we only need to show there is a naturally induced isomorphism
\[
H_{Zar}^i(S\times \bT_{\sigma},\cO_{S\times \bT_{\sigma}})
\to
H_{Zar}^i(S\times \bT_{\Delta'},\cO_{S\times \bT_{\Delta'}}).
\]
This claim is a special case of Proposition \ref{cot.7}.
\end{proof}

\begin{proof}[Proof of Theorem \ref{cot.5}]
The question is Zariski local on $X$.
We may assume $X$ has an fs chart $P$ with $P$ sharp, 
see \cite[Proposition II.2.3.7]{Ogu18}.
By \cite[Proposition A.11.5]{logDM}, 
there is a toric subdivision $\Sigma\to \Spec{P}$ and a naturally induced isomorphism 
\[
Y\times_{\bA_P}\bT_{\Sigma}
\xrightarrow{\cong}
X\times_{\bA_P}\bT_{\Sigma}.
\]
Due to \cite[Proposition A.11.3]{logDM}, there is a finite Zariski cover $\{U_i\}_{i\in I}$ of $Y$ such that each $U_i\to X$ admits a chart $P\to Q_i$ for which $P^\gp\to Q_i^\gp$ is an isomorphism.
Then $\Delta_i:=\Sigma\times_{\Spec{P}}\Spec{Q_i}$ is a subdivision of $\Spec{Q_i}$.
We need to show the claim for the projections $U_i\times_{\bA_{Q_i}}\bT_{\Delta_i}\to U_i$ and $X\times_{\bA_P}\bT_\Sigma\to X$.
These two cases have only notational differences, 
so we focus on the second case.
In the said case,
there is a cartesian square of fs log schemes
\[
Q
:=
\begin{tikzcd}[row sep = small]
Y\ar[d,"f"']\ar[r,"i'"]&
\ul{X}\times \bT_\Sigma\ar[d,"g"]
\\
X\ar[r,"i"]&
\ul{X}\times \bA_P.
\end{tikzcd}
\]
Since $i$ is strict,
$Q$ is cartesian in the category of log schemes.
It follows that the underlying square of schemes $\ul{Q}$ is cartesian too due to \cite[Proposition III.2.1.2]{Ogu18}.

Let us omit $R_{Zar}$ in the following for simplicity of notation.
We want to show the unit $\id \xrightarrow{ad} \ul{f}_*\ul{f}^*$ is an equivalence.
Since $\ul{i}_*$ is conservative, 
it suffices to show that the natural transformation 
$\ul{i}_*\xrightarrow{ad}\ul{i}_*\ul{f}_*\ul{f}^*$ is an equivalence.
There is a commutative diagram
\[
\begin{tikzcd}[row sep = small]
\ul{i}_*\ar[d,"ad"']\ar[rr,"\id"]&
&
\ul{i}_*\ar[d,"ad"]
\\
\ul{g}_*\ul{g}^*\ul{i}_*\ar[r,"\simeq"]&
\ul{g}_*\ul{i}_*'\ul{f}^*\ar[r,"\simeq"]&
\ul{i}_*\ul{f}_*\ul{f}^*,
\end{tikzcd}
\]
where $\ul{g}_*\ul{g}^*\ul{i}_*\xrightarrow{\simeq} \ul{g}_*\ul{i}_*'\ul{f}^*$ 
by proper base change.
Hence it suffices to show the left-hand vertical map is an equivalence.
The projection formula yields an equivalence
\[
\ul{g}_*\ul{g}^* \cO_{\ul{X}\times \ul{\bA_P}} \otimes \cE
\xrightarrow{\simeq}
\ul{g}_*\ul{g}^*\cE
\]
for all $\cE\in \Deri^b(\Coh(\ul{X}\times \ul{\bA_P}))$.
This concludes the proof since the unit 
\[
\cO_{\ul{X}\times \ul{\bA_P}}
\xrightarrow{ad}
\ul{g}_*\ul{g}^*\cO_{\ul{X}\times \ul{\bA_P}}
\]
is an equivalence; see Proposition \ref{cot.8}.
\end{proof}

\section{Sheaves on arrow categories}
\label{section:soac}

\subsection{Topologies on arrow categories}
One aspect of the dividing Nisnevich topology is that it involves non-affine schemes.
Hence to ``sheafify'' $\logHH$ with respect to the dividing Nisnevich topology, 
we need to define $\logHH(X/S)$ for arbitrary morphism of fs log schemes $X\to S$.
For this purpose, 
we first view $\logHH$ as a presheaf on an arrow category.
In this setup, 
we can sheafify with respect to the dividing Nisnevich topology.

\begin{df}
\label{arrow.1}
The \emph{arrow $\infty$-category of an $\infty$-category $\cC$} is defined to be $\Fun(\Delta^1,\cC)$.
Informally speaking,
the objects are the morphisms $(X\to S)$ in $\cC$,
and the morphisms $(X'\to S')\to (X\to S)$ are the commutative diagrams
\[
\begin{tikzcd}[row sep = small]
X'\ar[d]\ar[r]&S'\ar[d]
\\
X\ar[r]&S.
\end{tikzcd}
\]
If fiber products are representable in $\cC$, so are in $\Fun(\Delta^1,\cC)$.
Indeed, if $(X'\to S')\to (X\to S)$ and $(X''\to S'')\to (X\to S)$ are two morphisms in $\Fun(\Delta^1,\cC)$, then the fiber product is $(X'\times_X X''\to S'\times_S S'')$.
\end{df}

\begin{df}
\label{arrow.2}
Suppose $\cC$ is an $\infty$-category with a topology $t$.
A family of morphisms $\{(X_i\to S_i)\to (X\to S)\}_{i\in I}$ is called a \emph{$t$-covering sieve} if both families
\[
\{X_i\to X\}_{i\in I}
\text{ and }
\{S_i\to S\}_{i\in I}
\]
are $t$-covering sieve.
One can readily check that the set of $t$-covering sieves on $\Fun(\Delta^1,\cC)$ is a topology too.
\end{df}

\subsection{\texorpdfstring{$\logHH$}{logHH} as a presheaf on an arrow \texorpdfstring{$\infty$-}{infinity }category}

\begin{df}
An \emph{fs semi-log ring} is a pre-log ring $(A,M)$ such that $\ol{M}:=M/M^*$ is fine and saturated and the naturally induced map
\[
M\to \Gamma(\Spec{A,M},\cM_{\Spec{A,M}})
\]
is an isomorphism, where $\Spec{A,M}$ is the spectrum of $(A,M)$ \cite[Definition III.1.2.3]{Ogu18}.
Let $\SemiLog$ denote the full category of pre-log rings consisting of fs semi-log rings $(A,M)$ such that $A$ is noetherian and has a finite Krull dimension.
\end{df}

\begin{rmk}
There are at least three possible variants of affine log schemes. 
First, 
one may define an affine log scheme to be a log scheme $X=(\ul{X}, \mathcal{M}_X)$ such that 
$\ul{X}=\Spec{R}$ is an affine scheme. 
Second, 
following \cite[1.1]{Bei13}, 
one can add the condition that the sheaf $\ol{\mathcal{M}}_X$ is generated by global sections, 
i.e., 
$\Gamma({X}, \mathcal{M}_X) \to \ol{\mathcal{M}}_{X,x}$ is surjective for every point $x$. 
Lastly, 
following \cite{Kos21}, 
one can ask that the identity map on $\Gamma(X, \mathcal{M}_X)$ is a chart for $\mathcal{M}_X$.
The third one is taken as our definition of an affine log scheme.
\end{rmk}

\begin{prop}
The functor
$
\mathrm{Spec}
\colon
\SemiLog^{op}
\to
\lSch
$
is fully faithful.
\end{prop}
\begin{proof}
Let $(A,M)$ and $(R,P)$ be fs semi-log rings.
We have an isomorphism
\[
\Hom_{\SemiLog}((R,P),(A,M))
\cong
\Hom_{\CRing}(R,A)\times_{\Hom_{\CMon}(P,A)}\Hom_{\CMon}(P,M),
\]
where $\CMon$ (resp.\ $\CRing$) denotes the category of commutative monoids (resp.\ commutative rings).

We set $X:=\Spec{A,M}$ and $S:=\Spec{R,P}$ for simplicity of notation.
We have isomorphisms
\begin{align*}
\Hom_{\lSch}(X,S)
\cong &
\Hom_{\lSch}(X,\ul{S}\times_{\ul{\bA_P}}\bA_P)
\\
\cong &
\Hom_{\lSch}(X,\ul{S})\times_{\Hom_{\lSch}(X,\ul{\bA_P})}\Hom_{\lSch}(X,\ul{\bA_P})
\\
\cong &
\Hom_{\Sch}(\ul{X},\ul{S})\times_{\Hom_{\CMon}(P,\Gamma(X,\cO_X))}\Hom_{\CMon}(P,\Gamma(X,\cM_X)),
\end{align*}
where we use \cite[Proposition III.1.2.4]{Ogu18} for the last isomorphism.
Together with the isomorphisms $\Gamma(X,\cO_X)\cong A$ and $\Gamma(X,\cM_X)\cong M$, we finish the proof.
\end{proof}

\begin{df}
\label{arrow.3}
Let $\lAff$ be the full subcategory of $\lSch$ consisting of the spectra of fs semi-log rings.
An important fact about $\lAff$ that we will frequently use is that every $X\in \lSch$ admits a Zariski covering $\{U_i\to X\}_{i\in I}$ such that $U_i\in \lAff$ for all $i\in I$, see \cite[Lemmas 1.3.2, 1.3.3]{Tsu99}.

The assignment
\[
(\Spec{A,M}\to \Spec{R,P}) \mapsto \logHH((A,M) / (R, P))
\]
gives a presheaf $\logHH\in \infPsh(\Fun(\Delta^1,\lAff),\SCRing)$.
By composing with the canonical functor $\SCRing\to \infDAb$, we also view $\logHH$ as an object of $\infPsh(\Fun(\Delta^1,\lAff),\infDAb)$.
\end{df}

\begin{lemma}
\label{arrow.4}
Let $(R,P)\xrightarrow{f} (A,M)\xrightarrow{g} (B,N)$ be homomorphisms of fs semi-log rings.
If $f$ is integral log \'etale, then there is an equivalence
\[
\logHH((B,N) / (R, P))
\xrightarrow{\simeq}
\logHH((B,N) / (A, M)).
\]
\end{lemma}
\begin{proof}
According to Proposition \ref{prop:basechangeHH}, we have equivalences
\[
\logHH((B,N) / (A, M))
\simeq
A\otimes_{\logHH(A,M)}\logHH(B,N),\]
\[\logHH((B,N) / (R, P))
\simeq
R\otimes_{\logHH(R,P)}\logHH(B,N).
\]
Together with Corollary \ref{cor:unitmapeq}, we finish the proof.
\end{proof}

\begin{lemma}
\label{arrow.7}
Let $(R,P)\xrightarrow{f} (A,M)\xrightarrow{g} (B,N)$ be homomorphisms of fs semi-log rings.
If $g$ is integral log \'etale, then there is an equivalence
\[
B\otimes_A \logHH((A,M) / (R, P))
\xrightarrow{\simeq}
\logHH((B,N) / (R, P)).
\]
\end{lemma}

\begin{proof} This follows from Theorem \ref{prop:loghhbasechange}.
\end{proof}

A morphism $f\colon X\to S$ is called an \emph{integral log \'etale cover} if $f$ is integral, log \'etale, 
and universally surjective in the category $\lSch$.
This defines an \emph{integral log \'etale topology}.
Every integral log \'etale morphism is Kummer by \cite[Corollary I.4.3.10, Proposition IV.3.4.1]{Ogu18} 
(see also \cite[p.\ 344]{Ogu18}).
Hence the integral log \'etale topology is coarser than the Kummer log \'etale topology.

For a morphism of affine fs log schemes $X:=\Spec{A,M}\to S:=\Spec{R,P}$, we set
\[
\logHH(X/S)
:=
\logHH((A,M) / (R, P)).
\]

\begin{proposition}
\label{arrow.5}
The presheaf $\logHH\in \infPsh(\Fun(\Delta^1,\lAff),\infDAb)$ satisfies integral log \'etale descent. 
\end{proposition}
\begin{proof}
Suppose $(X'\to S')\to (X\to S)$ is an integral log \'etale cover in $\Fun(\Delta^1,\lAff)$.
By Lemma \ref{arrow.4}, we have an equivalence
\begin{equation}
\label{arrow.5.1}
\logHH(X'/S')
\simeq
\logHH(X'/S).
\end{equation}

If $(\mathscr{X}\to \mathscr{S})\to (X\to S)$ is an integral log \'etale hypercover in $\Fun(\Delta^1,\lAff)$, 
we need to show that the canonical map
\[
\logHH(X/S)
\to
\lim_{i\in \Delta} \logHH(\mathscr{X}_i/ {\mathscr{S}_i})
\]
is an equivalence.
Thanks to \eqref{arrow.5.1}, 
we are reduced to considering the canonical map
\[
\logHH(X/S)
\to
\lim_{i\in \Delta} \logHH(\mathscr{X}_i/S).
\]

Suppose $S$ (resp.\ $X$, $\sX_i$) is given by $\Spec{A,M}$ (resp.\ $\Spec{B,N}$, $\Spec{B_i,N_i}$).
The underlying morphism of schemes of an integral log \'etale morphism is flat by \cite[Theorem IV.4.3.5]{Ogu18}.
Hence,
by faithfully flat descent, 
we obtain the exact sequence of $A$-modules
\[
0\to B \to B_0 \to B_1\to \cdots.
\]
To conclude we apply $(-) \otimes_A \logHH(X/S)$ and appeal to Lemma \ref{arrow.7}.
\end{proof}

\begin{remark}
An anonymous referee suggested  to consider a topology on the $\infty$-category ${\rm Ani(PreLog)}^{\rm op}$ such that $\logHH$ is a hypersheaf with this topology.
This can be achieved as follows.

Consider the topology on ${\rm Ani(PreLog)}^{\rm op}$ generated by the family of derived formally log \'etale maps $\{f_i\colon (A,M)\to (B_i,N_i)\}_{i\in I}$ with finite $I$ such that $A\to \prod_{i\in I} B_i$ is faithfully flat in the sense of \cite[Definition B.6.1.1]{SAG}.
The faithfully flat descent theorem holds for simplicial commutative rings, see \cite[Corollary D.6.3.4]{SAG}.
Hence using Proposition \ref{prop:loghhbasechange},
we can argue as in the proof of Proposition \ref{arrow.5} to show that the presheaf $\logHH\in \infPsh(\Fun(\Delta^1, {\rm Ani(PreLog)}^{\rm op}),\infDAb)$ is a hypersheaf for the induced topology on $\Fun(\Delta^1,{\rm Ani(PreLog)}^{\rm op})$. Note that the condition that the underlying map of rings $A\to \prod_{i\in I} B_i$ is faithfully flat is fairly strong: for example, log blow-ups are formally log \'etale but clearly the underlying map of (classical) schemes is not faithfully flat.

This is the class of morphism we would be interested in considering for the dividing topology. It is unfortunately unclear how to generalize the dividing Nisnevich topology to derived setting of ${\rm Ani(PreLog)}^{\rm op}$.
The main technical advantage of the dividing topology is that we can invert log blow-ups by imposing dividing descent since a log blow-up is a monomorphism in the category of \emph{saturated} log schemes.
In the category of log schemes,
a nontrivial log blow-up is no longer a monomorphism.
Hence working within the category of saturated log schemes is essential to take  advantage of the dividing topology.
However, it is not clear what is the useful definition of a \emph{saturated animated} monoid. 

This is the reason why we have to stick to discrete fs log schemes when we need to use the dividing Nisnevich topology later.
\end{remark}

\begin{proposition}
\label{arrow.6}
There is an equivalence
\[
\Shv_{sNis}(\Fun(\Delta^1,\lAff))
\xrightarrow{\simeq}
\Shv_{sNis}(\Fun(\Delta^1,\lSch)).
\]
Similar equivalences hold for $\infShv(-,\infSpc)$ and $\infShv(-,\infDAb)$.
\end{proposition}
\begin{proof}
Suppose $(X\to S)$ is an object of $\Fun(\Delta^1,\lSch)$.
Choose a strict Nisnevich cover $S'\to S$ such that $S'\in \lAff$, and choose a strict Nisnevich cover $X'\to X\times_S S'$ such that $X'\in \lAff$.
Then the morphism $(X'\to S')\to (X\to S)$ is a strict Nisnevich cover.
Hence every object of $\Fun(\Delta^1,\lSch)$ has a strict Nisnevich cover in $\Fun(\Delta^1,\lSch)$.
Proposition \ref{dZarsheaf.7} finishes the proof.
\end{proof}

\begin{definition}
For $\cV=\infSpc$ or $\infDAb$, 
the inclusion functor $\infShv_t(\cC,\cV)\to \infPsh(\cC,\cV)$ admits a left adjoint
\[
L_t
\colon 
\infPsh(\cC,\cV) \to \infShv_t(\cC,\cV).
\]
This is called the \emph{$t$-localization functor}.
We say that $\cF\in \infPsh(\cC,\cV)$ satisfies \emph{$t$-descent} if $\cF$ is in the essential image of the inclusion functor $\infShv_t(\cC,\cV)\to \infPsh(\cC,\cV)$.
\end{definition}

\begin{definition}
\label{arrow.8}
As a consequence of Propositions \ref{arrow.5} and \ref{arrow.6}, we obtain a sheaf  
\[
\logHH\in \infShv_{sNis}(\Fun(\Delta^1,\lSch),\infDAb).
\]
 This allows us to define $\logHH(X/S):=\logHH(X\to S)\in \infDAb$ for every morphism $X\to S$ in $\lSch$. We also set
\[
\logHH_{dNis}:=L_{dNis}\logHH\in \infShv_{dNis}(\Fun(\Delta^1,\lSch),\infDAb)
\]
and 
\[
\logHH_{dNis}(X/S):=\logHH_{dNis}(X\to S)\in \infDAb.
\]
\end{definition}

\subsection{$\Omega^d$ and $\bL$ as presheaves on an arrow \texorpdfstring{$\infty$-}{infinity }category}

\begin{df}
\label{arrow.9}
For $i\geq 1$, 
we let $\Omega^i\in \infPsh(\Fun(\Delta^1,\lSch),\infDAb)$ be the presheaf given by 
\[
\Omega^i(X\to S):= \Omega_{X/S}^i(X)
\]
concentrated in degree $0$. 
By convention, 
$\cO =  \Omega^0\in \infPsh(\Fun(\Delta^1,\lSch),\infDAb)$ is the presheaf given by $\cO(X\to S):=\Gamma(\ul{X}, \cO_{\ul{X}})$.
\end{df}

\begin{proposition}
\label{arrow.10}
For $(X\to S)\in \Fun(\Delta^1,\lSch)$ and all $d\geq 0$, 
there are equivalences
\[
L_{sNis}\Omega^d(X\to S)
\simeq
R\Gamma_{sNis}(X,\Omega_{X/S}^d)
\text{ and }
L_{dNis}\Omega^d(X\to S)
\simeq
R\Gamma_{dNis}(X,\Omega_{X/S}^d).
\]
\end{proposition}
\begin{proof}
We first consider the strict Nisnevich case.
Suppose $(X'\to S')\to (X\to S)$ is a strict Nisnevich cover in $\Fun(\Delta^1,\lSch)$.
Since $\Omega_{S'/S}^1=0$, we have an isomorphism $\Omega_{X'/S}^1\xrightarrow{\simeq} \Omega_{X'/S'}^1$, 
and hence 
\begin{equation}
\label{arrow.10.1}
\Omega_{X'/S}^d
\xrightarrow{\simeq}
\Omega_{X'/S'}^d
\end{equation}
for every integer $d\geq 1$ (it holds trivially true for $d=0$).

Let $\Hyp_{(X\to S)}$ (resp.\ $\Hyp_X$) be the set of strict Nisnevich hypercovers of $(X\to S)$ in $\Fun(\Delta^1,\lSch)$ 
(resp.\ $X$ in $\lSch$).
By appealing to \eqref{arrow.10.1} and \cite[Theorem 8.6]{DHI04} we have
\begin{align*}
L_{sNis}\Omega^d(X\to S)
&\simeq
\colimit_{(\mathscr{X}\to \mathscr{S})\in \mathrm{Hyp}_{(X\to S)}} \lim_{i\in \Delta}\Omega^d(\mathscr{X}_i\to \mathscr{S}_i)
\\
&\simeq
\colimit_{\mathscr{X}\in \mathrm{Hyp}_X} \lim_{i\in \Delta}\Omega^d(\mathscr{X}_i\to S)
\simeq
R\Gamma_{sNis}(X,\Omega_{X/S}^d).
\end{align*}
The claim for the dividing Nisnevich case can be proven similarly.
\end{proof}

For a morphism $f\colon X\to S$ of fs log schemes, 
let $\bL_{X/S}$ denote Gabber's cotangent complex for log schemes defined in \cite[\S 8.29]{Ols05}.
This complex lives in the $\infty$-category of chain complexes of quasi-coherent $\cO_X$-modules.
If $f$ is a morphism of fs log schemes given by $\Spec{g,g^\flat}$ for some map $(g,g^\flat)$ of discrete log rings, 
then this is equivalent to the log cotangent complex in Remark \ref{rmk:qcohcotangent}.

\begin{df}
\label{arrow.11}
Let $\bL\in \infPsh(\Fun(\Delta^1,\lSch),\infDAb)$ be the presheaf given by $(X\to S)\mapsto R\Gamma_{sNis}(X,\bL_{X/S})$.
We set $\bL^{dNis}:=L_{dNis}\bL$, for the $dNis$-localization 
\[
L_{dNis}\colon \Psh(\Fun(\Delta^1,\lSch),\infDAb)\to \Shv_{dNis}(\Fun(\Delta^1,\lSch),\infDAb).
\]
We can similarly define $\Bigwedge^d \bL$ and $\Bigwedge^d \bL^{dNis}$ for all integer $d$.
\end{df}

\begin{proposition}
\label{arrow.12}
The presheaf $\bL\in \infPsh(\Fun(\Delta^1,\lSch),\infDAb)$ satisfies integral log \'etale descent.
\end{proposition}
\begin{proof}
Suppose $X'\xrightarrow{f} S'\xrightarrow{g} S$ are morphisms in $\lSch$.
If $g$ is an integral log \'etale cover, then $\bL_{S'/S}\simeq 0$ by 
Proposition \ref{prop:cotangent_vs_omega}(2). 
The transitivity triangle
\[
f^*\bL_{S'/S}\to \bL_{X'/S}\to \bL_{X'/S'}\to f^*\bL_{S'/S}[1]
\]
implies  $\bL_{X'/S}\xrightarrow{\simeq} \bL_{X'/S'}$.
Argue as in the proof of Proposition \ref{arrow.10} to finish the proof.
\end{proof}

\begin{proposition}
\label{dZarsheaf.2}
For every integer $d\geq 0$, $L_{sNis}\Omega^d\in \infShv_{sNis}(\Fun(\Delta^1,\lSch),\infDAb)$ satisfies dividing Nisnevich descent.
\end{proposition}
\begin{warn}
The localization map $\bL\to \bL^{dNis}$ is not an equivalence since there is a log \'etale morphism $f\colon X\to S$ 
in $\lSch$ such that $\bL_{X/S}$ is not contractible, see Example \ref{exm:logetale}.
\end{warn}

\begin{proof}[Proof of Proposition \ref{dZarsheaf.2}]
Let $(\mathscr{X}\to \mathscr{S})\to (X\to S)$ is a dividing Nisnevich hypercover in $\Fun(\Delta^1,\lSch)$.
We need to show that the naturally induced map
\[
L_{sNis}\Omega^d(X\to S)
\to
\colimit_{(\sX\to \sS)\in \Hyp_{(X\to S)}^{dNis} }\lim_{i\in \Delta} L_{sNis}\Omega^d(\sX_i\to \sS_i)
\]
is an equivalence, where $\Hyp_{(X\to S)}^{dNis}$ is the set of dividing Nisnevich hypercovers of $(X\to S)$ in $\Fun(\Delta^1,\lSch)$.
There is an isomorphism $\Omega_{\mathscr{X}_j/\mathscr{S}_j}^1\cong \Omega_{\mathscr{X}_j/S}^1$ for every integer $j\geq 0$ 
since $\mathscr{S}_j\to S$ is log \'etale.
Hence it suffices to show that the naturally induced homomorphism
\begin{equation}
\label{dZarsheaf.2.1}
H_{sNis}^i(X,\Omega_{X/S}^d)
\to
\mathbf{H}_{sNis}^i(\mathscr{X},\Omega_{\mathscr{X}/S}^d)
\end{equation}
is an isomorphism for every integer $i\geq 0$.
It remains to check that 
$$
R\Gamma_{sNis}(-,\Omega_{-/S}^d)\in \infShv_{sNis}(\lSch/S,\infDAb)
$$
satisfies dividing Nisnevich descent;
this is done in Corollary \ref{cot.13}.
\end{proof}

\begin{df}
Let $\lSmAr$ be the full subcategory of $\Fun(\Delta^1,\lSch)$ spanned by the log smooth morphisms.
\end{df}

\begin{lemma}
\label{dZarsheaf.11}
The inclusion functor
\[
\lSmAr\to \Fun(\Delta^1,\lSch)
\]
is a cocontinuous morphism of sites for the dividing Nisnevich topologies.
\end{lemma}
\begin{proof}
Suppose $(Y'\to X')\to (Y\to X)$ is a dividing Nisnevich cover in $\Fun(\Delta^1,\lSch)$.
If $Y\to X$ is log smooth, then $Y'\to X'$ is log smooth too by \cite[Remark III.3.1.2]{Ogu18}.
\end{proof}

\begin{df}
Let $f\colon X\to S$ be a morphism of fs log schemes.
We say that $f$ is \emph{derived log smooth} (resp.\ \emph{derived log \'etale}) if $\ul{f}$ is locally of finite presentation 
and $\bL_{X/S}$ is a finitely generated locally free $\cO_X$-module (resp.\ $\bL_{X/S}\simeq 0$).
\end{df}

Observe that the conditions of being derived log smooth and derived log \'etale are Zariski local on $X$ and $S$.

Let $(g,g^\flat)\colon (R,P)\to (A,M)$ be a map of log rings.
If $(g,g^\flat)$ is derived log smooth, then it is log smooth, so $\bL_{(A,M)/(R,P)}$ is a finitely generated projective $A$-module by Proposition \ref{prop:cotangent_vs_omega2} and \cite[Proposition IV.3.2.1]{Ogu18}.
Thus $(g,g^\flat)$ is derived log smooth if and only if $\Spec{g,g^\flat}$ is derived log smooth.
We have a similar claim for the derived log \'etale case.
Hence the implications \eqref{eq:implications1} and \eqref{eq:implications2} are still valid for morphisms of fs log schemes.

As in Proposition \ref{prop:composition}, one can show that the classes of derived log smooth and derived log \'etale morphisms of fs log schemes are closed under compositions and pullbacks.

\begin{proposition}
\label{dZarsheaf.12}
Suppose $f\colon X\to S$ is a log smooth morphism in $\lSch$.
Then there is a canonical equivalence
\begin{equation}
\label{dZarsheaf.12.1}
\bL^{dNis}(X\to S)
\simeq
R\Gamma_{dNis}(X,\Omega_{X/S}^1).
\end{equation}
In particular, 
if $f$ is log \'etale, then $\bL^{dNis}(X\to S)\simeq 0$.
\end{proposition}
\begin{proof}
Owing to Proposition \ref{arrow.10} and Lemma \ref{dZarsheaf.11}, we need to show that the canonical map
\[
\bL^{dNis}
\to
L^{dNis} \Omega^1
\]
is an equivalence in $\infShv_{dNis}(\lSmAr,\infDAb)$.
Let $\dlSmAr$ be the full subcategory of $\lSmAr$ consisting of derived log smooth morphisms.
For every $(X\to S)\in \lSmAr$, 
we can choose a dividing Nisnevich cover $S'\to S$ such that the projection $X\times_S S'\to S'$ 
is integral log smooth by
\cite[Theorem 1.1]{Katointegral}.
Together with Proposition \ref{dZarsheaf.7}, we have an equivalence
\begin{equation}
\label{dZarsheaf.12.2}
\infShv_{dNis}(\lSmAr,\infDAb)
\simeq
\infShv_{dNis}(\dlSmAr,\infDAb).
\end{equation}
Hence we may assume $f$ is derived log smooth.
Now it suffices to show there is a canonical equivalence
\[
\bL_{X/S}\xrightarrow{\simeq} \Omega_{X/S}^1
\]
This question is Zariski local on $X$ and $S$, 
so we further reduce to the case when $S,X\in \lAff$.
In this case, 
Proposition \ref{prop:cotangent_vs_omega}(1) gives \eqref{dZarsheaf.12.1}.
Suppose $f$ is log \'etale.
For every log \'etale morphism $Y\to S$, we have $\Omega_{Y/S}^1=0$.
Hence the right-hand side of \eqref{dZarsheaf.12.1} vanishes.
\end{proof}

\begin{proposition}
\label{dZarsheaf.13}
Suppose $f\colon X\to S$ is a log smooth morphism in $\lSch$.
Then there is a canonical equivalence
\[
\bL^{dNis}(X\to S)
\simeq
R\Gamma_{sNis}(X,\Omega_{X/S}^1).
\]
\end{proposition}
\begin{proof}
Due to Propositions \ref{arrow.10} and \ref{dZarsheaf.12}, 
it only remains to check that $R\Gamma_{sNis}(-,\Omega_{X/S}^1)$ satisfies dividing Nisnevich descent.
This is done in Corollary \ref{cot.13}.
\end{proof}

\begin{example}
\label{arrow.13}
Let $f\colon X'\to Y'$ be the log \'etale map in Example \ref{exm:logetale} with non-vanishing log cotangent complex $\bL_{X'/Y'}$.
Let $P$ (resp.\ $P_1$, $P_2$) be the submonoid of $\Z^2$ generated by $(2,0)$, $(1,1)$, and $(0,2)$ 
(resp.\ $(2,0)$ and $(-1,1)$, $(0,2)$ and $(1,-1)$).
We set $P_{12}:=P_1+P_2$, $Y_1':=Y'\times_{\bA_P}\bA_{P_1}$, $Y_2':=Y'\times_{\bA_P}\bA_{P_2}$, and $Y_{12'}:=Y'\times_{\bA_P}\bA_{P_{12}}$.
We similarly define $X_1'$, $X_2'$, and $X_{12'}$.
Then the glueing of $Y_1'$ and $Y_2'$ along $Y_{12'}$ is a dividing cover of $Y'$.
Furthermore, the projections $X_1'\to Y_1'$, $X_2'\to Y_2'$, and $X_{12'}\to Y_{12'}$ are integral.
Hence
\[
\bL(X_1'\to Y_1'),\bL(X_2'\to Y_2'),\bL(X_{12}'\to Y_{12}')\simeq 0
\]
by Proposition \ref{prop:cotangent_vs_omega}(2).
This means that the induced sequence
\[
\bL(X'\to Y')\to \bL(X_1'\to Y_1')\oplus \bL(X_2'\to Y_2')\to \bL(X_{12}'\to Y_{12}')
\]
is not a cofiber sequence, so $\bL$ does not satisfy dividing Nisnevich descent.

Theorem \ref{prop:loghhbasechange} shows that the induced sequence
\[
\logHH(X'/Y') \to \logHH(X_1' / Y_1')\oplus \logHH(X_2' / Y_2') \to \logHH(X_{12}' / Y_{12}')
\]
is not a cofiber sequence, so $\logHH$ does not satisfy dividing Nisnevich descent either.
\end{example}

\begin{remark}[Comparison to Olsson's cotangent complex]

As explained in Section \ref{sec:differentials_and_cotangent}, 
we use Gabber's log cotangent complex.
It is a more or less direct extension to the context of pre-log (animated) rings of the classical construction due to 
Illusie and Quillen. 
There is another approach to the deformation theory of log schemes:
In \cite[Definition 3.2]{Ols05}, 
Olsson defines, 
for every morphism $f\colon X\to S$ of fine log schemes, 
a projective system $\mathbb{L}^O_{X/S} = (\bL^{\geq -n}_{X/S})_n$ 
of essentially constant ind-objects in $\mathcal{D}^b(\mathrm{Mod}_{\mathcal{O}_X})$. 
As such, 
a direct comparison with $\mathbb{L}_{X/S}$ considered in this paper (and in \cite[\S 8]{Ols05}) is not immediate. 
We will make a condition comparison, in the following sense. 

Assume the existence of an object $\bL^O\in \infPsh(\Fun(\Delta^1,\lSch),\infDAb)$ with the following two formal properties:
\begin{enumerate}
\item[(i)] The projective system
\[
(\cdots \to \tau_{\geq -n-1} \bL^O(X\to S) \to \tau_{\geq -n} \bL^O(X\to S) \to \cdots \to \tau_{\geq 0} \bL^O(X\to S))
\]
is canonically equivalent to $\bL_{X/S}^O$.
\item[(ii)] There exists a map $\bL\to \bL^O$ such that for all integers $n\geq 0$, the canonical map
\[
\tau_{\geq -n}\bL(X\to S)\to \tau_{\geq -n}\bL^O(X\to S)
\]
is naturally induced by the map constructed in \cite[\S 8.31]{Ols05}.
\end{enumerate}
In other words, $\mathbb{L}^O(-)$ is a presheaf on the arrow $\infty$-category $\Fun(\Delta^1,\lSch)$ whose Postnikov tower is 
precisely the projective system considered by Olsson. 

We claim that under the above assumptions, 
the map $\mathbb{L}\to \mathbb{L}^O$ becomes an equivalence after dividing Nisnevich sheafification. 
Indeed,  
let $\iAr$ be the full subcategory of $\Fun(\Delta^1,\lSch)$ consisting of the integral morphisms in $\lSch$.
If $(X\to S)$ is an object of $\Fun(\Delta^1,\lSch)$, 
then there exists a dividing cover $(X'\to S')$ of $(X\to S)$ with $(X'\to S')\in \iAr$ by
\cite[Theorem 1.1]{Katointegral}.
Proposition \ref{dZarsheaf.7} gives an equivalence
\[
\Shv_{dNis}(\iAr,\infDAb)
\xrightarrow{\simeq}
\Shv_{dNis}(\Fun(\Delta^1,\lSch),\infDAb).
\]
To conclude, 
we observe that the restriction of $\bL\to \bL^O$ to $\iAr$ is an equivalence owing to \cite[Corollary 8.34]{Ols05}.
\end{remark}

\section{Representability of \texorpdfstring{$\logHH$}{logHH} in logarithmic motivic homotopy theory}
\label{sec:motivic_repr}
Let us briefly recall the construction of the category of (effective) log motives over an fs noetherian log scheme of finite Krull dimension $S\in \lSch$. Let $\Lambda$ be a ring of coefficients and denote by $\mathcal{D}(\Lambda)$ the $\infty$-category of complexes of $\Lambda$-modules. We write 
$\boxx$ for the fs log scheme $(\mathbb{P}^1, \infty)$.
\begin{df}(1) A complex $C\in \infShv_{dNis}(\lSm/S, \mathcal{D}(\Lambda))$ is called $\boxx$-local if for all $X\in \lSm/S$, the map $R\Gamma_{dNis}( X, C) \to R\Gamma_{dNis}( X \times \boxx, C)$ is an equivalence in $\mathcal{D}(\Lambda).$

(2) $\logDAeff(S, \Lambda)$ is the full $\infty$-subcategory of $\infShv_{dNis}(\lSm/S, \mathcal{D}(\Lambda))$ consisting of $\boxx$-local objects. It is equivalent to the underlying $\infty$-category of the model category of chain complexes of presheaves of $\Lambda$-modules on $\lSm/S$ with the $(\boxx, dNis)$-local model structure. 
\end{df}
The interested reader can verify that this definition agrees with the one proposed in \cite{logSH} or \cite[Definition 2.9]{BMpurity}. For $\Lambda=\mathbb{Z}$, we simply write $\logDAeff(S)$. It is a monoidal, stable, presentable $\infty$-category. For $X\in \lSm/S$, we denote by $M_S(X)$ or simply by $M(X)$ the image in $\logDAeff(S)$ of the representable presheaf $\mathbb{Z}(X)$ via the localization functor.

If we denote by $\unit = M(S)$ the $\otimes$-unit of $\logDAeff(S)$, we set 
\[\unit(1) = \cofib(M(S)\xrightarrow{i_0} M(\mathbb{P}_S^1))[-2],\] where $i_0\colon S\to \mathbb{P}^1_S$ is the $0$-section, and for $n\geq 1$ we write $\unit(n)$ for the $n$-fold tensor product of $\unit(1)$ with itself. As for usual motives, this is called the $n$-th Tate twist.  
If $Y\to X$ is a morphism in $\lSm/S$, we set
\[
M(X/Y)
:=
\cofib(M(Y)\to M(X)).
\]
In \cite{logSH}, $\logDA(S)$ is defined to be the $\bP^1$-stabilization
\[
\logDA(S)
:=
\mathrm{Stab}_{\bP^1}(\logDAeff(S))
\]
in the sense of \cite[Theorem 2.14]{zbMATH06374152}.

For $X\in \lSm/S$, we keep writing $M(X)$ for the object of $\logDA(S)$ representable by $X$.

\subsection{$\boxx$-invariance of $\Omega^d$ and $\bL$}
This subsection aims to show that hypercohomology with coefficients in 
$\Omega^d$ and 
$\bL$ can be computed in terms of mapping spaces in $\logDAeff$. 
See also \cite[\S 9]{logDM} for related results.

\begin{proposition}
\label{rep.1}
Suppose $f\colon X\to S$ is a log smooth morphism in $\lSch$.
There is a canonical quasi-isomorphism
\begin{equation}
\label{rep.1.2}
\Omega_{X/S}^d
\simeq
R_{Zar}p_* \Omega_{X\times \boxx / S}^d
\end{equation}
for every integer $d\geq 0$, where $p\colon \ul{X}\times \bP^1\to \ul{X}$ denotes the projection.
\end{proposition}
\begin{proof}
By \cite[Proposition IV.1.2.15]{Ogu18}, there is a canonical isomorphism
\[
\Omega_{X\times \boxx / S}^1
\cong
p^*\Omega_{X / S}^1 \oplus \Omega_{X\times \boxx / X}^1.
\]
By taking exterior powers we obtain an isomorphism
\[
\Omega_{X\times \boxx / S}^d
\cong
\bigoplus_{j=0}^d p^*\Omega_{X / S}^j \otimes \Omega_{X\times \boxx / X}^{d-j}.
\]
Since $\Omega_{X / S}^j$ is locally free by \cite[Proposition IV.3.2.1]{Ogu18}, 
the projection formula gives a quasi-isomorphism
\[
R_{Zar}p_*\Omega_{X\times \boxx / S}^d
\simeq
\bigoplus_{j=0}^d \Omega_{X / S}^j \otimes R_{Zar}p_*\Omega_{X\times \boxx / X}^{d-j}.
\]

It remains to check that the naturally induced morphism
\begin{equation}
\label{rep.1.1}
\Omega_{X / X}^d
\to
R_{Zar}p_* \Omega_{X\times \boxx / X}^d
\end{equation}
is a quasi-isomorphism.
By \cite[Proposition IV.1.2.15]{Ogu18}, 
there is a canonical isomorphism
\[
\Omega_{X\times \boxx / X}^1
\cong
\Omega_{\ul{X}\times \boxx / \ul{X}}^1.
\]
Hence we reduce to the case when $X=\ul{X}$.
We may also assume that $X$ is the spectrum of a local ring $A$.
Elementary computations show that there are isomorphisms
\begin{gather*}
\Omega_{X\times \boxx / X}^1(X\times \bA^1)
\cong
A[t]dt,\quad
\Omega_{X\times \boxx / X}^1(X\times (\bP^1-0))
\cong
A[1/t]dt/t,
\\
\Omega_{X\times \boxx / X}^1(X\times \mathbb{G}_m)
\cong
A[t,1/t]dt.
\end{gather*}
By computing the \v{C}ech cohomology of the Zariski cover $\{X\times \bA^1,X\times (\bP^1-0)\to X\times \bP^1\}$, 
we see that \eqref{rep.1.1} is a quasi-isomorphism for $d=1$.
Moreover, $\Omega_{X\times \boxx / X}^d=0$ if $d>1$, so \eqref{rep.1.1} is a quasi-isomorphism for $d>1$.
If $d=0$, then \eqref{rep.1.1} becomes $\cO_X\simeq R_{Zar}p_*\cO_{X\times \bP^1}$.
\end{proof}

\begin{proposition}
\label{rep.2}
Suppose $f\colon X\to S$ is a log smooth morphism in $\lSch$ and $d,i\geq 0$.
There is a canonical isomorphism
\begin{equation}
\label{rep.2.1}
H_{dNis}^i(X,\Omega_{X/S}^d)
\cong
H_{dNis}^i(X\times \boxx,\Omega_{X\times \boxx / S}^d).
\end{equation}
\end{proposition}
\begin{proof}
Combine \eqref{cot.0.1}, Corollary \ref{cot.13}, and Proposition \ref{rep.1}.
\end{proof}

\begin{theorem}
\label{rep.3}
Suppose $f\colon X\to S$ is a log smooth morphism in $\lSch$.
Then
\[
L_{dNis}\Omega^d((-)\to S), \quad \bL^{dNis}((-)\to S)\in \infShv_{dNis}(\lSm/X,\infDAb)
\]
are in $\logDAeff(X)$.
\end{theorem}
\begin{proof}
Combine Propositions \ref{arrow.10}, \ref{dZarsheaf.12}, and \ref{rep.2}.
\end{proof}

\subsection{$\boxx$-invariance of \texorpdfstring{$\logHH$}{logHH}}
We are now ready to show $\boxx$-invariance and consequently representability of log Hochschild homology in the setting of log motives.

From the convergence of the log Andr\'e-Quillen spectral sequence in Theorem \ref{thm:logaqss}, 
we obtain an infinite sequence 
\begin{equation}
\label{dZarsheaf.14.1}
\cdots
\to
\Fil^{-1} \logHH
\to
\Fil^0 \logHH := \logHH
\end{equation}
in $\infShv_{sNis}(\Fun(\Delta^1,\lAff),\infDAb)$ such that
\begin{equation}
\label{dZarsheaf.14.4}
\lim_{q} \Fil^d \logHH \simeq 0
\end{equation}
and
\begin{equation}
\label{dZarsheaf.14.5}
\fiber(\Fil^{d-1}\logHH\to \Fil^d\logHH)
\simeq
\Bigwedge^{-d} \bL[d]
\end{equation}
for all integers $d$.
We may view \eqref{dZarsheaf.14.1} as an infinite sequence in $\infShv_{sNis}(\Fun(\Delta^1,\lSch),\infDAb)$ due to Proposition \ref{arrow.6}.

\begin{proposition}
\label{dZarsheaf.14}
Suppose $f\colon X\to S$ is a derived log smooth morphism in $\lSch$.
There is a naturally induced equivalence
\[
\logHH(X / S)
\xrightarrow{\simeq}
\logHH(X\times \boxx / S).
\]
\end{proposition}
\begin{proof}
By restriction, we regard \eqref{dZarsheaf.14.1} as an infinite sequence in $\infShv_{sNis}(\dlSmAr,\infDAb)$.
As observed in the proof of Corollary \ref{cor:loghkr},
we have an equivalence
\begin{equation}
\label{dZarsheaf.14.3}
\fiber(\Fil^{d-1}\logHH\to \Fil^{d}\logHH)\simeq L_{sNis}\Omega^{-d}[d]
\end{equation}
for all integer $d$.
Our goal is to show $\logHH(X\to S)\simeq \logHH(X\times \boxx\to S)$.
By \eqref{dZarsheaf.14.4} and \eqref{dZarsheaf.14.3} we are done if 
$L_{sNis}\Omega^d(X\to S)\simeq L_{sNis}\Omega^d(X\times \boxx\to S)$.
Propositions \ref{arrow.10} and \ref{rep.1} finish the proof.
\end{proof}

\begin{prop}
\label{dZarsheaf.17}
Suppose $f\colon X\to S$ is a derived log smooth morphism in $\lSch$ (for example, $f$ is log smooth and integral).
Then there is a canonical equivalence
\[
\logHH(X/S)
\simeq
\logHH_{dNis}(X/S).
\]
\end{prop}
\begin{proof}
Owing to \eqref{dZarsheaf.12.2}, it suffices to show
\begin{equation}
\label{dZarsheaf.17.1}
\logHH\in \infShv_{dNis}(\dlSmAr,\infDAb).
\end{equation}
We have the infinite sequences \eqref{dZarsheaf.14.1} in $\infShv_{sNis}(\dlSmAr,\infDAb)$.
As observed in the proof of Proposition \ref{dZarsheaf.13}, 
the graded pieces \eqref{dZarsheaf.14.3} are in $\infShv_{dNis}(\dlSmAr,\infDAb)$.
The functor
\[
\infShv_{sNis}(\dlSmAr,\infDAb)
\to
\infShv_{dNis}(\dlSmAr,\infDAb)
\]
is a right adjoint and commutes with limits.
By combining with \eqref{dZarsheaf.14.4}, we deduce \eqref{dZarsheaf.17.1}.
\end{proof}

\begin{proposition}
\label{dZarsheaf.15}
Suppose $f\colon X\to S$ is a log smooth morphism in $\lSch$.
There is a naturally induced equivalence
\[
\logHH_{dNis}^S(X)
\xrightarrow{\simeq}
\logHH_{dNis}^S(X\times \boxx).
\]
\end{proposition}
\begin{proof}
The question is dividing Nisnevich local on both $X$ and $S$.
Hence we may assume $f$ is integral log smooth by 
\cite[Theorem 1.1]{Katointegral}.
To conclude we combine Propositions \ref{dZarsheaf.14} and \ref{dZarsheaf.17}.
\end{proof}

This immediately implies the following fundamental result.

\begin{theorem}
\label{dZarsheaf.16}
For $S\in \lSch$ we have 
\[
\logHH_{dNis}(- / S)\in \logDAeff(S).
\]
\end{theorem}
Suppose $S$ has a valuative log structure.
Then every morphism $f\colon X\to S$ of fs log schemes is integral by \cite[Proposition I.4.6.3(5)]{Ogu18}.
Together with Propositions \ref{prop:cotangent_vs_omega}(1) and \ref{dZarsheaf.17}, we deduce
\begin{equation}
\logHH(- / S)\in \logDAeff(S).
\end{equation}

\subsection{Orientation of log Hochschild homology}
In \cite{logSH} we develop a notion of oriented cohomology theories for fs log schemes.
The main purpose of this subsection is to show that log Hochschild homology is oriented.
This observation is a crucial ingredient used for establishing the generalized residue sequence in Theorem \ref{Gysin.2}.

\begin{proposition}
\label{ori.2}
Suppose $S\in \lSch$ and $X,Y\in \lSch/S$.
There exists a canonical equivalence
\[
\logHH(X\times_S Y / S)
\simeq
\logHH(X / S)\otimes_{R\Gamma(S,\cO_{\underline{S}})} \logHH(Y / S).
\]
\end{proposition}
\begin{proof}
The question is strict Nisnevich local on $S$, $X$, and $Y$.
Hence we reduce to the case when $S$, $X$, and $Y$ are affine log schemes. Let $S = {\rm Spec}(R, P)$, $X = {\rm Spec}(A, M)$ and $Y = {\rm Spec}(B, N)$. By \Cref{prop:rognescomparison2} and the corresponding fact for ordinary Hochschild homology (see for example \cite[\S 6]{uniqBGT}), it suffices to prove that \[{\mathbb Z}[B^\rep_P(M \oplus_P N)] \simeq {\mathbb Z}[B^\rep_P(M)] \otimes_{{\mathbb Z}[P]} {\mathbb Z}[B^\rep_P(N)].\] This is readily seen using Lemma \ref{lem:discreterepletebar} twice.
\end{proof}

\begin{proposition}
\label{ori.1}
Suppose $X\to S$ is a morphism in $\lSch$.
If $\Bigwedge^d \bL_{X / S}\simeq 0$ for $d\gg 0$,
there exists a
convergent spectral sequence
\[
E_2^{pq}
:=
\mathbb{H}_{Zar}^p(X,\Bigwedge^{-q}\bL_{X / S})
\Rightarrow
\pi_{-p-q}\logHH(X / S).
\]
\end{proposition}
\begin{proof}
By restriction,
we regard $\logHH(X / S)$ as an object of $\infShv_{Zar}(X_{Zar},\infDAb)$.
From \eqref{dZarsheaf.14.1},
we obtain an infinite sequence 
\begin{equation}
\cdots
\to
\Fil^{-1} \logHH(X / S)
\to
\Fil^0 \logHH(X / S) := \logHH(X / S)
\end{equation}
in $\infShv_{Zar}(X_{Zar},\infDAb)$ such that
\begin{equation}
\lim_{d} \Fil^d \logHH(X / S) \simeq 0
\end{equation}
and
\begin{equation}
\fiber(\Fil^{d-1}\logHH(X / S)\to \Fil^d\logHH(X / S))
\simeq
\Bigwedge^{-d}\bL_{X/S}[d]
\end{equation}
for all integers $d$.
The vanishing condition on $\Bigwedge^d \bL_{X/S}$ implies that $\Fil^d \logHH(X / S) \simeq 0$ for $d\ll 0$.
Using \cite[Proposition 1.2.2.14]{HA} we obtain a convergent spectral sequence
\[
E'^1_{pq}
:=
H_{Zar}^{-2p-q}(X,\Bigwedge^{-p}\bL_{X/S})
\Rightarrow
\pi_{p+q}\logHH(X / S).
\]
To obtain the desired spectral sequence, we set $E_r^{pq}:=E'^{r-1}_{q,-p-2q}$ for all $r\geq 2$.
\end{proof}

\begin{proposition}
\label{ori.3}
Suppose $S$ is an affine log scheme $\Spec{A,M}$.
Then $\logHH(\bP_S^n / S)$ is concentrated in degree $0$, and there is a canonical isomorphism of $A$-modules
\[
\pi_0(\logHH(\bP_S^n / S))
\cong
A^{\oplus n+1}.
\]
\end{proposition}
\begin{proof}
We have the direct computation
\[
H_{Zar}^p(\bP_S^n,\Omega_{\bP_S^n}^{q})
\cong
\left\{
\begin{array}{ll}
A & \text{if }0\leq p=q\leq n,
\\
0 & \text{otherwise.}
\end{array}
\right.
\]
The convergent spectral sequence
\[
E_2^{pq}:=H_{Zar}^p(\bP_S^n,\Omega_{\bP_S^n}^{-q})
\Rightarrow
\pi_{-p-q}\logHH(\bP_S^n / S)
\]
obtained from Proposition \ref{ori.1} finishes the proof.
\end{proof}

For topological Hochschild homology of schemes, 
the following result is a special case of the projective bundle formula \cite[Theorem 1.5]{BM12}.

\begin{proposition}
\label{ori.4}
Suppose $S\in \lSch$ and $X\in \lSch/S$.
There is a canonical equivalence in $\infDAb$
\[
\logHH(X\times \bP^n / S)
\simeq
\logHH(X / S)^{\oplus n+1}.
\]
\end{proposition}
\begin{proof}
Let $\cC$ be the full subcategory of $\Fun(\Delta^1,\lSch)$ consisting of morphisms whose targets are in $\lAff$.
As in Proposition \ref{arrow.6}, there is an equivalence
\[
\Shv_{sNis}(\cC,\infDAb)
\simeq
\Shv_{sNis}(\Fun(\Delta^1,\lSch),\infDAb).
\]
Hence it suffices to give a natural equivalence
\[
((X\to S)\in \cC \mapsto \logHH(X\times \bP^n / S))
\xrightarrow{\simeq}
((X\to S)\in \cC \mapsto \logHH(X / S)^{\oplus n+1}).
\]
Due to Proposition \ref{ori.2}, we reduce to constructing a natural equivalence
\[
(S\in \lAff \mapsto \logHH(S\times \bP^n / S))
\xrightarrow{\simeq}
(S\in \lAff \mapsto \logHH(S / S)^{\oplus n+1}).
\]
This follows from Proposition \ref{ori.3}.
\end{proof}

Hence for $S\in \Sch$, we obtain an equivalence in $\logDA(S)$
\begin{equation}
\label{Ori.4.1}
\logHH(- / S)
\simeq
\Omega_{\bP^1}\logHH(- / S),
\end{equation}
where $\Omega_{\bP^1}:=\ul{\Hom}(M(\bP^1/\pt),-)$, and $\ul{\Hom}$ denotes the internal Hom functor.

\begin{definition}
The logarithmic Hochschild homology of $S\in \Sch$ is given by 
\[
\blogHH(- / S)
:=
(\logHH(- / S),\logHH(- / S),\ldots)\in \logDA(S),
\]
where the bonding maps are given by \eqref{Ori.4.1}.

For every map $(R,P)\to (A,P)$ of log rings, $\logHH((A,M) / (R, P))$ is an animated commutative ring.
Hence $\logHH$ is an $\bE_\infty$-ring in $\Psh(\Fun(\Delta^1,\lAff),\infDAb)$. The localization functor is symmetric monoidal, so $\logHH$ is an $\bE_\infty$-ring in
\[
\Shv_{sNis}(\Fun(\Delta^1,\lAff),\infDAb)
\simeq
\Shv_{sNis}(\Fun(\Delta^1,\lSch),\infDAb).
\]
The restriction and localization functors
\[
\Shv_{sNis}(\Fun(\Delta^1,\lSch),\infDAb)
\to
\Shv_{sNis}(\lSm/S,\infDAb)
\to
\logDAeff(S)
\]
are symmetric monoidal too, so $\logHH$ is an $\bE_\infty$-ring in $\logDAeff(S)$.

By 
\cite[Construction 6.2.3]{logSH},
one can show that $\blogHH(- / S)$ is a homotopy commutative monoid in $\logDA(S)$, 
i.e., there are multiplication and unit maps  
\[
\blogHH(- / S)\wedge \blogHH(- / S)
\to
\blogHH(- / S)
\text{ and }
\unit \to \blogHH(- / S)
\]
satisfying the associative, commutative, and unital axioms in the homotopy category of $\logDA(S)$. 
For all $X\in \lSm/S$, 
there is a canonical equivalence
\begin{equation}
\label{ori.6.1}
\map_{\logDA(S)}(M(X),\blogHH(- / S))
\simeq
\logHH(X / S), 
\end{equation}
where $\map_{\logDA(S)}$ denotes the mapping spectrum in $\logDA(S)$, 
and we regard the right hand side as a spectrum using the Dold-Kan correspondence in \cite[\S 1.2.4]{HA}.
Moreover, 
for all $X\in \lSm/S$, 
there are canonical equivalences
\begin{equation}
\label{ori.6.2}
\begin{split}
\blogHH(- / S)
= 
(\logHH(- / S),\logHH(- / S),\ldots)
\simeq &
(\Omega_{\bP^1}\logHH(- / S),\Omega_{\bP^1}\logHH(- / S),\ldots)
\\
\simeq &
\blogHH(- / S)(-1)[-2].
\end{split}
\end{equation}
\end{definition}

We note that \eqref{ori.6.2} is analogous to the Bott periodicity witnessed in the algebraic and complex $K$-theory spectra 
\cite{zbMATH01194164}.

The next notion is a direct adaptation of
\cite[Definition 7.1.1]{logSH}
to $\logDA(S)$.

\begin{definition}
\label{ori.6}
Suppose $S\in \Sch$.
A homotopy commutative monoid $\bE$ in $\logDA(S)$ is called \emph{oriented} if there is a class 
\[
c_\infty \in \Hom_{\logDA(S)}(M(\bP_S^\infty/\pt),\bE(1)[2])
\]
whose restriction to $M(\bP_S^1/\pt)$ is given by
\[
M(\bP_S^1/\pt)\otimes \unit 
\colon
M(\bP_S^1/\pt)
\to
M(\bP_S^1/\pt)\otimes \bE
\simeq
\bE(1)[2],
\]
where $\unit$ denotes the unit of $\bE$, and $M(\bP_S^\infty/\pt):=\colimit_{n\to \infty}M(\bP_S^n/\pt)$.
\end{definition}

\begin{theorem}
\label{ori.5}
The homotopy commutative monoid $\blogHH(- / S)\in \logDA(S)$ is oriented for all $S\in \Sch$.
\end{theorem}
\begin{proof}
By Proposition \ref{ori.4}, there are isomorphisms
\[
\Hom_{\logDA(S)}(M(\bP_S^\infty/\pt),\blogHH(- / S)(1)[2])
\cong
\lim_{n\to \infty} {\pi_0}\logHH(\bP_S^n / S)
\cong
\prod_{n=1}^\infty \Gamma(S,\cO_S).
\]
The element $(1,0,\ldots)$ of the above infinite product gives an orientation class.
\end{proof}

\section{Generalized residue  sequences for \texorpdfstring{$\logHH$}{logHH}}\label{sec:gysinseq}
Let $S=\Spec{A}$ be an affine scheme. In Section \ref{sec:motivic_repr} we have seen how 
\[\logHH(- / S)\in \mathrm{Alg}_{\mathbb{E}_\infty}(\Shv_{dNis}(\mathrm{SmAff}_S, \cD(\mathrm{Ab})))\] gives rise to a homotopy commutative monoid object $\blogHH(- / S)$ in $\logDA(S)$. 
Combined with the formal properties of the motivic category $\logDA(S)$ we will deduce sequences relating logarithmic and ordinary
Hochschild homology. The fundamental ingredient is the Gysin, or residue, sequence in $\logDA(S)$. 
Let us recall the statement from \cite{logDM}.

\begin{definition}[{\cite[Definition 7.4.3]{logDM}}]
Suppose $S\in \Sch$, $X\in \lSm/S$, and $\cE\to X$ is a vector bundle (i.e., $\xi\colon\underline{\cE}\to \underline{X}$ is a vector bundle on the underlying scheme $\ul{X}$, and the log structure on $\cE$  is the pullback of the log structure on $X$ via $\xi$).
The \emph{Thom motive of $\cE$ over $X$} is defined to be
\[
MTh(\cE):=M(\cE/(\Bl_Z\cE,E)) \in \logDA(S),
\]
where $Z$ is the zero section of $\cE$, and $E$ is the exceptional divisor. 
\end{definition}

\begin{df}
\label{Gysin.5}
Suppose $S\in \Sch$, $X\in \Sm/S$, $Z_1+\cdots+Z_r$ is a strict normal crossing divisor on $X$, and $Z$ is a smooth closed subscheme of $X$ that is strict normal crossing with $Z_1+\cdots+Z_r$ in the sense of \cite[Definition 7.2.1]{logDM}.
Let $p\colon N_Z X\to Z$ be the projection, where $N_ZX$ denotes the normal bundle of $Z$ in $X$.
We set
\[
N_Z(X,Z_1+\cdots+Z_r)
:=
N_Z X\times_Z(Z,Z\cap Z_1+\cdots+Z\cap Z_r).
\]
Observe that $N_Z(X,Z_1+\cdots+Z_r)$ is a vector bundle over $(Z,Z\cap Z_1+\cdots + Z\cap Z_r)$.
\end{df}

\begin{theorem}
\label{Gysin.6}
There exists a canonical cofiber sequence
\[
M(\Bl_Z X,\widetilde{Z}_1+\cdots+\widetilde{Z}_r+E)
\xrightarrow{M(u)}
M(X,Z_1+\cdots+Z_r)
\to 
MTh(N_Z (X,Z_1+\cdots+Z_r))
\]
in $\logDA(S)$, where $\widetilde{Z}_i$ is the strict transform of $Z_i$ in $\Bl_Z X$ for $i=1,\ldots,r$, $E$ is the exceptional divisor on $\Bl_Z X$, and $u\colon (\Bl_Z X,\widetilde{Z}_1+\cdots+\widetilde{Z}_r+E)\to (X,Z_1+\cdots+Z_r)$ is the obvious morphism.
\end{theorem}
\begin{proof}
We refer to \cite[Theorem 7.5.4]{logDM}.
\end{proof}

\begin{theorem}
\label{Gysin.4}
Suppose $S\in \Sch$ and $\bE$ is an oriented homotopy commutative monoid in $\logDA(S)$.
For $X\in \lSm/S$, a rank $d$ vector bundle $\cE\to X$, and integers $p$ and $q$, there is a canonical equivalence
\begin{equation}
\label{Gysin.4.1}
\map_{\logDA(S)}(M(X),\bE(q)[p])
\simeq
\map_{\logDA(S)}(MTh(\cE),\bE(q+d)[p+2d]).
\end{equation}
\end{theorem}
\begin{proof}
When $S$ is the spectrum of a perfect field $k$ with resolution of singularities,
and $\logDA(S)$ is replaced by the version with transfers $\mathrm{log}\mathcal{DM}(k)$,
this is a consequence of \cite[Theorem 8.3.7]{logDM}.
Indeed,
there is an isomorphism $MTh(\cE)\simeq M(X)(d)[2d]$,
and an orientation on $\bE$ is not needed.
The equivalence \eqref{Gysin.4.1} holds in greater generality:
our proof of
\cite[Proposition 7.3.9]{logSH}
adapts to $\logDA(S)$ under the assumption that $\bE$ is oriented. 
\end{proof}

\begin{theorem}[Generalized residue sequence]
\label{Gysin.2}
With the notations of Definition \ref{Gysin.5} and Theorem \ref{Gysin.6}, there exists a canonical cofiber sequence
\begin{align*}
\logHH((Z,Z\cap Z_1+\cdots+Z\cap Z_r) / S)
\to &
\logHH((X,Z_1+\cdots+Z_r) / S)
\\
\xrightarrow{u^*} &
\logHH((\Bl_Z X,\widetilde{Z}_1+\cdots+\widetilde{Z}_r+E) / S).
\end{align*}
The above cofiber sequence is called the \emph{Gysin (or residue) sequence for $\logHH$.}

As a particular case, we obtain a canonical cofiber sequence
\begin{equation}\label{particularcase}
\HH(Z / S)
\to
\HH(X / S)
\to
\logHH((\Bl_Z X,E) / S).
\end{equation}
\end{theorem}
\begin{proof}
Apply Theorem \ref{Gysin.4} to the oriented homotopy commutative monoid $\blogHH(-/S)$ and use \eqref{ori.6.1}, 
\eqref{ori.6.2}, and Theorem \ref{Gysin.6}.
\end{proof}
The interested reader can note that \eqref{particularcase} is similar to the generalized Gysin sequence available 
for the cohomology of reciprocity sheaves in \cite[Theorem 7.16]{BRS}.

\subsection{Comparison of cofiber sequences} Let $A$ be a discrete commutative ring and let $a \in A$ be a non-zero divisor. By \cite[Theorem 5.5, Example 5.7]{RSS15}, there is a cofiber sequence \[\THH(A) \to \THH(A, \langle a \rangle) \to \THH(A/(a))[1],\] where $\langle a \rangle \to (A, \cdot)$ is the pre-log structure generated by $a$. Here $\THH(A, \langle a \rangle)$ denotes log topological Hochschild homology in the sense of \cite{RSS15}. By cobase change along the augmentation $\THH({\mathbb Z}) \to {\mathbb Z}$, 
\Cref{prop:rognescomparison2} and \cite[Lemma 5.4]{Lun21} yield the cofiber sequence \begin{equation}\label{rsscofiberseq}{\rm HH}(A) \to {\rm logHH}(A, \langle a \rangle) \to {\rm HH}(A/(a))[1]. \end{equation}

We point out that the cofiber sequence \eqref{rsscofiberseq} is a special case of Theorem \ref{Gysin.2}. Indeed, set $S = {\rm Spec}({\mathbb Z})$ (with trivial log structure), $X = {\rm Spec}({\mathbb Z}[x])$ and $Z = {\rm Spec}({\mathbb Z})$ in \eqref{particularcase}, so that the resulting cofiber sequence reads \[{\rm HH}({\mathbb Z}[x]) \to {\rm logHH}({\mathbb Z}[x], \langle x \rangle \times {\rm GL}_1({\mathbb Z}[x])) \to {\rm HH}({\mathbb Z})[1].\] Now we cobase-change along ${\rm HH}({\mathbb Z}[x]) \to {\rm HH}(A)$ and apply logification invariance of ${\rm logHH}$ to obtain a cofiber sequence of the form \eqref{rsscofiberseq}. Logification invariance of ${\rm logHH}$ can be seen using cobase change along $\THH({\mathbb Z}) \to {\mathbb Z}$ and a similar equivalence in log $\THH$, see \cite[Theorem 4.24]{RSS15}, or by using our Quillen spectral sequence and the analogous statement for the 
log cotangent complex in \cite[Theorem 8.16]{Ols05}.

When $a\in A$ is a non-zero divisor, the subscheme $Z=\Spec{A/(a)}$ is an effective Cartier divisor in $X$, and clearly $\Bl_ZX = X$, i.e., the blow-up is irrelevant. It is then clear by construction that the sequences \eqref{particularcase} and \eqref{rsscofiberseq} agree. 

\begin{remark}\label{rem:regularsequence}
Theorem \ref{Gysin.2} extends the results in \cite{RSS15} in two directions. 
First, 
it works for log schemes (and not just for log rings), 
and the log structure is allowed to be non-trivial in all the terms of the sequence. 
Second, 
it allows for a higher-dimensional generalization of the residue sequence. 
Allowing for non-affine log schemes is essential to that end.

To explain it better, let us consider the following situation. Let $R$ be a noetherian commutative ring of finite Krull dimension, and let $A$ be a smooth $R$-algebra. Let  $f_1, \ldots, f_r$ be a regular sequence in $A$, such that the quotient $A/(f_1, \ldots, f_r)$ is  a smooth complete intersection $R$-algebra. Then, since the normal bundle of $Z:=\Spec{A/(f_1, \ldots, f_r)}$ in $X:=\Spec{A}$ is trivial, 
\eqref{particularcase} reads as
\[ \HH(A / R) \to \logHH((\Bl_{(f_1, \ldots, f_r)}\Spec{A}, E) / R) \to \HH((A/(f_1, \ldots, f_r)) / R)[1], \]
where $E\cong \mathbb{P}(N_Z X) \cong \mathbb{P}_Z^{r-1}$. 
For $r=1$, 
using the HKR Theorem, 
we obtain the classical residue sequence (see e.g., \cite[2.3.a]{EV})
\[0\to \Omega^n_{A/R} \to \Omega^n_{A/R}(\log f) \xrightarrow{Res} \Omega^{n-1}_{(A/f) / R} \to 0.\]
Note that 
the map $\pi_n\HH((A/f_1) / R) \to \pi_n\HH(A / R)$ is trivial for all $n\geq 0$ in this case since 
the differentials inject into the log differentials.
\end{remark}

We end with a discussion about residue sequences in log $\THH$, 
and how this could potentially relate to the results of this paper:

\begin{remark}\label{rem:bp} Quillen's localization theorem \cite[\S 5]{Quillen} produces a cofiber sequence 
\[K(\bZ_p) \to K(\bQ_p) \to K(\bF_p)[1]\] in algebraic $K$-theory, while Blumberg--Mandell 
\cite[Localization Theorem]{BM08} establish a cofiber sequence \[K({\rm ku}_p) \to K({\rm KU}_p) \to K(\bZ_p)[1].\] 
Here ${\rm ku}_p$ and ${\rm KU}_p$ are the connective and periodic complex ($p$-complete) $K$-theory spectra. 
These examples answer the following question in the affirmative for $n = 0, 1$: Is there a cofiber sequence \begin{equation}\label{ktheorybpcofiber}K(BP \langle n \rangle) \to K(BP \langle n \rangle[v_n^{-1}]) \to K(BP \langle n - 1 \rangle)[1] \end{equation} for the truncated Brown--Peterson spectrum $BP \langle n \rangle$? For $n\geq 2$, 
such cofiber sequences were proven to \emph{not} exist in \cite{ABG18}.

The search for cofiber sequences in log $\THH$ is often motivated by the problem of finding close approximations 
of cofiber sequences in algebraic $K$-theory that do not exist in ordinary Hochschild homology: 
Such a perspective is appealing for the examples of \cite{RSS15}.

We sketch a possible construction of
a cofiber sequence \begin{equation}\label{logthhbpncofiber}{\rm THH}(BP \langle n \rangle) \to {\rm logTHH}(BP \langle n \rangle, \langle v_n \rangle_*) \to {\rm THH}(BP \langle n - 1 \rangle)[1].\end{equation} 
This is a cofiber sequence in log ${\rm THH}$ involving $BP \langle n \rangle$ for which there is no corresponding $K$-theory sequence \eqref{ktheorybpcofiber} when $n\geq 2$. The truncated Brown--Peterson spectra $BP \langle n \rangle$ do in general not admit the structure of ${\mathbb E}_{\infty}$-rings \cite[Theorem 1.1.2]{Law18}, and hence they do not fall under the framework of logarithmic ${\rm THH}$ as developed by Rognes--Sagave--Schlichtkrull \cite{RSS15}, \cite{RSS18}. We shall combine groundbreaking work of Hahn--Wilson \cite{HW21} with a construction of Sagave--Schlichtkrull \cite{SS19} to give an \emph{ad hoc} definition of ${\rm logTHH}(BP \langle n \rangle, \langle v_n \rangle_*)$ participating in cofiber sequences \eqref{logthhbpncofiber}.

We start with the cofiber 
sequence \[{\rm THH}({\rm MUP}_{\ge 0}) \to {\rm logTHH}({\rm MUP}_{\ge 0}, V^{\cal U}_{\ge 0}) \to {\rm THH}({\rm MU})[1]\] of Sagave--Schlichtkrull \cite[Example 8.6]{SS19}. Here ${\rm MUP}_{\ge 0} = \oplus_{i \ge 0} {\rm MU}[2i]$ is the positively weighted part of the periodic complex cobordism spectrum ${\rm MUP}$, equipped with a log structure $V^{\cal U}_{\ge 0}$ localizing to ${\rm MUP}$. Subobjects of ${\rm MUP}_{\ge 0}$ 
whose weights are multiples of $p^n - 1$ are denoted by ${\rm MU}[y]$ in Hahn--Wilson \cite[\S 2.1]{HW21}. We think of ${\rm MU}[y]$ as a polynomial ring over ${\rm MU}$, and Sagave--Schlichtkrull's argument still applies to produce a cofiber sequence \begin{equation}\label{muycofiber}{\rm THH}({\rm MU}[y]) \to {\rm logTHH}({\rm MU}[y], \langle y \rangle_*) \to {\rm THH}({\rm MU})[1].\end{equation} As part of their construction of $\bE_3$-${\rm MU}$-algebra structures on $BP\langle n \rangle$, Hahn--Wilson \cite{HW21} construct $\bE_3$-${\rm MU}[y]$-algebra structures on $BP \langle n \rangle$, where $y$ acts by $v_n$. Applying a cobase change of \eqref{muycofiber} along the resulting map ${\rm THH}({\rm MU}[y]) \to {\rm THH}(BP\langle n \rangle)$ 
gives \eqref{logthhbpncofiber}, where ${\rm logTHH}(BP \langle n \rangle, \langle v_n \rangle_*)$ is defined as the
cobase change of ${\rm logTHH}({\rm MU}[y], \langle y \rangle_*)$ along ${\rm THH}({\rm MU}[y]) \to {\rm THH}(BP\langle n \rangle)$. A version of this sequence will appear in the forthcoming work of Ausoni--Bayındır--Moulinos.
This further accentuates the difference between the behavior of cofiber sequences in log ${\rm THH}$ and localization sequences in algebraic $K$-theory. 

One may attempt to define log structures on $BP\langle n \rangle$ generated by all the $v_i$'s in the hope of obtaining 
a sequence in ${\rm logTHH}$ analogous to the $K$-theory sequence of Barwick--Lawson \cite[Corollary 5.2]{BL14}. 
By \cite[Remark 2.1.2]{HW21}, the approach sketched here does not apply to produce such cofiber sequences in ${\rm logTHH}$. However, Theorem \ref{Gysin.2} suggests one way to approach this problem: Replace $BP \langle n \rangle$ with an appropriate blow-up in the context of spectral algebraic geometry and prove that it participates in 
a ${\rm logTHH}$ residue sequence. In upcoming work, we hope to address this question, as well as how the resulting sequence relates to variants of algebraic $K$-theory.
\end{remark}

\noindent\textbf{Acknowledgements.} The authors wish to thank the anonymous referee for a detailed and informative report, that significantly helped to streamline the exposition. 

The authors also thank Bj\o rn Dundas, Elden Elmanto, Lars Hesselholt, 
Teruhisa Koshikawa, Mauro Porta, John Rognes, and Steffen Sagave for conversations and comments on the 
subject of this paper.

\bibliographystyle{siam}
\bibliography{logHKRPostRevision}

\end{document}